\documentclass{lmcs}
\pdfoutput=1

\usepackage{lastpage}
\lmcsdoi{16}{1}{2}
\lmcsheading{}{\pageref{LastPage}}{}{}%
{Aug.~01,~2017}{Jan.~08,~2020}{}

\title{Modalities in homotopy type theory}
\author[Egbert Rijke]{Egbert Rijke\rsuper{a}}
\address{\lsuper{a}University of Ljubljana, Ljubljana, Slovenia}
\email{egbert.rijke@fmf.uni-lj.si}

\author[Michael Shulman]{Michael Shulman\rsuper{b}}
\address{\lsuper{b}University of San Diego, San Diego, CA}
\email{shulman@sandiego.edu}

\thanks{This material is based
    on research sponsored by The United States Air Force Research
    Laboratory under agreement number FA9550--15--1--0053.  The
    U.S. Government is authorized to reproduce and distribute re\-prints
    for Governmental purposes notwithstanding any copyright notation
    thereon.  The views and con\-clusions contained herein are those of
    the author and should not be interpreted as necessarily
    representing the official policies or endorsements, either
    expressed or implied, of the United States Air Force Research
    Laboratory, the U.S. Government, or Carnegie Mellon University.}
\author[Bas Spitters]{Bas Spitters\rsuper{c}}
\address{\lsuper{c}Aarhus University, Aarhus, Denmark}
\email{spitters@cs.au.dk}
\thanks{Partially funded by the Guarded homotopy type theory
    project (12386) of the Villum Foundation.}
\date{\today}

\usepackage[utf8]{inputenc}
\usepackage[english]{babel}

\usepackage{courier}
\usepackage{microtype}

\usepackage{fancyhdr} 
\usepackage{enumitem,mathtools,xspace,xcolor}
\usepackage{comment}
\usepackage{ifthen}
\usepackage{pifont}

\usepackage{graphicx}
\usepackage{caption}
\usepackage{xcolor}

\DeclareCaptionFormat{myformat}{{\bf #1}#2#3{\color{black!20}\hrulefill}}
\usepackage[status=final]{fixme}
\fxsetup{targetlayout=color}

\usepackage{tikz-cd}
\usepackage{tikz}
\usetikzlibrary{decorations.pathmorphing}
\usepackage[inference]{semantic}
\usepackage{booktabs}
\usepackage{mathpartir}

\definecolor{darkgreen}{rgb}{0,0.45,0}

\usepackage{amssymb,amsmath,amsthm,stmaryrd,mathrsfs,wasysym}
\usepackage{aliascnt}
\usepackage[capitalize]{cleveref}

\usepackage{braket} 

\usepackage{etoolbox}           

\makeatletter
\newcommand{\sectionNotes}{\phantomsection\section*{Notes}\addcontentsline{toc}{section}{Notes}\markright{\textsc{\@chapapp{} \thechapter{} Notes}}}
\newcommand{\sectionExercises}[1]{\phantomsection\section*{Exercises}\addcontentsline{toc}{section}{Exercises}\markright{\textsc{\@chapapp{} \thechapter{} Exercises}}}
\makeatother

\newcommand{\jdeq}{\equiv}      
\newcommand{\defeq}{\vcentcolon\equiv}  

\newcommand{\define}[1]{\textbf{#1}}

\makeatletter
\def\prd#1{\@ifnextchar\bgroup{\prd@parens{#1}}{\@ifnextchar\sm{\prd@parens{#1}\@eatsm}{\prd@noparens{#1}}}}
\def\prd@parens#1{\@ifnextchar\bgroup%
  {\mathchoice{\@dprd{#1}}{\@tprd{#1}}{\@tprd{#1}}{\@tprd{#1}}\prd@parens}%
  {\@ifnextchar\sm%
    {\mathchoice{\@dprd{#1}}{\@tprd{#1}}{\@tprd{#1}}{\@tprd{#1}}\@eatsm}%
    {\mathchoice{\@dprd{#1}}{\@tprd{#1}}{\@tprd{#1}}{\@tprd{#1}}}}}
\def\@eatsm\sm{\sm@parens}
\def\prd@noparens#1{\mathchoice{\@dprd@noparens{#1}}{\@tprd{#1}}{\@tprd{#1}}{\@tprd{#1}}}
\def\lprd#1{\@ifnextchar\bgroup{\@lprd{#1}\lprd}{\@@lprd{#1}}}
\def\@lprd#1{\mathchoice{{\textstyle\prod}}{\prod}{\prod}{\prod}({\textstyle #1})\;} 
\def\@@lprd#1{\mathchoice{{\textstyle\prod}}{\prod}{\prod}{\prod}({\textstyle #1}),\ } 
\def\tprd#1{\@tprd{#1}\@ifnextchar\bgroup{\tprd}{}}
\def\@tprd#1{\mathchoice{{\textstyle\prod_{(#1)}}}{\prod_{(#1)}}{\prod_{(#1)}}{\prod_{(#1)}}}
\def\dprd#1{\@dprd{#1}\@ifnextchar\bgroup{\dprd}{}}
\def\@dprd#1{\prod_{(#1)}\,}
\def\@dprd@noparens#1{\prod_{#1}\,}

\def\lam#1{{\lambda}\@lamarg#1:\@endlamarg\@ifnextchar\bgroup{.\,\lam}{.\,}}
\def\@lamarg#1:#2\@endlamarg{\if\relax\detokenize{#2}\relax #1\else\@lamvar{\@lameatcolon#2},#1\@endlamvar\fi}
\def\@lamvar#1,#2\@endlamvar{(#2\,{:}\,#1)}
\def\@lameatcolon#1:{#1}
\def\lamu#1{{\lambda}\@lamuarg#1:\@endlamuarg\@ifnextchar\bgroup{.\,\lamu}{.\,}}
\def\@lamuarg#1:#2\@endlamuarg{#1}

\def\fall#1{\forall (#1)\@ifnextchar\bgroup{.\,\fall}{.\,}} 

\def\exis#1{\exists (#1)\@ifnextchar\bgroup{.\,\exis}{.\,}} 

\def\sm#1{\@ifnextchar\bgroup{\sm@parens{#1}}{\@ifnextchar\prd{\sm@parens{#1}\@eatprd}{\sm@noparens{#1}}}}
\def\sm@parens#1{\@ifnextchar\bgroup%
  {\mathchoice{\@dsm{#1}}{\@tsm{#1}}{\@tsm{#1}}{\@tsm{#1}}\sm@parens}%
  {\@ifnextchar\prd%
    {\mathchoice{\@dsm{#1}}{\@tsm{#1}}{\@tsm{#1}}{\@tsm{#1}}\@eatprd}%
    {\mathchoice{\@dsm{#1}}{\@tsm{#1}}{\@tsm{#1}}{\@tsm{#1}}}}}
\def\@eatprd\prd{\prd@parens}
\def\sm@noparens#1{\mathchoice{\@dsm@noparens{#1}}{\@tsm{#1}}{\@tsm{#1}}{\@tsm{#1}}}
\def\lsm#1{\@ifnextchar\bgroup{\@lsm{#1}\lsm}{\@@lsm{#1}}}
\def\@lsm#1{\mathchoice{{\textstyle\sum}}{\sum}{\sum}{\sum}({\textstyle #1})\;} 
\def\@@lsm#1{\mathchoice{{\textstyle\sum}}{\sum}{\sum}{\sum}({\textstyle #1}),\ } 
\def\tsm#1{\@tsm{#1}\@ifnextchar\bgroup{\tsm}{}}
\def\@tsm#1{\mathchoice{{\textstyle\sum_{(#1)}}}{\sum_{(#1)}}{\sum_{(#1)}}{\sum_{(#1)}}}
\def\dsm#1{\@dsm{#1}\@ifnextchar\bgroup{\dsm}{}}
\def\@dsm#1{\sum_{(#1)}\,}
\def\@dsm@noparens#1{\sum_{#1}\,}

\def\wtype#1{\@ifnextchar\bgroup%
  {\mathchoice{\@twtype{#1}}{\@twtype{#1}}{\@twtype{#1}}{\@twtype{#1}}\wtype}%
  {\mathchoice{\@twtype{#1}}{\@twtype{#1}}{\@twtype{#1}}{\@twtype{#1}}}}
\def\lwtype#1{\@ifnextchar\bgroup{\@lwtype{#1}\lwtype}{\@@lwtype{#1}}}
\def\@lwtype#1{\mathchoice{{\textstyle\mathsf{W}}}{\mathsf{W}}{\mathsf{W}}{\mathsf{W}}({\textstyle #1})\;} 
\def\@@lwtype#1{\mathchoice{{\textstyle\mathsf{W}}}{\mathsf{W}}{\mathsf{W}}{\mathsf{W}}({\textstyle #1}),\ } 
\def\twtype#1{\@twtype{#1}\@ifnextchar\bgroup{\twtype}{}}
\def\@twtype#1{\mathchoice{{\textstyle\mathsf{W}_{(#1)}}}{\mathsf{W}_{(#1)}}{\mathsf{W}_{(#1)}}{\mathsf{W}_{(#1)}}}
\def\dwtype#1{\@dwtype{#1}\@ifnextchar\bgroup{\dwtype}{}}
\def\@dwtype#1{\mathsf{W}_{(#1)}\,}

\def\wtypeh#1{\@ifnextchar\bgroup%
  {\mathchoice{\@lwtypeh{#1}}{\@twtypeh{#1}}{\@twtypeh{#1}}{\@twtypeh{#1}}\wtypeh}%
  {\mathchoice{\@@lwtypeh{#1}}{\@twtypeh{#1}}{\@twtypeh{#1}}{\@twtypeh{#1}}}}
\def\lwtypeh#1{\@ifnextchar\bgroup{\@lwtypeh{#1}\lwtypeh}{\@@lwtypeh{#1}}}
\def\@lwtypeh#1{\mathchoice{{\textstyle\mathsf{W}^h}}{\mathsf{W}^h}{\mathsf{W}^h}{\mathsf{W}^h}({\textstyle #1})\;} 
\def\@@lwtypeh#1{\mathchoice{{\textstyle\mathsf{W}^h}}{\mathsf{W}^h}{\mathsf{W}^h}{\mathsf{W}^h}({\textstyle #1}),\ } 
\def\twtypeh#1{\@twtypeh{#1}\@ifnextchar\bgroup{\twtypeh}{}}
\def\@twtypeh#1{\mathchoice{{\textstyle\mathsf{W}^h_{(#1)}}}{\mathsf{W}^h_{(#1)}}{\mathsf{W}^h_{(#1)}}{\mathsf{W}^h_{(#1)}}}
\def\dwtypeh#1{\@dwtypeh{#1}\@ifnextchar\bgroup{\dwtypeh}{}}
\def\@dwtypeh#1{\mathsf{W}^h_{(#1)}\,}

\makeatother

\let\setof\Set

\newcommand{\proj}[1]{\ensuremath{\mathsf{pr}_{#1}}\xspace}

\newcommand{\ind}[1]{\mathsf{ind}_{#1}}

\newcommand{\pairr}[1]{{\mathopen{}(#1)\mathclose{}}} 

\newcommand{\im}{\ensuremath{\mathsf{im}}} 

\newcommand{\modal}{\ensuremath{\ocircle}}
\newcommand{\modaltype}{\ensuremath{\type_\modal}}

\newcommand{\idsym}{{=}}
\newcommand{\id}[3][]{\ensuremath{#2 =_{#1} #3}\xspace}

\newcommand{\dpath}[4]{#3 =^{#1}_{#2} #4}

\newcommand{\refl}[1]{\ensuremath{\mathsf{refl}_{#1}}\xspace}

\newcommand{\ct}{%
  \mathchoice{\mathbin{\raisebox{0.5ex}{$\displaystyle\centerdot$}}}%
             {\mathbin{\raisebox{0.5ex}{$\centerdot$}}}%
             {\mathbin{\raisebox{0.25ex}{$\scriptstyle\,\centerdot\,$}}}%
             {\mathbin{\raisebox{0.1ex}{$\scriptscriptstyle\,\centerdot\,$}}}
}

\newcommand{\opp}[1]{\mathord{{#1}^{-1}}}
\newcommand{\trans}[2]{\ensuremath{{#1}_{*}\mathopen{}\left({#2}\right)\mathclose{}}\xspace}
\newcommand{\mapfunc}[1]{\ensuremath{\mathsf{ap}_{#1}}\xspace} 

\let\apfunc\mapfunc%
\let\ap\map%
\newcommand{\idfunc}[1][]{\ensuremath{\mathsf{id}_{#1}}\xspace}

\newcommand{\htpy}{\sim}

\newcommand{\eqv}[2]{\ensuremath{#1 \simeq #2}\xspace}

\newcommand{\eqvsym}{\simeq}    

\newcommand{\isequiv}{\ensuremath{\mathsf{isequiv}}}

\newcommand{\hfib}[2]{{\mathsf{fib}}_{#1}(#2)} 

\newcommand{\total}[1]{\ensuremath{\mathsf{total}(#1)}}

\newcommand{\UU}{\ensuremath{\mathcal{U}}\xspace}
\let\type\UU%

\renewcommand{\prop}{\ensuremath{\mathsf{Prop}}\xspace}

\newcommand{\pointed}[1]{\ensuremath{#1_\bullet}}

\newcommand{\iscontr}{\ensuremath{\mathsf{isContr}}}

\newcommand{\trunc}[2]{\mathopen{}\left\Vert#2\right\Vert_{#1}\mathclose{}}

\newcommand{\truncf}[1]{\Vert \blank \Vert_{#1}} 

\newcommand{\brck}[1]{\trunc{}{#1}}

\newcommand{\Parens}[1]{\Bigl(#1\Bigr)}

\newcommand{\emptyt}{\ensuremath{\mathbf{0}}\xspace}

\newcommand{\unit}{\ensuremath{\mathbf{1}}\xspace}
\newcommand{\ttt}{\ensuremath{\star}\xspace}

\newcommand{\bool}{\ensuremath{\mathbf{2}}\xspace}
\newcommand{\btrue}{{1_{\bool}}}
\newcommand{\bfalse}{{0_{\bool}}}

\newcommand{\inlsym}{{\mathsf{inl}}}
\newcommand{\inrsym}{{\mathsf{inr}}}
\newcommand{\inl}{\ensuremath\inlsym\xspace}
\newcommand{\inr}{\ensuremath\inrsym\xspace}

\newcommand{\Sn}{\mathbb{S}}

\newcommand{\susp}{\Sigma}

\newcommand{\blank}{\mathord{\hspace{1pt}\text{--}\hspace{1pt}}}

\newcommand{\nameless}{\mathord{\hspace{1pt}\underline{\hspace{1ex}}\hspace{1pt}}}

\newcommand{\inv}[1]{{#1}^{-1}}

\renewcommand{\fact}{\mathsf{fact}}

\makeatletter
\def\defthm#1#2#3{%
  \newaliascnt{#1}{thm}
  \newtheorem{#1}[#1]{#2}
  \aliascntresetthe{#1}
  \crefname{#1}{#2}{#3}}

\crefname{thm}{Theorem}{Theorems}
\theoremstyle{plain}            
\defthm{corollary}{Corollary}{Corollaries}
\defthm{lemma}{Lemma}{Lemmas}
\theoremstyle{definition}
\defthm{defn}{Definition}{Definitions}
\theoremstyle{remark}
\defthm{rmk}{Remark}{Remarks}
\defthm{eg}{Example}{Examples}
\defthm{egs}{Examples}{Examples}
\defthm{notes}{Notes}{Notes}

\crefformat{section}{\S#2#1#3}
\Crefformat{section}{Section~#2#1#3}
\crefrangeformat{section}{\S\S#3#1#4--#5#2#6}
\Crefrangeformat{section}{Sections~#3#1#4--#5#2#6}
\crefmultiformat{section}{\S\S#2#1#3}{ and~#2#1#3}{, #2#1#3}{ and~#2#1#3}
\Crefmultiformat{section}{Sections~#2#1#3}{ and~#2#1#3}{, #2#1#3}{ and~#2#1#3}
\crefrangemultiformat{section}{\S\S#3#1#4--#5#2#6}{ and~#3#1#4--#5#2#6}{, #3#1#4--#5#2#6}{ and~#3#1#4--#5#2#6}
\Crefrangemultiformat{section}{Sections~#3#1#4--#5#2#6}{ and~#3#1#4--#5#2#6}{, #3#1#4--#5#2#6}{ and~#3#1#4--#5#2#6}

\crefformat{appendix}{Appendix~#2#1#3}
\Crefformat{appendix}{Appendix~#2#1#3}
\crefrangeformat{appendix}{Appendices~#3#1#4--#5#2#6}
\Crefrangeformat{appendix}{Appendices~#3#1#4--#5#2#6}
\crefmultiformat{appendix}{Appendices~#2#1#3}{ and~#2#1#3}{, #2#1#3}{ and~#2#1#3}
\Crefmultiformat{appendix}{Appendices~#2#1#3}{ and~#2#1#3}{, #2#1#3}{ and~#2#1#3}
\crefrangemultiformat{appendix}{Appendices~#3#1#4--#5#2#6}{ and~#3#1#4--#5#2#6}{, #3#1#4--#5#2#6}{ and~#3#1#4--#5#2#6}
\Crefrangemultiformat{appendix}{Appendices~#3#1#4--#5#2#6}{ and~#3#1#4--#5#2#6}{, #3#1#4--#5#2#6}{ and~#3#1#4--#5#2#6}

\crefname{part}{Part}{Parts}

\crefformat{paragraph}{\S#2#1#3}
\Crefformat{paragraph}{Paragraph~#2#1#3}
\crefrangeformat{paragraph}{\S\S#3#1#4--#5#2#6}
\Crefrangeformat{paragraph}{Paragraphs~#3#1#4--#5#2#6}
\crefmultiformat{paragraph}{\S\S#2#1#3}{ and~#2#1#3}{, #2#1#3}{ and~#2#1#3}
\Crefmultiformat{paragraph}{Paragraphs~#2#1#3}{ and~#2#1#3}{, #2#1#3}{ and~#2#1#3}
\crefrangemultiformat{paragraph}{\S\S#3#1#4--#5#2#6}{ and~#3#1#4--#5#2#6}{, #3#1#4--#5#2#6}{ and~#3#1#4--#5#2#6}
\Crefrangemultiformat{paragraph}{Paragraphs~#3#1#4--#5#2#6}{ and~#3#1#4--#5#2#6}{, #3#1#4--#5#2#6}{ and~#3#1#4--#5#2#6}

\crefname{figure}{Figure}{Figures}

\let\autoref\cref%

\let\c@equation\c@thm%
\numberwithin{equation}{section}

\setitemize[1]{leftmargin=2em}
\setenumerate[1]{leftmargin=*}

\setitemize[1]{itemsep=-0.2em}
\setenumerate[1]{itemsep=-0.2em}

\def\noteson{%
\gdef\note##1{\mbox{}\marginpar{\color{blue}\textasteriskcentered\ ##1}}}

\noteson%

\newcounter{symindex}

\makeatletter

\renewcommand{\id}[3][]{
  \@ifnextchar\bgroup%
    {#2 \mathbin{\idsym_{#1}} \id[#1]{#3}}
    {#2 \mathbin{\idsym_{#1}} #3}
  }

\renewcommand{\eqv}[2]{
  \@ifnextchar\bgroup%
    {#1 \eqvsym \eqv{#2}} 
    {#1 \eqvsym #2} 
  }

\newcommand{\ctsym}{%
  \mathchoice{\mathbin{\raisebox{0.5ex}{$\displaystyle\centerdot$}}}%
             {\mathbin{\raisebox{0.5ex}{$\centerdot$}}}%
             {\mathbin{\raisebox{0.25ex}{$\scriptstyle\,\centerdot\,$}}}%
             {\mathbin{\raisebox{0.1ex}{$\scriptscriptstyle\,\centerdot\,$}}}
  }

\renewcommand{\ct}[3][]{
  \@ifnextchar\bgroup%
    {#2 \mathbin{\ctsym_{#1}} \ct[#1]{#3}}
    {#2 \mathbin{\ctsym_{#1}} #3}
  }

\makeatother

\makeatletter
\renewcommand{\@dprd}{\@tprd}
\renewcommand{\@dsm}{\@tsm}
\renewcommand{\@dprd@noparens}{\@tprd}
\renewcommand{\@dsm@noparens}{\@tsm}

\renewcommand{\@tprd}[1]{\mathchoice{{\textstyle\prod_{(#1)}\,}}{\prod_{(#1)}\,}{\prod_{(#1)}\,}{\prod_{(#1)}\,}}
\renewcommand{\@tsm}[1]{\mathchoice{{\textstyle\sum_{(#1)}\,}}{\sum_{(#1)}\,}{\sum_{(#1)}\,}{\sum_{(#1)}\,}}

\newcommand{\implicitargumentson}{\boolean{true}}

\newcommand{\@ifnextchar@starorbrace}[2]
  {\@ifnextchar*{#1}{\@ifnextchar\bgroup{#1}{#2}}}

\renewcommand{\prd}{\@ifnextchar*{\@iprd}{\@prd}}

\newcommand{\@prd}[1]
  {\@ifnextchar@starorbrace%
    {\prd@parens{#1}}
    {\@ifnextchar\sm{\prd@parens{#1}\@eatsm}{\prd@noparens{#1}}}}
\newcommand{\@prd@parens}{\@ifnextchar*{\@iprd}{\prd@parens}}
\renewcommand{\prd@parens}[1]
  {\@ifnextchar@starorbrace%
    {\@theprd{#1}\@prd@parens}
    {\@ifnextchar\sm{\@theprd{#1}\@eatsm}{\@theprd{#1}}}}
\newcommand{\@theprd}[1]
  {\mathchoice{\@dprd{#1}}{\@tprd{#1}}{\@tprd{#1}}{\@tprd{#1}}}
\renewcommand{\dprd}[1]{\@dprd{#1}\@ifnextchar@starorbrace{\dprd}{}}
\renewcommand{\tprd}[1]{\@tprd{#1}\@ifnextchar@starorbrace{\tprd}{}}

\newcommand{\@theiprd}[1]{\mathchoice{\@diprd{#1}}{\@tiprd{#1}}{\@tiprd{#1}}{\@tiprd{#1}}}
\newcommand{\@iprd}[2]{\@ifnextchar@starorbrace%
  {\@theiprd{#2}\@prd@parens}%
  {\@ifnextchar\sm%
    {\@theiprd{#2}\@eatsm}%
    {\@theiprd{#2}}}}
\def\@tiprd#1{
  \ifthenelse{\implicitargumentson}
    {\@@tiprd{#1}\@ifnextchar\bgroup{\@tiprd}{}}
    {\@tprd{#1}}}
\def\@@tiprd#1{\mathchoice{{\textstyle\prod_{\{#1\}}\,}}{\prod_{\{#1\}}\,}{\prod_{\{#1\}}\,}{\prod_{\{#1\}}\,}}
\def\@diprd{
  \ifthenelse{\implicitargumentson}
    {\@tiprd}
    {\@tprd}}

\def\@eatprd\prd{\@prd@parens}

\makeatother

\makeatletter
\def\tfall#1{\forall_{(#1)}\@ifnextchar\bgroup{\,\tfall}{\,}}
\renewcommand{\fall}{\tfall}

\def\texis#1{\exists_{(#1)}\@ifnextchar\bgroup{\,\texis}{\,}}
\renewcommand{\exis}{\texis}

\def\uexis#1{\exists!_{(#1)}\@ifnextchar\bgroup{\,\uexis}{\,}}
\makeatother

\makeatletter

\newcommand{\mfam}[2][]{%
  \mathcal{F}_{\default@ctxext #2}^{#1}} 
\newcommand{\@mfam@nested}[1]{\@mfam@parens}
\newcommand{\@mfam@parens}[2][]{(\mfam[#1]{#2})}

\newcommand{\mtm}[2][]{%
  \mathcal{T}_{\default@ctxext #2}^{#1}} 
\newcommand{\@mtm@nested}[1]{\@mtm@parens}
\newcommand{\@mtm@parens}[2][]{(\mtm[#1]{#2})}

\newcommand{\tfemp}[1]{%
  \typefont{emp}_{\default@ctxext #1}} 
\newcommand{\tft}[1]{%
  \typefont{t}_{\default@ctxext #1}} 

\newcommand{\tfext}[1]{%
  \typefont{ext}_{\default@ctxext #1}} 

\newcommand{\tfsubst}[1]{%
  \typefont{subst}_{\default@ctxext #1}} 

\newcommand{\tfwk}[1]{%
  \typefont{wk}_{\default@ctxext #1}} 

\newcommand{\tfid}[1]{%
  \typefont{idtm}_{\default@ctxext #1}} 
\makeatother

\newcommand{\typefont}{\mathsf} 

\renewcommand{\UU}{\typefont{U}}
\renewcommand{\isequiv}{\typefont{isEquiv}}

\renewcommand{\type}{\typefont{U}}

\renewcommand{\susp}{\typefont{\Sigma}}

\newcommand{\hfibfunc}[1]{\typefont{fib}_{#1}}

\newcommand{\strunc}[2]{\Vert#2\Vert_{#1}}

\renewcommand{\modal}{{\ensuremath{\ocircle}}}

\renewcommand{\sslash}{/\!\!/}

\newcommand{\tfindf}[2][]{\typefont{ind}_{#2}^{#1}}

\makeatletter
\newcommand{\tfind}[3][]{\tfindf[#1]{#2}(\default@ctxext #3)} 
\makeatother

\makeatletter
\newcommand{\pt}[1][]{*_{
  \@ifnextchar\undergraph{\@undergraph@nested}
    {\@ifnextchar\underovergraph{\@underovergraph@nested}{}}#1}}
\makeatother

\makeatletter
\newcommand{\pts}[1]{{\@graphop@nested{#1}}_{0}}
\newcommand{\edg}[1]{{\@graphop@nested{#1}}_{1}}
\newcommand{\@graphop@nested}[1]
  {\@ifnextchar\ctxext{\@ctxext@nested}
      {\@ifnextchar\undergraph{\@undergraph@nested}
         {\@ifnextchar\underovergraph{\@underovergraph@nested}{}}}
    #1}
\makeatother

\makeatletter
\newcommand{\@undergraphtest}[2]{\@ifnextchar({#1}{#2}} 
\newcommand{\undergraph}[2]{\@undergraphtest{\@undergraph@parens{#1}{#2}}{\@undergraph{#1}{#2}}}
\newcommand{\@undergraph}[2]{{#2/#1}}
\newcommand{\@undergraph@nested}[3]{\@undergraph@parens{#2}{#3}}
\newcommand{\@undergraph@parens}[2]{(\@undergraph{#1}{#2})}
\makeatother

\makeatletter
\newcommand{\underovergraph}[2]{\@underovergraphtest{\@underovergraph@parens{#1}{#2}}{\@underovergraph{#1}{#2}}}
\newcommand{\@underovergraph}[2]{{#2}\,{\parallel}\,{#1}}
\newcommand{\@underovergraphtest}{\@undergraphtest}
\newcommand{\@underovergraph@parens}[2]{(\@underovergraph{#1}{#2})}
\newcommand{\@underovergraph@nested}[3]{\@underovergraph@parens{#2}{#3}}
\makeatother

\tikzset{patharrow/.style={double,double equal sign distance,-,font=\scriptsize}}
\tikzset{description/.style={fill=white,inner sep=2pt}}
\tikzset{fib/.style={->>,font=\scriptsize}}

\tikzset{commutative diagrams/column sep/Huge/.initial=18ex}

\newlength\minalignvsep%

\makeatletter
\def\align@preamble{%
   &\hfil
    \setboxz@h{\@lign$\m@th\displaystyle{##}$}%
    \ifnum\row@>\@ne%
    \ifdim\ht\z@>\ht\strutbox@%
    \dimen@\ht\z@%
    \advance\dimen@\minalignvsep%
    \ht\strutbox\dimen@%
    \fi\fi%
    \strut@%
    \ifmeasuring@\savefieldlength@\fi%
    \set@field%
    \tabskip\z@skip%
   &\setboxz@h{\@lign$\m@th\displaystyle{{}##}$}%
    \ifnum\row@>\@ne%
    \ifdim\ht\z@>\ht\strutbox@%
    \dimen@\ht\z@%
    \advance\dimen@\minalignvsep%
    \ht\strutbox@\dimen@%
    \fi\fi%
    \strut@%
    \ifmeasuring@\savefieldlength@\fi%
    \set@field%
    \hfil%
    \tabskip\alignsep@%
}
\makeatother

\setdescription[1]{itemsep=-0.2em}

\newcommand{\modalunit}[1][]{{\eta_{#1}}}

\newcommand{\localization}[1]{\mathcal{L}_{#1}}
\newcommand{\localhit}[2]{\mathcal{J}_{#1}(#2)} 
\newcommand{\factorhit}[3]{\mathcal{J}_{#1}^{#2}(#3)} 
\newcommand{\ismodal}{\ensuremath{\mathsf{isModal}}}

\usetikzlibrary{matrix}
\tikzset{ar/.style={->,font=\scriptsize}}
\tikzset{std/.style={matrix of math nodes, row sep=3em, column
    sep=4em, text height=1.5ex, text depth=0.25ex}}

\newcommand{\fillers}[1]{\mathsf{fill}(#1)} 
\newcommand{\dfill}[1]{\mathsf{dfill}(#1)} 
\newcommand{\cL}{\mathcal{L}}
\newcommand{\cR}{\mathcal{R}}

\newcommand{\open}[1]{\mathsf{Op}_{#1}}
\newcommand{\closed}[1]{\mathsf{Cl}_{#1}}
\newcommand{\truncmod}[1]{\mathcal{T}\!\mathit{r}_{#1}}
\newcommand{\shmod}[1]{\mathcal{S}\mathit{h}_{#1}}

\newcommand{\rsu}[1][\UU]{\mathsf{RSU}_{#1}}
\newcommand{\mdl}[1][\UU]{\mathsf{Mdl}_{#1}}
\newcommand{\lex}[1][\UU]{\mathsf{Lex}_{#1}}
\newcommand{\tpl}[1][\UU]{\mathsf{Top}_{#1}}
\newcommand{\accrsu}[1][\UU]{\mathsf{AccRSU}_{#1}}
\newcommand{\accmdl}[1][\UU]{\mathsf{AccMdl}_{#1}}
\newcommand{\acclex}[1][\UU]{\mathsf{AccLex}_{#1}}

\begin{document}

\maketitle

\begin{abstract}
Univalent homotopy type theory (HoTT) may be seen as a language for the category of $\infty$-groupoids.
It is being developed as a new foundation for mathematics and as an
internal language for (elementary) higher toposes.
We develop the theory of factorization systems, reflective subuniverses, and modalities in homotopy type theory, including their construction using a ``localization'' higher inductive type.
This produces in particular the ($n$-connected, $n$-truncated)
factorization system as well as internal presentations of subtoposes,
through lex modalities.
We also develop the semantics of these constructions.
\end{abstract}

\setcounter{tocdepth}{2}
\tableofcontents



\section*{Introduction}

In traditional modal logic, a \emph{modality} is a unary operation on propositions.
The classical examples are $\Box$ (``it is necessary that'') and $\lozenge$ (``it is possible that'').
In type theory and particularly dependent type theory, such as homotopy type theory, where propositions are regarded as certain types, it is natural to extend the notion of modality to a unary operation on \emph{types}.
For emphasis we may call this a ``typal modality'', or a ``higher modality'' since it acts on the ``higher types'' available in homotopy type theory (not just ``sets'' but types containing higher homotopy).

There are many kinds of propositional modalities, but many of them are either \emph{monads} or \emph{comonads}.
Monads and comonads on a poset (such as the poset of propositions) are also automatically \emph{idempotent}, but this is no longer true for more general monads and comonads.
Thus there are many possible varieties of typal and higher modalities.

Typal modalities in non-dependent type theory have a wide range of applications in computer science.
In particular, following the pioneering work of~\cite{moggi:monads}, monadic typal modalities are commonly used to model effects in programming languages.
Non-dependent modal type theory is now a flourishing field with this and many other applications; see~\cite{dpgm:modal-tt} for an overview.

In this paper we take a first step towards the study of higher modalities in \emph{homotopy} type theory, restricting our attention to \emph{idempotent, monadic} ones.
These are especially convenient for a number of reasons.
One is that in homotopy type theory, as in higher category theory, we expect a general monad (or comonad) to require infinitely many higher coherence conditions, which we don't know how to express in the finite syntax of type theory; whereas an \emph{idempotent} one can instead be described using the universal property of a reflector into a subcategory.
(We can still use \emph{particular} non-idempotent monadic modalities, such as the ``partial elements'' monad of~\cite{Partiality,ek:partial}, without making all this coherence explicit, but it is harder to develop a general theory of them.)

Another is that in good situations, an idempotent monad can be extended to all slice categories consistently, and thereby represented ``fully internally'' in type theory as an operation $\modal : \UU\to\UU$ on a type universe.
Idempotent \emph{comonadic} modalities have also been considered in dependent type theory and homotopy type theory (see for instance~\cite{npp:ctx-modal-tt,pr:fib-modal-tt,SchreiberShulman,shulman:bfp-realcohesion}), but they generally require modifying the judgmental structure of type theory.
By contrast, our theory of modalities can be (and has been) formalized in existing proof assistants without modifying the underlying type theory.

Idempotent monadic modalities also include many very important
examples. The $(-1)$-truncation in homotopy type theory is a
higher-dimensional version of the bracket modality, which in 1-category theory
characterizes regular categories~\cite{AwodeyBauer2004}. More
generally, the $n$-truncation modalities are prominent examples of modalities; indeed almost all of the theory of truncation and connectedness in~\cite[Chapter 7]{TheBook} is just a specialization of the theory of a general modality.
More generally, we can produce idempotent monadic modalities by \emph{localization} or \emph{nullification} at small families, using a higher inductive type.
Finally, among idempotent monadic modalities we also find the \emph{left exact} ones, which correspond semantic\-ally to subtoposes.

For the rest of this paper we will say simply \emph{modality} to mean an idempotent monadic modality.
However, this should be regarded as only a local definition; in more general contexts the word ``modality'' should continue to encompass comonadic modalities and other sorts.

In fact, our use of the word ``modality'' will be a little more specific even than this.
If we express internally the most na\"{\i}ve notion of ``idempotent monad on $\UU$'', we obtain a notion that we call a \emph{reflective subuniverse}.
However, many reflective subuniverses that arise in practice, including truncation and left exact modalities (and, in fact, all concrete examples we will consider in this paper), satisfy the further property of being closed under $\Sigma$-types; it is these that we will call \emph{modalities}.
We emphasize this property not just because it holds in many examples, but because it can be equivalently expressed by giving the modal operator a \emph{dependent elimination principle} analogous to that of an inductive type.
This is a very natural thing to ask for when generalizing propositional modalities to typal operations.

The naturalness of this notion of modality is further supported by the fact that it has many equivalent characterizations.
In addition to a reflective subuniverse closed under $\Sigma$-types and a modal operator with a dependent eliminator, a modality can be defined using a ``dependent universal property'', and more interestingly as a stable orthogonal factorization system.
The right class of maps in the factorization system consists of those whose fibers belong to the subuniverse (``modal maps''), while the left class consists of those whose fibers have contractible reflection into the subuniverse (``connected maps'').
The internal nature of the definition means that a \emph{stable} factorization system is entirely determined by the fibers of its right class, which form a modality.\footnote{Non-stable factorization systems are not so determined, although they do have an underlying reflective subuniverse, and most reflective subuniverses can be extended to factorization systems.}
We prove the equivalence of all these definitions in \autoref{sec:modal-refl-subun}, developing along the way some basic theory of reflective subuniverses, connected maps, and factorization systems.

In unaugmented Martin-L\"{o}f type theory we can define a few particular modalities, such as the double-negation modality, and the ``open modality'' associated to any mere proposition.
However, most interesting modalities require higher inductive types for their construction, including the $n$-truncations and the dual ``closed modality'' associated to a proposition.
In \autoref{sec:localization} we give a general construction of modalities using a higher inductive \emph{localization} type: given a family of maps $F:\prd{a:A} B(a) \to C(a)$, a type $X$ is \emph{$F$-local} if the precomposition map $(C(a)\to X) \to (B(a)\to X)$ is an equivalence for all $a:A$, and the \emph{$F$-localization} $\localization{F} X$ is the universal $F$-local type admitting a map from $X$.
We call a modality \emph{accessible} if it can be generated by localization; this is inspired by the corresponding notion in category theory.
Accessible modalities include the $n$-truncation and open and closed modalities, as well as many examples from homotopy theory, where localization is a standard technique; thus we expect them to be a useful tool in the synthetic development of homotopy theory inside type theory.\footnote{Our notion of localization, being \emph{internal}, is a little stronger than the standard sort of localization in homotopy theory; but in many cases it is equivalent.
  The higher inductive construction of localization, when interpreted model-categorically according to the semantics of~\cite{ls:hits}, also appears to be new and may be of independent interest in homotopy theory.}

In general, localization at a family of maps produces a reflective subuniverse (and, in fact, an orthogonal factorization system), but not necessarily a modality.
However, there is a simple condition which ensures that we do get a modality, namely that $C(a)=\unit$ for all $a:A$.
In this case the local types are those for which ``every map $B(a)\to X$ is uniquely constant''; following standard terminology in homotopy theory we call them \emph{$B$-null} and the corresponding localization \emph{$B$-nullification}.
Any accessible modality can be presented as a nullification.

A very important class of modalities that \emph{excludes} the $n$-truncations are the left exact, or \emph{lex}, ones, which we study in \cref{sec:left-exact-modal}.
These have many equivalent characterizations, but the most intuitive is simply that the reflector preserves finite limits.
When homotopy type theory is regarded as an internal language for higher toposes, lex modalities correspond to subtoposes.
In the traditional internal logic of 1-toposes, subtoposes are
represented by \emph{Lawvere-Tierney operators} on the subobject
classifier, which generate a subtopos by internal
sheafification. Goldblatt~\cite{goldblatt2010cover} provides an overview
of the modal logic perspective on these operators on propositions.
Dependent type theory allows us to speak directly about the subtopos as an operation on a type universe (the lex modality), and show internally that any Lawvere-Tierney operator on the universe of propositions gives rise to a lex modality.

There is an additional subtlety here that only arises for $\infty$-toposes and homotopy type theory.
In 1-topos theory, and indeed in $n$-topos theory for any $n<\infty$, every lex modality (subtopos) arises from a Lawvere-Tierney operator; but in $\infty$-topos theory this is no longer true.
The subtoposes that are determined by their behavior on propositions are called \emph{topological} in~\cite{lurie2009higher}, and we appropriate this name for lex modalities of this sort as well.
The dual \emph{cotopological} sort of lex modalities, including the \emph{hypercompletion}, are harder to construct in type theory, but we can at least show that insofar as they exist they behave like their $\infty$-categorical analogues.

When this paper was written, we did not know any condition on a type family $B$ that ensured that $B$-nullification is lex and such that \emph{any} accessible lex modality can be presented by such a $B$.
But as we were preparing it for final publication,~\cite{abfj:lexloc} found such a condition: that $B$ is closed under taking path spaces.
In this case we may refer to $B$-nullification as a \emph{lex nullification}.

\cref{fig:venn} displays in a Venn diagram all the different structures discussed above.
Lex modalities are a subclass of modalities, which are a subclass of reflective subuniverses.
In principle all three structures can be either accessible or non-accessible, although in practice non-accessible ones are very hard to come by; with topological modalities a subclass of the accessible lex ones.
Individual examples are displayed in single boxes, while general classes of examples (obtained by localization and restricted classes thereof) are displayed in double boxes.

\begin{figure}
  \centering
  \begin{tikzpicture}
    \draw (0,0.7) ellipse (6 and 3.8);
    \node at (0,4.2) {\small reflective subuniverses};
    \draw (-1.3,.5) ellipse (4.3 and 3.2);
    \node at (-1.3,3.4) {\small accessible};
    \draw (1.2,.3) ellipse (4.2 and 3.1);
    \node at (2.8,2.7) {\small modalities};
    \draw (1.3,-.5) ellipse (3.7 and 2);
    \node at (3.5,.8) {\small lex};
    \draw (0,-.8) ellipse (2 and 1.3);
    \node at (0,0.2) {\small topological};
    \node[rectangle,draw] at (0,2.5) {\tiny $n$-truncation};
    \node[rectangle,draw,double] at (0,2) {\tiny nullifications};
    \node[rectangle,draw,double] at (1.5,1) {\tiny lex nullifications};
    \node[rectangle,draw,double] at (-4.3,1) {\tiny localizations};
    \node[rectangle,draw] at (3.8,1.7) {\tiny double neg.};
    \node[rectangle,draw] at (-.5,-.5) {\tiny o\smash{p}en};
    \node[rectangle,draw] at (.5,-.5) {\tiny c\smash{l}ose\smash{d}};
    \node[rectangle,draw,double] at (0,-1) {\tiny prop.~nullifications};
    \node[rectangle,draw] at (3,0) {\tiny hypercompletion?};
  \end{tikzpicture}
  \caption{Modalities and related structures}%
  \label{fig:venn}
\end{figure}

Viewing accessible lex modalities as subtoposes, we naturally expect that the subtopos should support its own internal language.
This is true, although we do not prove it precisely; we simply observe that the universe of modal types is closed under many type constructors and admits its own versions of all the others.
In particular, the universe of modal types for an accessible lex modality is \emph{itself} a modal type for the same modality (in fact, this characterizes lex modalities among accessible ones).
Since any $\infty$-topos arises as a subtopos of a presheaf $\infty$-topos, we can essentially reduce the problem of finding univalent internal languages for $\infty$-toposes to that of finding them for presheaf $\infty$-toposes (and of finding universes closed under accessible lex modalities; see \cref{thm:subtopos-model,sec:semantics}).
A similar argument, using judgementally strict idempotent monads, has already been used in the so-called ``cubical stack'' models of type theory~\cite{Stacks,Coquand:stack} (which do not actually in general lie in $\infty$-stack toposes) to prove independence results for homotopy type theory.

We end the main part of the paper with a general ``fracture and gluing'' theorem about modalities: if $\modal$ is any modality and $\lozenge$ is a lex modality that is ``strongly disjoint'' from $\modal$, then the join $\lozenge\lor\modal$ in the poset of modalities can be constructed using a ``pullback fracture square''.
When applied to the open and closed modalities associated to a proposition, this specializes to an internal viewpoint on Artin gluing.
We call it a ``fracture theorem'' since the pullback squares appear formally analogous to the fracture squares in the classical theory of localization and completion at primes, though we do not know of a precise relationship.

In the final part of the paper, \autoref{sec:semantics}, we sketch a semantic interpretation of our theory in terms of comprehension categories and $(\infty,1)$-toposes.
In particular, we show that well-behaved reflective subcategories of $(\infty,1)$-toposes give rise to modalities in their internal languages, while dually modalities give rise to reflective subcategories of syntactic $(\infty,1)$-categories.
In this discussion we ignore the issue of universes, which it is not known how to model semantically in general $(\infty,1)$-toposes (except in a weak sense).

We will freely use the results and the notations from~\cite{TheBook}.
In fact, parts of this work have already appeared as~\cite[sec~7.6--7]{TheBook}.
We generalize much of this section~7.6 to general modalities in our \cref{sec:modal-refl-subun}, which also sharpens the results in~\cite[sec 7.7]{TheBook}.
In particular, we will freely use function extensionality and the univalence axiom, often without comment.

Finally, we note that many of the results in this paper have been formalized in the Coq proof assistant~\cite{HoTT-CPP}.
However, the organization of results in the library is rather different than in this paper.
A rough correspondence is as follows; unless otherwise noted all files are in the \texttt{Modalities/} directory.
\begin{center}
  \begin{tabular}{ll}
    Sections & Library files \\\toprule
    \cref{sec:ssrs} & \texttt{ReflectiveSubuniverse.v} and \texttt{Modality.v}\\
    Examples (\cref{sec:higher-modalities}) & \texttt{Identity.v}, \texttt{Notnot.v}, \texttt{Open.v}, \texttt{Closed.v}, \\ & and \texttt{../HIT/Truncations.v} \\
    \cref{sec:ofs} & \texttt{../Factorization.v} \\
    \cref{sec:sofs2} & \texttt{Modality.v} \\
    \cref{sec:local-types,sec:localizing} & \texttt{Localization.v} \\
    \cref{sec:nullification} & \texttt{Nullification.v} and \texttt{Accessible.v} \\
    \cref{sec:lex-top-cotop} & \texttt{Lex.v} and \texttt{Topological.v} \\
    \cref{sec:fracture} & \texttt{Fracture.v}
  \end{tabular}
\end{center}
There are also some differences in the proof techniques used in the library and in this paper.
In the library, localizations are constructed using ``$\infty$-extendability'' as a characterization of equivalences to avoid function extensionality hypotheses, as described in~\cite{shulman:up-wo-fe}.
In addition, much attention is paid to ensuring appropriate universe polymorphism with parametrized modules; this is described in~\cite[\S5]{HoTT-CPP}.
We will not discuss these issues further here; see the cited references and the comments in the library for more information.

\section{Modalities, reflective subuniverses and factorization systems}\label{sec:modal-refl-subun}

In this section we will introduce the following four notions of modality
and prove that they are all equivalent:
\begin{enumerate}
\item Higher modalities
\item Uniquely eliminating modalities
\item $\Sigma$-closed reflective subuniverses
\item Stable orthogonal factorization systems
\end{enumerate}
After their equivalence has been established, we will call all of them simply \emph{modalities}.

The first three definitions have the following data in common: by a \define{modal operator} we mean a function $\modal:\UU\to\UU$, and by a \define{modal unit} we mean a family of functions $\modalunit^\modal:\prd*{A:\UU}A\to\modal A$.\footnote{In general we write $f:\prd*{x:A} B(x)$ instead of $f:\prd{x:A} B(x)$ to indicate that the argument $x$ of $f$ is implicit.}
Given these data, we say a type $X$ \define{is modal} if $\modalunit[X]:X\to\modal X$ is an equivalence, and we write $\UU_\modal \defeq \sm{X:\UU} \ismodal(X)$ for the \define{subuniverse of modal types}.
More generally, if $\mathcal{M}:\UU\to\prop$ is any predicate on the universe, we write $\UU_{\mathcal{M}} \defeq \sm{X:\UU} \mathcal{M}(X)$.

\begin{defn}\label{defn:highermod}
A \define{higher modality} consists of a modal operator and modal unit together with
\begin{enumerate}
\item for every $A:\UU$ and every dependent type $P:\modal A\to\UU$, a
function
\begin{equation*}
\mathsf{ind}^\modal_A:\big(\prd{a:A}\modal(P(\eta(a)))\big)\to\prd{z:\modal A}\modal(P(z)).
\end{equation*}
\item An identification
\begin{equation*}
\mathsf{comp}^\modal_A(f,x):\id{\mathsf{ind}^\modal_A(f)(\eta(x))}{f(x)}
\end{equation*}
for each $f:\prd{x:A}\modal(P(\eta(x)))$ and $x:A$.
\item For any $x,y:\modal A$ the modal unit $\modalunit[(\id{x}{y})]:\id{x}{y}\to \modal(\id{x}{y})$ is an equivalence.
\end{enumerate}
\end{defn}

\noindent
One might think of eliminating into a $P:\modal A\to \UU_\modal$ directly rather than into $\modal\circ P$ for a $P:\modal A\to \UU$, but in that case we would be unable to show that $\modal A$ is a modal type (\cref{thm:modal-is-modal}).

\begin{defn}\label{defn:modunique}
A \define{uniquely eliminating modality} consists of
a modal operator and modal unit such that the function
\begin{equation*}
\lam{f} f\circ\modalunit[A] : (\prd{z:\modal A}\modal(P(z)))\to(\prd{x:A}\modal(P(\modalunit[A](x))))
\end{equation*}
is an equivalence for any $A$ and any $P:\modal A\to\UU$.
\end{defn}

\begin{defn}\label{defn:ssrs}
A \define{reflective subuniverse} is a family $\ismodal:\UU\to\prop$, together with a
modal operator and modal unit such that $\ismodal(\modal A)$ for
every $A:\UU$, and for every $B:\UU$ satisfying
$\ismodal(B)$, the function
\begin{equation*}
\lam{f} f\circ \modalunit[A]:(\modal A\to B)\to (A\to B)
\end{equation*}
is an equivalence.
A reflective subuniverse is \define{$\Sigma$-closed} if whenever $\ismodal(X)$ and $\ismodal(P(x))$ for all $x:X$, we have $\ismodal(\sm{x:X}P(x))$.
\end{defn}

Note that unlike \cref{defn:highermod,defn:modunique}, in \cref{defn:ssrs} the notion of ``modal type'' is part of the data.
However, we will show in \cref{lem:subuniv-modal} that $\ismodal(A)$ if and only if $\modalunit[A]$ is an equivalence.

\begin{defn}\label{defn:sofs}
An \define{orthogonal factorization system} consists of
predicates $\mathcal{L},\mathcal{R}:\prd*{A,B:\UU} (A\to B)\to\prop$ such that
\begin{enumerate}
\item $\mathcal{L}$ and $\mathcal{R}$ are closed under composition and contain all identities (i.e.\ they are subcategories of the category of types that contain all the objects), and
\item the type $\fact_{\mathcal{L},\mathcal{R}}(f)$ of factorizations
\begin{equation*}
\begin{tikzcd}
A \arrow[rr,"f"] \arrow[dr,swap,"f_{\mathcal{L}}"] & & B \\
& \im_{\mathcal{L},\mathcal{R}}(f) \arrow[ur,swap,"f_{\mathcal{R}}"]
\end{tikzcd}
\end{equation*}
of $f$, with $f_{\mathcal{L}}$ in $\mathcal{L}$ and $f_{\mathcal{R}}$ in $\mathcal{R}$, is contractible.
\end{enumerate}
More precisely, the type $\fact_{\mathcal{L},\mathcal{R}}(f)$ is defined to
be the type of
tuples
\begin{equation*}
(\im_{\mathcal{L},\mathcal{R}}(f),(f_{\mathcal{L}},p),(f_{\mathcal{R}},q),h)
\end{equation*}
consisting of a type $\im_{\mathcal{L},\mathcal{R}}(f)$, a function $f_{\mathcal{L}}:A\to \im_{\mathcal{L},\mathcal{R}}(f)$ with
$p:\mathcal{L}(f_{\mathcal{L}})$, a function $f_{\mathcal{R}}:\im_{\mathcal{L},\mathcal{R}}(f)\to B$ with $q:\mathcal{R}(f_{\mathcal{R}})$, and an identification $h:\id{f}{f_{\mathcal{R}}\circ f_{\mathcal{L}}}$. The type $\im_{\mathcal{L},\mathcal{R}}(f)$ is called
the \define{$(\mathcal{L},\mathcal{R})$-image of $f$}.

A type $X$ is said to be \define{$(\mathcal{L},\mathcal{R})$-modal} if
the map $!:X\to\unit$ is in $\mathcal{R}$ (and hence $!_\mathcal{L}$
is an equivalence).

An orthogonal factorization system is said to be \define{stable} if the class
$\mathcal{L}$ is stable under pullbacks (By
\autoref{lem:ofs_rightstable}, $\mathcal{R}$ is always stable under pullbacks).
\end{defn}

\begin{rmk}
  By univalence, the fact that $\mathcal{L}$ and $\mathcal{R}$ contain all identities implies that they each contain all equivalences.
  Conversely, if $f\in \mathcal{L}\cap\mathcal{R}$, then $(\idfunc,f)$ and $(f,\idfunc)$ are both $(\mathcal{L},\mathcal{R})$-factorizations of $f$, and hence equal; which implies that $f$ is an equivalence.
  Thus, $\mathcal{L}\cap\mathcal{R}$ consists exactly of the equivalences.
\end{rmk}

We now consider a few examples.
Since we will eventually prove all the definitions to be equivalent, we can use any one of them to describe any particular example.

\begin{eg}
  The prime example is the \textbf{$n$-truncation modality} $\truncf n$ as studied in~\cite[Chapter 7]{TheBook}, which we also denote $\truncmod{n}$.
  This can be given as a higher modality, using its induction principle and the fact that $\trunc n A$ is an $n$-type and the identity types of an $n$-type are again $n$-types (indeed, $(n-1)$-types).
  The corresponding stable orthogonal factorization system, consisting of $n$-connected and $n$-truncated maps, is also constructed in~\cite[Chapter 7]{TheBook}; our construction in \cref{thm:sofs_from_ssrs} will be a generalization of this.
\end{eg}

\begin{eg}\label{eg:open}
  Let $Q$ be a mere proposition.
  The \textbf{open modality} determined by $Q$ is defined by $\open Q A = (Q\to A)$, with unit $\modalunit[A](x) = \lam{\nameless}x : A \to (Q \to A)$.
  (We call it ``open'' because semantically, it generalizes the \emph{open subtopos} associated to a subterminal object of a topos, which in turn is so named because in the case of sheaves on a topological space $X$ it specializes to the open subspaces of $X$.)
  To show that this is a higher modality, suppose we have $P: (Q\to A) \to \UU$ and $f:\prd{a:A} Q \to P(\lam{\nameless} a)$.
  Then for any $z:Q\to A$ and $q:Q$ we have $f(z(q),q) : P(\lam{\nameless} z(q))$.
  And since $Q$ is a mere proposition, we have $z(q) = z(q')$ for any $q':Q$, hence $e(z,q) : (\lam{\nameless} z(q)) = z$ by function extensionality.
  This gives
  \[ \lam{z}{q} \trans{e(z,q)}{(f(z(q),q))} : \prd{z:Q\to A} Q \to P(z). \]
  For the computation rule, we have
  \begin{align*}
    (\lam{z}{q} \trans{e(z,q)}{(f(z(q),q))})(\lam{\nameless} a) &= \lam{q} \trans{e(\lam{\nameless} a,q)}{(f(a,q))}\\
    &= \lam{q} f(a,q) = f(a)
  \end{align*}
  by function extensionality, since $e(\lam{\nameless} a,q) = \refl{}$.
  Finally, if $x,y:Q\to A$, then $(x=y) \simeq \prd{q:Q} x(q) = y(q)$, and the map
  \[ \Big(\prd{q:Q} x(q) = y(q)\Big) \to \Big( Q \to \prd{q:Q} x(q) = y(q)\Big) \]
  is (by currying) essentially precomposition with a product projection $Q\times Q\to Q$, and that is an equivalence since $Q$ is a mere proposition.
\end{eg}

\begin{eg}\label{eg:closed}
  Again, let $Q$ be a mere proposition.
  The \textbf{closed modality} determined by $Q$ is defined by $\closed Q A = Q \ast A$, the \emph{join} of $Q$ and $A$ (the pushout of $Q$ and $A$ under $Q\times A$).
  (As for open modalities, closed modalities generalize closed subtoposes, which in turn generalize closed subspaces of topological spaces.)
  We show that this is a $\Sigma$-closed reflective subuniverse.
  Define a type $B$ to be modal if $Q \to \iscontr(B)$, and note that it is indeed the case that $Q\to\iscontr(Q\ast A)$, for any type $A$.
  By the universal property of pushouts, a map $Q \ast A \to B$ consists of a map $f:A\to B$ and a map $g:Q\to B$ and for any $a:A$ and $q:Q$ an identification $p:f(a)=g(q)$.
  But if $Q \to \iscontr(B)$, then $g$ and $p$ are uniquely determined, so this is just a map $A\to B$.
  Thus $(\closed Q A \to B) \to (A\to B)$ is an equivalence, so we have a reflective subuniverse.
  It is $\Sigma$-closed since the dependent sum of a contractible family of types over a contractible base is contractible.
\end{eg}

\begin{eg}\label{eg:dneg}
  The \textbf{double negation modality} is defined by $A\mapsto \neg\neg A$, i.e.\ $(A\to \emptyt)\to \emptyt$, with $\modalunit(a) = \lam{g} g(a)$.
  We show that this is a uniquely eliminating modality.
  Since the map $\lam{f}f\circ \modalunit[A]$ that must be an equivalence has mere propositions as domain and codomain, it suffices to give a map in the other direction.
  Thus, let $P: \neg\neg A \to \UU$ and $f:\prd{a:A} \neg \neg P(\lam{g} g(a))$; given $z:\neg\neg A$ we must derive a contradiction from $g:\neg P(z)$.
  Since we are proving a contradiction, we can strip the double negation from $z$ and assume given an $a:A$.
  And since $\neg\neg A$ is a mere proposition, we have $z = \lam{g} g(a)$, so that we can transport $f(a)$ to get an element of $\neg\neg P(z)$, contradicting $g$.
\end{eg}

\begin{eg}
  The \textbf{trivial modality} is the identity function on $\UU$.
  It coincides with $\open \top$ and with $\closed\bot$.

  Dually, the \textbf{zero modality} sends all types to $\unit$.
  It is equivalently the $(-2)$-truncation, and coincides with $\open\bot$ and with $\closed \top$.
\end{eg}

\paragraph*{Summary.}
In each of \autoref{defn:highermod,defn:modunique,defn:ssrs,defn:sofs}
we have defined what it means for a type to be modal. In each case, being
modal is a family of mere propositions indexed by the universe, i.e.~a subuniverse.
We will show in \autoref{thm:subuniv-highermod,thm:subuniv-modunique,thm:subuniverse-rs,thm:subuniv-sofs} that each kind of structure is completely determined by this subuniverse.
(\autoref{thm:subuniverse-rs} is more general, not requiring $\Sigma$-closedness.)

It follows that the type of all modalities of each
kind is a subset of the set $\UU\to\prop$ of all subuniverses, and in particular is a set.
This makes it easier to establish
the equivalences of the different kinds of modalities.
It suffices
to show that any modality of one kind determines a modality of the next kind
with the same modal types, which we will do as follows:
\begin{center}
\begin{tikzcd}
  & \text{higher modality} \ar[dr,bend left,"\text{\autoref{thm:modunique_from_highermod}}"] \\
  \parbox{3cm}{\centering stable factorization system} \ar[ur,bend left,"\text{\autoref{thm:highermod_from_sofs}}"] &&
  \parbox{3cm}{\centering uniquely eliminating modality} \ar[dl,bend left,"\text{\autoref{thm:ssrs_from_modunique}}"] \\
  & \parbox{3cm}{\centering $\Sigma$-closed reflective subuniverse} \ar[ul,bend left,"\text{\autoref{thm:sofs_from_ssrs}}"]
\end{tikzcd}
\end{center}
Before \autoref{thm:sofs_from_ssrs} we take the opportunity to develop a bit more theory of reflective subuniverses, including closure under identity types (\autoref{lem:rs_idstable}) and dependent products
(\autoref{lem:modal-Pi}), along with several equivalent characterizations of $\Sigma$-closedness (\autoref{thm:ssrs-characterize}).

Of these equivalences, the most surprising is that a stable factorization system is uniquely determined by its underlying reflective subuniverse of types.
This is false for stable factorization systems on arbitrary categories. However, an analogous fact is true in classical set-based mathematics for stable factorization systems on the category of sets (although in that case there are much fewer interesting examples). It is this fact about the category of sets which is analogous to the statement we prove \emph{in type theory} about factorization systems \emph{on the category of types}.

We will also see in \cref{sec:semantics} that when type theory is interpreted in a higher category, the data of a reflective subuniverse or modality has to be interpreted ``fiberwise'', giving a richer structure than a single reflective subcategory.

\subsection{Higher modalities}%
\label{sec:higher-modalities}

We start by showing that a higher modality is determined by its modal types, and gives rise to a uniquely eliminating modality.

\begin{lemma}\label{thm:modal-is-modal}
If $\modal$ is a higher modality, then any type of the form $\modal X$ is modal.
\end{lemma}

\begin{proof}
  We want to show that the modal unit $\modalunit[\modal X]:\modal X\to\modal\modal X$
is an equivalence. By the induction principle and the computation rule for
higher modalities, we find a function $f:\modal \modal X\to\modal X$ with
the property that $f\circ \modalunit[\modal X]\htpy\idfunc[\modal X]$. We wish to
show that we also have $\modalunit[\modal X]\circ f\htpy\idfunc$. Since identity
types of types of the form $\modal Y$ are declared to be modal, it is
equivalent to find a term of type
\begin{equation*}
\prd{z:\modal \modal X}\modal(\modalunit[\modal X](f(z))=z).
\end{equation*}
Now we are in the position to use the induction principle of higher modalites
again, so it suffices to show that $\modalunit(f(\modalunit(z)))=\modalunit(z)$
for any $z:\modal X$. This follows from the fact that $f\circ\modalunit=\idfunc$.
\end{proof}

\begin{thm}\label{thm:subuniv-highermod}
The data of two higher modalites $\modal$ and $\modal'$
are identical if and only if they have the same modal types.
\end{thm}

\begin{proof}
Another way of stating this is that the function from the type of \emph{all}
modalities on $\UU$ to the type $\UU\to\prop$ of predicates on $\UU$, given
by mapping a modality to the predicate $\ismodal$, is an embedding. Thus, we
need to show that for any predicate $\mathcal{M}:\UU\to\prop$, we can find at
most one modality for which $\mathcal{M}$ is the class of modal types.

To be precise, consider for any $\mathcal{M}:\UU\to\prop$ and $X:\UU$, the type
of tuples $(Y,p,\pi,I,C)$ such that
\begin{itemize}
\item $Y$ is a type.
\item $p:\mathcal{M}(Y)$.
\item $\pi:X\to Y$.
\item $I_P:(\prd{x:X} P(\pi(x)))\to(\prd{y:Y} P(y))$ for any $P:Y\to\UU_{\mathcal{M}}$.
\item $C$ witnesses that each $I_P$ is a right inverse of precomposing with $\pi$.
\end{itemize}
We will show that this type is a mere proposition.
First, we show that the
type of pairs $(I,C)$, with $I$ and $C$ of the indicated types, is a mere
proposition for any $(Y,p,\pi)$. After that, we show that the type of triples
$(Y,p,\pi)$ is also a mere proposition. These two facts combined prove the
statement.

Consider a type $Y$ satisfying $\mathcal{M}$, and a function $\pi:X\to Y$, and
let $(I,C)$ and $(I',C')$ be two terms witnessing that $Y$ satisfies an induction
principle with a computation rule. We want to show that $(I,C)=(I',C')$, and of
course it suffices to show that $(I(s),C(s))=(I'(s),C(s))$ for any
$P:Y\to\UU_{\mathcal{M}}$ and $s:\prd{x:X}P(\pi(x))$.

To show that $I(s,y)=I'(s,y)$ for any $y:Y$, we use
the induction principle $(I,C)$. So it suffices to show that
$I(s,\pi(x))=I'(s,\pi(x))$. Both of these terms are equal to $s(x)$. Thus,
we obtain a proof $J(s,y)$ that $I(s,y)=I'(s,y)$, with the property that
$J(s,\pi(x))=\ct{C(s,x)}{\inv{C'(s,x)}}$.
Now we need to show that $\trans{J(s)}{C(s)}=C'(s)$, which is equivalent
to the property we just stated. This finishes the proof that the type of
the induction principle and computation rule is a mere proposition.

It remains to show that $(Y,\pi)=(Y',\pi')$, provided that $Y$ and $Y'$ are both
in $\mathcal{M}$, and that both sides satisfy
the induction principle and computation rule. It suffices to find an equivalence
$f:Y\to Y'$ such that $f\circ \pi=\pi'$.

From the induction principles of $Y$ resp. $Y'$, we obtain a function
$f:Y\to Y'$ with the property that $f\circ \pi=\pi'$, and a function
$f':Y'\to Y$ with the property that $f'\circ \pi'=\pi$.
To show that $f'\circ f=\idfunc$ we use the induction principle
of $Y$. Since the type $f'(f(y))=y$ is in $\mathcal{M}$, it suffices to show that
$f'(f(\pi(y)))=\pi(y)$. This readily follows from the defining properties of $f$
and $f'$. Similarly, we have $f\circ f'=\idfunc$.
\end{proof}

\begin{thm}\label{thm:modunique_from_highermod}
A higher modality is a uniquely eliminating modality, with the
same modal types.
\end{thm}

\begin{proof}
Let $\modal$ be a modality with modal units $\modalunit[A]$. Our goal is to show
that the pre-composition map
\begin{equation*}
\lam{s}s\circ\modalunit[A]:(\prd{z:\modal A}\modal(P(z)))\to(\prd{a:A}\modal(P(\modalunit[A](a))))
\end{equation*}
is an equivalence for each $A:\UU$ and $P:\modal A\to\UU$.
By the given induction principle and computation rule, we obtain a
right inverse $\mathsf{ind}^\modal_A$ of $\blank\circ\modalunit[A]$.

To show that it is a left inverse, consider $s:\prd{z:\modal A}\modal(P(z))$.
We need to find a homotopy
\begin{equation*}
\prd{z:\modal A}\id{s(z)}{\mathsf{ind}^\modal_A(s\circ \modalunit_A)(z)}.
\end{equation*}
By assumption we have that $P(x)$ is
modal for each $z:\modal A$ and hence it follows that $\id{s(x)}{\mathsf{ind}^\modal_A(s\circ \modalunit_A)(x)}$
is modal for each $x$. Hence it suffices to find a function of type
\begin{equation*}
\prd{a:A}\id{s(\modalunit_A(a))}{\mathsf{ind}^\modal_A(s\circ \modalunit_A)(\modalunit_A(a))}.
\end{equation*}
This follows straight from the computation rule of higher modalities.
\end{proof}

\subsection{Uniquely eliminating modalities}%
\label{sec:uniq-elim}

Next, we show that a uniquely eliminating modality is determined by its modal types, and gives rise to a $\Sigma$-closed reflective subuniverse.

\begin{lemma}
Given a uniquely eliminating modality, $\modal X$ is modal for any type $X$.
\end{lemma}

\begin{proof}
Using the elimination principle of $\modal \modal X$, we find a function
$f:\modal \modal X\to\modal X$ and an identification $f\circ\modalunit[\modal X]=\idfunc[\modal X]$.
By uniqueness, the function
\[ (\modal \modal X\to\modal \modal X) \to (\modal X\to\modal \modal X) \]
is an equivalence, and hence its fiber over $\modalunit[\modal X]$:
\begin{equation*}
\sm{g:\modal \modal X\to\modal \modal X} g\circ\modalunit[\modal X]=\modalunit[\modal X]
\end{equation*}
is contractible. Since both $\idfunc[\modal \modal X]$ and $\modalunit[\modal X]\circ f$
are in this type (with suitable identifications), we find that $f$ is also the
right inverse of $\modalunit[\modal X]$. This shows that $\modalunit[\modal X]$ is an
equivalence, so $\modal X$ is modal.
\end{proof}

\begin{thm}\label{thm:subuniv-modunique}
The data of two uniquely eliminating modalities $\modal$ and $\modal'$ are equivalent if and only if both have the same modal types.
\end{thm}

\begin{proof}
We need to show that the type of uniquely eliminating modalities
with a given class $\mathcal{M}:\UU\to\prop$ of modal types
is a mere proposition. Since the types of the form $\modal X$ are modal,
it suffices to show that for any class $\mathcal{M}
:\UU\to\prop$ and any type $X$, the type of tuples $(Y,p,\pi,H)$ is a mere proposition, where:
\begin{itemize}
\item $Y:\UU$.
\item $p:\mathcal{M}(Y)$.
\item $\pi:X\to Y$.
\item For each $P$, $H_P$ witnesses that the function
  \begin{equation*}
    (\lam{s}s\circ \pi):(\prd{y:Y}\modal(P(y)))\to(\prd{x:X}\modal(P(\pi(x))))
  \end{equation*}
  is an equivalence.
\end{itemize}
Let $(Y,p,\pi,H)$ and $(Y',p',\pi',H')$ be such tuples. To show that they are
equal, it suffices to show that $(Y,\pi)=(Y',\pi')$ because the other things
in the list are terms of mere propositions. Furthermore, showing that
$(Y,\pi)=(Y',\pi')$ is equivalent to finding an equivalence $f:\eqv{Y}{Y'}$ with
the property that $f\circ\pi=\pi'$. By $H$, there is such a function, and by
$H'$ there is a function $f':Y'\to Y$ such that $f'\circ\pi'=\pi$. Now the
uniqueness gives that $f'\circ f$ is the only function from $Y$ to $Y$ such
that $f'\circ f\circ\pi=\pi$ and of course $\idfunc[Y]$ is another such function.
Therefore it follows that $f'\circ f=\idfunc$, and similarly it follows that
$f\circ f'=\idfunc$.
\end{proof}

\begin{thm}\label{thm:ssrs_from_modunique}
Any uniquely eliminating modality determines a $\Sigma$-closed reflective
subuniverse with the same modal types.
\end{thm}

\begin{proof}
It is immediate from the definition of uniquely eliminating modalities
that every map $f:A\to B$ into a modal type $B$ has a homotopy unique extension to $\modal A$
along the modal unit:
\begin{equation*}
\begin{tikzcd}
A \arrow[dr,"f"] \arrow[d,swap,"\modalunit_A"] \\ \modal A \arrow[r,densely dotted,swap,"\tilde f"] & B.
\end{tikzcd}
\end{equation*}
Since the types of the form $\modal X$ are modal, we obtain a reflective subuniverse.
It remains to verify  that the type $\sm{z:\modal X}\modal(P(z))$ is modal for
any type $X$ and $P:X\to\UU$. We have the function
\begin{equation*}
\varphi\defeq\lam{m}\pairr{f(m),g(m)}:\modal(\sm{z:\modal X}\modal(P(z)))\to\sm{z:\modal X}\modal(P(z)),
\end{equation*}
where
\begin{alignat*}{2}
f & \defeq \ind{\modal}(\lam{x}{u} x) && : \modal(\sm{z:\modal X}\modal(P(z)))\to \modal X \\
g & \defeq \ind{\modal}(\lam{x}{u} u) && : \prd{w:\modal(\sm{z:\modal X}\modal(P(z)))} \modal(P(f(w))).
\end{alignat*}
Our goal is to show that $\varphi$ is an inverse to the modal unit.

Note that
\begin{equation*}
\varphi(\modalunit(x,y)) \jdeq \pairr{f(\modalunit(x,y)),g(\modalunit(x,y))} \jdeq \pairr{x,y},
\end{equation*}
so we see immediately that $\varphi$ is a left inverse of $\modalunit$.

To show that $\varphi$ is a right inverse of $\modalunit$, note that the type
of functions $h$ fitting in a commuting triangle of the fom
\begin{equation*}
\begin{tikzcd}[column sep=-3em]
\modal(\sm{z:\modal X}\modal(P(z))) \arrow[rr,densely dotted,"h"] & & \modal(\sm{z:\modal X}\modal(P(z))) \\
\phantom{\modal(\sm{z:\modal X}\modal(P(z)))} & \sm{z:\modal X}\modal(P(z)) \arrow[ul,"\modalunit"] \arrow[ur,swap,"\modalunit"] & \phantom{\modal(\sm{z:\modal X}\modal(P(z)))}
\end{tikzcd}
\end{equation*}
is a fiber over $\modalunit$ of a precomposition equivalence, and hence contractible.
Since this type also contains the identity function, it suffices
to show that $(\modalunit\circ\varphi)\circ\modalunit=\modalunit$; but this follows
from the fact that $\varphi$ is a left inverse of the modal unit.
\end{proof}

\subsection{\texorpdfstring{$\Sigma$}{Σ}-closed reflective subuniverses}%
\label{sec:ssrs}

Now we study reflective subuniverses in a bit more detail, and end by
showing that $\Sigma$-closed ones give rise to stable factorization
systems. $\Sigma$-closure is used in \autoref{thm:sofs_from_ssrs} to
show that left maps and right maps are closed under composition.

\subsubsection{Properties of reflective subuniverses}%
\label{sec:prop-rfsu}

\begin{lemma}\label{lem:reflective_uniqueness}
  For any $\mathcal{M}:\UU\to\prop$ and any type $X$, the type of triples $(Y,f,I)$ consisting of
  \begin{itemize}
  \item $Y:\UU_{\mathcal{M}}$,
  \item $f:X\to Y$, and
  \item $I:\prd{Z:\UU_{\mathcal{M}}}\isequiv(\lam{g}g\circ f:(Y\to Z)\to(X\to Z))$
  \end{itemize}
  is a mere proposition.
\end{lemma}

\begin{proof}
Consider $(Y,f,I)$ and $(Y',f',I')$ of the described type. Since $I$ and $I'$
are terms of a mere proposition, it suffices to show that $(Y,f)=(Y',f')$. In
other words, we have to find an equivalence $g:Y\to Y'$ such that $g\circ f'=f$.

By $I(Y')$, the type of
pairs $(g,h)$ consisting of a function $g:Y\to Y'$ such that $h:g\circ f=f'$ is contractible. By
$I'(Y)$, the type of pairs $(g',h')$ consisting of a function $g':Y'\to Y$
such that $h':g'\circ f'=f$ is contractible.

Now $g'\circ g$ is a function such that $g'\circ g\circ f=g'\circ f'=f$, as
is $\idfunc[Y]$. By contractibility, it follows that $g'\circ g=\idfunc[Y]$.
Similarly, $g\circ g'=\idfunc[Y']$.
\end{proof}

\begin{thm}\label{thm:subuniverse-rs}
The data of any two reflective subuniverses with the same modal types are the same.
\end{thm}
\begin{proof}
  Given the modal types, the rest of the data of a reflective subuniverse consists of, for each type $X$, a triple $(Y,f,I)$ as in \cref{lem:reflective_uniqueness}.
  Thus, by \cref{lem:reflective_uniqueness}, these data form a mere proposition.
\end{proof}

\begin{lemma}\label{lem:subuniv-modal}
  Given a reflective subuniverse, a type $X$ is modal if and only if $\modalunit[X]$ is an equivalence.
\end{lemma}
\begin{proof}
  Certainly if $\modalunit[X]$ is an equivalence, then $X$ is modal since it is equivalent to the modal type $\modal X$.
  Conversely, if $X$ is modal then we have a triple $(X,\idfunc[X],\nameless)$ inhabiting the type from \cref{lem:reflective_uniqueness}, which also contains $(\modal X,\modalunit[X],\nameless)$.
  Since this type is a mere proposition, these two elements are equal; hence $\modalunit[X]$ is, like $\idfunc[X]$, an equivalence.
\end{proof}

\begin{lemma}\label{thm:modalunit-retract-equiv}
  Given a reflective subuniverse, if a modal unit $\modalunit[X]$ has a left inverse (i.e.\ a retraction), then it is an equivalence, and hence $X$ is modal.
\end{lemma}
\begin{proof}
  Suppose $f$ is a left inverse of $\modalunit[X]$, i.e.\ $f\circ \modalunit[X] = \idfunc[X]$.
  Then $\modalunit[X]\circ f\circ \modalunit[X] = \modalunit[X]$, so $\modalunit[X]\circ f$ is a factorization of $\modalunit[X]$ through itself.
  By uniqueness of such factorizations, $\modalunit[X]\circ f = \idfunc[\modal X]$.
  Thus $f$ is also a right inverse of $\modalunit[X]$, hence $\modalunit[X]$ is an equivalence.
\end{proof}

In the following lemma we show that any reflective subuniverse is a `a functor up to homotopy', i.e.~that the localization operation has an action on morphisms which preserves composition and identities.

\begin{lemma}
  Given $f:A\to B$ we have an induced map $\modal f : \modal A \to \modal B$, preserving identities and composition up to homotopy.
  Moreover, for any $f$ the naturality square
  \begin{equation*}
    \begin{tikzcd}
      A \arrow[r,"f"] \arrow[d,swap,"\modalunit"] & B \arrow[d,"\modalunit"] \\
      \modal A \arrow[r,swap,"\modal f"] & \modal B
    \end{tikzcd}
  \end{equation*}
  commutes.
\end{lemma}

\begin{proof}
  Define $\modal f$ to be the unique function such that $\modal f \circ \modalunit[A] = \modalunit[B] \circ f$, using the universal property of $\modalunit[A]$.
  The rest is easy to check using further universal properties.
\end{proof}

\begin{lemma}
  Given a reflective subuniverse and any type $X$, the map $\modal \modalunit[X] : \modal X \to \modal\modal X$ is an equivalence.
\end{lemma}
\begin{proof}
  By naturality, we have $\modal \modalunit[X] \circ \modalunit[X] = \modalunit[\modal X] \circ \modalunit[X]$.
  Hence $\modal \modalunit[X] = \modalunit[\modal X]$ by the universal property of $\modalunit[X]$, but $\modalunit[\modal X]$ is an equivalence by \cref{lem:subuniv-modal}.
\end{proof}

\begin{lemma}\label{thm:rsu-galois}
  Given a reflective subuniverse, a type $X$ is modal if and only if $(\blank \circ f) : (B\to X) \to (A\to X)$ is an equivalence for any function $f:A\to B$ such that $\modal f$ is an equivalence.
\end{lemma}
\begin{proof}
  If $\modal f$ is an equivalence and $X$ is modal, then by the universal property of $\modalunit$, we have a commutative square
  \[
  \begin{tikzcd}
    (B\to X) \ar[r,"\blank\circ f"] & (A\to X) \\
    (\modal B\to X) \ar[r,"\blank\circ\modal f"'] \ar[u,"{\blank\circ \modalunit[B]}"] &
    (\modal A \to X) \ar[u,"{\blank\circ \modalunit[A]}"']
  \end{tikzcd}
  \]
  in which all but the top map are equivalences; thus so is the top map.

  Conversely, since $\modal\modalunit[X]$ is an equivalence, the hypothesis implies that
  $(\blank \circ \modalunit[X]) : (\modal X\to X) \to (X\to X)$
  is an equivalence.
  In particular, its fiber over $\idfunc[X]$ is inhabited, i.e.\ $\modalunit[X]$ has a retraction; hence $X$ is modal.
\end{proof}

\begin{lemma}\label{lem:sum_idempotent}
Consider a reflective subuniverse with modal operator $\modal$, and let $P:X\to\UU$ for some type $X:\UU$.
Then the unique map for which the triangle
\begin{equation*}
\begin{tikzcd}
\sm{x:X}P(x) \arrow[d,swap,"\modalunit"] \arrow[dr,"{\lam{\pairr{x,y}}\modalunit(x,\modalunit(y))}"] \\
\modal(\sm{x:X}P(x)) \arrow[r,densely dotted] & \modal(\sm{x:X}\modal(P(x)))
\end{tikzcd}
\end{equation*}
commutes, is an equivalence.
\end{lemma}
\begin{proof}
  Since both codomains are modal, it suffices to show that ${\lam{\pairr{x,y}}\modalunit(x,\modalunit(y))}$ has the universal property of $\modalunit[\sm{x:X}P(x)]$, i.e.\ that any map $(\sm{x:X}P(x)) \to Y$, where $Y$ is modal, extends uniquely to $\modal(\sm{x:X}\modal(P(x)))$.
  But we have
  \begin{align*}
    ((\sm{x:X}P(x)) \to Y)
    &\simeq
    \prd{x:X} P(x) \to Y\\
    &\simeq
    \prd{x:X} \modal(P(x)) \to Y\\
    &\simeq
    (\sm{x:X}\modal(P(x))) \to Y\\
    &\simeq
    \modal (\sm{x:X}\modal(P(x))) \to Y
  \end{align*}
  and it is easy to see that this is the desired precomposition map.
\end{proof}

\begin{lemma}\label{lem:rs_idstable}
  For any reflective subuniverse, if $X$ is modal, then so is the identity type $x=y$ for any $x,y:X$.
\end{lemma}

\begin{proof}
Let $X$ be a modal type, and let $x,y:X$. We have a map
$\modal(x=y)\to\unit$. The outer square in the diagram
\begin{equation*}
\begin{tikzcd}
\modal(x=y) \arrow[ddr,bend right=15] \arrow[drr,bend left=15] \\
& (x=y) \arrow[r] \arrow[d] \arrow[ul,"\modalunit"] \arrow[dr, phantom, "\lrcorner", very near start] & \unit \arrow[d,"x"] \\
& \unit \arrow[r,swap,"y"] & X
\end{tikzcd}
\end{equation*}
commutes, because both maps extend the map $(x=y)\to X$ along $\modalunit$, and
such extensions are unique because $X$ is assumed to be modal.
Hence the universal property of the pullback gives
a left inverse of $\modalunit:(x=y)\to\modal(x=y)$, so by \cref{thm:modalunit-retract-equiv} $(x=y)$ is modal.
\end{proof}

\begin{lemma}\label{lem:modal-Pi}
Given a reflective subuniverse,
if $P(x)$ is modal for all $x:X$, then so is $\prd{x:X}P(x)$.
\end{lemma}

\begin{proof}
By \cref{thm:modalunit-retract-equiv}, it suffices to define a left inverse of the modal unit
$\modalunit:(\prd{x:A}P(x))\to \modal(\prd{x:A}P(x))$. By the universal property
of dependent product, extending
\begin{equation*}
\begin{tikzcd}
\prd{x:A}P(x) \arrow[r,"{\idfunc}"] \arrow[d,"\modalunit"] & \prd{a:A}P(a) \arrow[d,"{\psi\,\defeq\,\lam{f}{a}\modalunit[P(a)](f(a))}"] \\
\modal(\prd{x:A}P(x)) \arrow[r,densely dotted] & \prd{a:A}\modal(P(a))
\end{tikzcd}
\end{equation*}
is equivalent to extending
\begin{equation*}
\begin{tikzcd}[column sep=large]
\prd{x:A}P(x) \arrow[r,"{\mathsf{ev}_a}"] \arrow[d,swap,"{\modalunit}"]
& P(a) \arrow[d,"{\modalunit}"] \\
\modal(\prd{x:A}P(x)) \arrow[r,densely dotted,swap,"{\modal(\mathsf{ev}_a)}"] & \modal(P(a))
\end{tikzcd}
\end{equation*}
for any $a:A$. Thus, we find
\begin{equation*}
f\defeq\lam{m}{a}\modal(\mathsf{ev}_a)(m):\modal(\prd{x:A}P(x))\to\prd{a:A}P(a)
\end{equation*}
as the solution to the first extension problem. In the first extension problem,
the function $\psi$ is an equivalence by the assumption that each $P(a)$ is
modal, so we obtain a retraction of the modal unit.
\end{proof}

Taking $X=\unit+\unit$, so that $P:X\to\UU$ is just a pair of types, we conclude that if $A$ and $B$ are modal then so is $A\times B$.
Moreover, we have:

\begin{lemma}\label{thm:modal-pres-prod}
Given any reflective subuniverse, the modal operator $\modal$ preserves finite cartesian products (including the unit type).
\end{lemma}

\begin{proof}
  In the nullary case, the statement is that the unit type $\unit$ is modal, which follows directly from \cref{thm:modalunit-retract-equiv}.
  In the binary case, we have to show that the modal extension
\begin{equation*}
\begin{tikzcd}
X\times Y \arrow[d,swap,"{\modalunit[X\times Y]}"] \arrow[dr,"\lam{\pairr{x,y}}\pairr{\modalunit[X](x),\modalunit[Y](y)}"] \\
\modal(X\times Y) \arrow[r,densely dotted] & \modal X\times\modal Y
\end{tikzcd}
\end{equation*}
is an equivalence.
But $(\modal(X\times Y),\modalunit[X\times Y],\nameless)$ inhabits the type from \cref{lem:reflective_uniqueness}, so if we can show that $(\modal X\times \modal Y,\lam{\pairr{x,y}}\pairr{\modalunit[X](x),\modalunit[Y](y)})$ also extends to an inhabitant of that type, then they will be equal, inducing an equivalence that by uniqueness must be the map above.
To show this, first note that $\modal X\times \modal Y$ is modal, as remarked above.
And for any modal type $Z$ we have
\begin{align*}
  (X\times Y \to Z)
  &\eqvsym X\to (Y\to Z)\\
  &\eqvsym X\to (\modal Y\to Z)\\
  &\eqvsym \modal X\to (\modal Y\to Z)\\
  &\eqvsym \modal X\times \modal Y\to Z
\end{align*}
given by precomposition as desired.
Here in the penultimate step we use the fact that the function type $\modal Y\to Z$ is modal since $Z$ is, by \cref{lem:modal-Pi}.
\end{proof}

\begin{lemma}\label{lem:modal-pres-prop}
Given any reflective subuniverse, the modal operator preserves mere propositions.
\end{lemma}
\begin{proof}
  A type $P$ is a mere proposition if and only if the diagonal $P\to P\times P$ is an equivalence.
  The result then follows from \cref{thm:modal-pres-prod}.
\end{proof}

By contrast, even modalities do not generally preserve $n$-types for any $n\ge 0$.
For instance, the ``shape'' modality of~\cite{shulman:bfp-realcohesion} takes the topological circle, which is a 0-type, to the homotopical circle, which is a 1-type, and the topological 2-sphere, which is also a 0-type, to the homotopical 2-sphere, which is (conjecturally) not an $n$-type for any finite $n$.
However, we will see in~\autoref{modaln-truncated} that lex modalities do preserve $n$-types for all $n$.

\begin{rmk}
  The basic properties of types and maps in homotopy type theory, such as being contractible, being a proposition, being an $n$-type, being an equivalence, and so on, are all constructed (perhaps inductively) out of identity types and $\Sigma$- and $\Pi$-types.
  Thus, a $\Sigma$-closed reflective subuniverse is closed under them as well.
  That is, if $A$ and $B$ are modal and $f:A\to B$, then the propositions ``$A$ is contractible'', ``$A$ is an $n$-type'', ``$f$ is an equivalence'', and so on, are all modal as well.
\end{rmk}

\subsubsection{$\Sigma$-closed reflective subuniverses}%
\label{sec:sigma-closed}

\begin{defn}\label{defn:connected}
Let $\mathcal{M}:\UU\to\prop$ be a reflective subuniverse with modal
operator $\modal$. We say
that a type $X$ is \define{$\modal$-connected} if $\modal X$ is contractible,
and we say that a function $f:X\to Y$ is \define{$\modal$-connected} if each
of its fibers is. Similarly, we say that $f$ is \define{modal} if each of its
fibers is.
\end{defn}

Note that a type $X$ is modal or $\modal$-connected just when the map $X\to\unit$ is.

\begin{eg}\label{eg:closed-connected}
  Recall from \cref{eg:open} that the open modality associated to a proposition $Q$ is defined by $\open Q(A) \defeq (Q\to A)$.
  We claim that $A$ is $\open Q$-connected if and only if $Q \to\iscontr(A)$.
  In other words, $(Q \to\iscontr(A))\eqvsym \iscontr(Q\to A)$.
  For on the one hand, if $Q\to \iscontr(A)$, then $Q\to A$; while any two $f,g:Q\to A$ can be shown equal by function extensionality, since if $Q$ then $A$ is contractible.
  But on the other hand, if $\iscontr(Q\to A)$ and $Q$, then $\eqv{(Q\to A)}{A}$, hence $\iscontr(A)$.

  Note that $Q \to\iscontr(A)$ is also the defining condition for the $\closed Q$-modal types from \cref{eg:closed}.
  That is, the $\open Q$-connected types coincide with the $\closed Q$-modal types.
  We will come back to this relationship in \cref{eg:artin}.
\end{eg}

The following theorem combines Lemma 7.5.7 and Theorem 7.7.4 of~\cite{TheBook}.

\begin{thm}\label{thm:ssrs-characterize}
Given a reflective universe with modal operator $\modal$,
the following are equivalent:
\begin{enumerate}
\item It is $\Sigma$-closed.\label{item:sc1}
\item It is uniquely eliminating.\label{item:sc2}
\item The modal units are $\modal$-connected.\label{item:sc3}
\end{enumerate}
\end{thm}

\begin{proof}
  To show~\ref{item:sc1}$\Leftrightarrow$\ref{item:sc2}, let $Y$ be modal and $P:Y \to UU_\modal$, and consider for any $X$ the following commuting square:
  \begin{equation*}
    \begin{tikzcd}
      \Big(\modal X \to \sm{y:Y}P(y)\Big) \arrow[r] \arrow[d] & \Big(X \to \sm{y:Y}P(y)\Big) \arrow[d] \\
      \sm{g:\modal X\to Y}\prd{z:\modal X}P(g(z)) \arrow[r] & \sm{f:X\to Y}\prd{x:X}P(f(x))
    \end{tikzcd}
  \end{equation*}
  The vertical maps are equivalences, so for any $X,Y,P$ the top map is an equivalence if and only if the bottom is.

  If~\ref{item:sc1} holds, the top map is an equivalence for all $X,Y,P$.
  But the converse is also true, since we can take $X \defeq \sm{y:Y}P(y)$ to obtain a retraction for its unit.

  The bottom map is induced by the map $(\modal X\to Y) \to (X\to Y)$, which is an equivalence since $Y$ is modal, and the family of maps
  \[\Big(\prd{z:\modal X} P(g(z))\Big) \to \Big(\prd{x:X} P(g(\modalunit[X](x)))\Big) \]
  for all $g:\modal X\to Y$; thus it is an equivalence just when each of these maps is.
  If~\ref{item:sc2} holds, then this is true for all $X,Y,P,g$.
  But the converse is also true, since we can take $Y \defeq \modal X$ and $g\defeq \idfunc[\modal X]$.
  This completes the proof of~\ref{item:sc1}$\Leftrightarrow$\ref{item:sc2}.

  To show~\ref{item:sc2}$\Rightarrow$\ref{item:sc3}, we want a term of type
  \begin{equation*}
\prd{z:\modal X}\iscontr(\modal(\hfib{\modalunit}{z})).
\end{equation*}
Using the dependent eliminators, it is easy to find a term
$s:\prd{z:\modal X}\modal(\hfib{\modalunit}{z})$ with the property that
$s\circ\modalunit(x)=\modalunit(x,\refl{\modalunit(x)})$. Now we need to show
that
\begin{equation*}
\prd{z:\modal X}{w:\modal(\hfib{\modalunit}{z})}w=s(z).
\end{equation*}
Since the type $w=s(z)$ is modal, this is equivalent to
\begin{equation*}
\prd{z:\modal X}{x:X}{p:\modalunit(x)=z} \modalunit(x,p)=s(z).
\end{equation*}
Moreover, the type $\sm{z:\modal X}\modalunit(x)=z$ is contractible, so this
is equivalent to
\begin{equation*}
\prd{x:X} \modalunit(x,\refl{\modalunit(x)})=s(\modalunit(x)),
\end{equation*}
of which we have a term by the defining property of $s$.

Finally, to show~\ref{item:sc3}$\Rightarrow$\ref{item:sc2} we show that for \emph{any} $\modal$-connected map $f:X\to Y$ and any family $P:Y \to \UU_\modal$ of modal types of $Y$, the precomposition map
\begin{equation*}
  \Big(\prd{y:Y}P(y)\Big)\to \Big(\prd{x:X}P(f(x))\Big)
\end{equation*}
is an equivalence. This is because we have a commuting square
\begin{equation*}
  \begin{tikzcd}
    \prd{y:Y}\Big(\modal(\hfib{f}{y})\to P(y)\Big) \arrow[r] \arrow[d] & \prd{y:Y}\Big(\hfib{f}{y}\to P(y)\Big) \arrow[d] \\
    \prd{y:Y}P(y) \arrow[r] & \prd{x:X}P(f(x))
  \end{tikzcd}
\end{equation*}
In this square the map on the left is an equivalence by the contractibility of $\modal(\hfib{f}{y})$; the map on the right is an equivalence by the dependent universal property of identity types; and the top map is an equivalence by the universal property of modalities. Therefore the bottom map is an equivalence.
\end{proof}

\begin{lemma}\label{thm:rsu-compose-cancel}
  Given $f:A\to B$ and $g:B\to C$ and a reflective subuniverse $\modal$, if $f$ is $\modal$-connected, then $g$ is $\modal$-connected if and only if $g\circ f$ is $\modal$-connected.
  That is, $\modal$-connected maps are closed under composition and right cancellable.
\end{lemma}
\begin{proof}
  Recall that for $f:X\to Y$ and $g:Y\to Z$, one has $\hfib{g\circ f}{z}=\sm{p:\hfib{g}{z}}\hfib{f}{\proj1(p)}$.
  Thus, for any $z:C$ we have
  \begin{align*}
    \modal(\hfib{g\circ f}{z})
    & \eqvsym
      \modal(\sm{p:\hfib{g}{z}}\hfib{f}{\proj1(p)}) \\
    & \eqvsym
      \modal(\sm{p:\hfib{g}{z}}\modal(\hfib{f}{\proj1(p)}))
    \qquad \text{(by \cref{lem:sum_idempotent})}\\
    & \eqvsym
      \modal(\sm{p:\hfib{g}{z}}\unit) \\
    & \eqvsym
      \modal\hfib{g}{z}
  \end{align*}
  using the fact that $f$ is $\modal$-connected.
  Thus, one is contractible if and only if the other is.
\end{proof}

In general it is not true that if $g$ and $g\circ f$ are $\modal$-connected then $f$ is; this is one of the equivalent characterizations of lex modalities (\cref{thm:lex-modalities}).

\begin{thm}\label{thm:sofs_from_ssrs}
A $\Sigma$-closed reflective subuniverse determines a stable orthogonal factorization system with the same
modal types.
\end{thm}

\begin{proof}
Define $\mathcal{L}$ to be the class of $\modal$-connected
maps and $\mathcal{R}$ to be the the class of modal maps.
We first show that both $\mathcal{L}$ and $\mathcal{R}$ are closed under
composition.
Since $\hfib{g\circ f}{z}=\sm{p:\hfib{g}{z}}\hfib{f}{\proj1(p)}$, by $\Sigma$-closedness if $f$ and $g$ are both in $\mathcal{R}$ then so is $g\circ f$.
Thus $\cR$ is closed under composition; while \cref{thm:rsu-compose-cancel} implies that $\cL$ is closed under composition.
And since the fibers of an identity map are contractible, and contractible types are both modal and $\modal$-connected, both $\mathcal{L}$ and $\mathcal{R}$ contain all identities.


To obtain a factorization system,
it remains to show that the type of
$(\mathcal{L},\mathcal{R})$-factorizations of any function $f:X\to Y$ is contractible.
Since \[\pairr{X,f}=_{(\sm{Z:\UU} Z\to Y)} \pairr{\sm{y:Y}\hfib{f}{y},\proj1},\] it is sufficient to
show that $\fact_{\mathcal{L},\mathcal{R}}(\proj1)$ is contractible for any
$\proj1:\sm{y:Y}P(y)\to Y$. But $\proj1$ factors as
\begin{equation*}
\begin{tikzcd}
\sm{y:Y}P(y) \arrow[r,"p_\mathcal{L}"] & \sm{y:Y}\modal(P(y)) \arrow[r,"p_\mathcal{R}"] & Y
\end{tikzcd}
\end{equation*}
where $p_\mathcal{L}\defeq\total{\modalunit[P(\blank)]}$ and $p_\mathcal{R}\defeq\proj1$.
The fibers of $p_\mathcal{R}$ are $\modal(P(\blank))$, so it follows
immediately that $p_\mathcal{R}$ is in $\mathcal{R}$.
Moreover, since
$\eqv{\hfib{\total{\modalunit}}{\pairr{y,u}}}{\hfib{\modalunit[P(y)]}{u}}$ and each $\modalunit$ is $\modal$-connected, it follows that $p_\mathcal{L}$ is in
$\mathcal{L}$.

Now consider any other factorization $(I,g,h,H)$ of $\proj1$ into
an $\cL$-map $g:(\sm{y:Y}P(y))\to I$ followed by an $\cR$-map $h:I\to Y$. Since
$I=\sm{y:Y}\hfib{h}{y}$, we have a commuting square
\begin{equation*}
\begin{tikzcd}
\sm{y:Y}P(y) \arrow[r,"g"] \arrow[d,swap,"{\total{\gamma}}"]
  & I \arrow[d,"h"] \\
\sm{y:Y}\hfib{h}{y} \arrow[ur,equals] \arrow[r,swap,"\proj1"] & Y
\end{tikzcd}
\end{equation*}
in which $\gamma(y,u)\defeq \pairr{g(y,u),H(y,u)}$.
It follows that
\[(I,g,h,H)=\left(\tsm{y:Y}\hfib{h}{y},\total{\gamma},\proj1,\nameless\right).\]
Thus it suffices to show that there is a commuting triangle
\begin{equation*}
\begin{tikzcd}[column sep=0]
\phantom{\hfib{h}{y}} & P(y) \arrow[dl,swap,"\modalunit"] \arrow[dr,"{\gamma_y}"] & \phantom{\modal(P(y))} \\
\modal(P(y)) \arrow[rr,equals] & & \hfib{h}{y}
\end{tikzcd}
\end{equation*}
for all $y:Y$.
We will do this using \cref{lem:reflective_uniqueness}, by showing that $\gamma_y$ has the same universal property as $\modalunit[P(y)]$.
This follows from the following calculation:
\begin{align*}
(\hfib{h}{y}\to Z) & \eqvsym ((\sm{w:\hfib{h}{y}}\modal(\hfib{g}{\proj1(w)}))\to Z) \\
& \eqvsym ((\sm{w:\hfib{h}{y}}\hfib{g}{\proj1(w)})\to Z) \\
& \eqvsym (\hfib{h\circ g}{y}\to Z) \\
& \eqvsym (P(y)\to Z),
\end{align*}
which we can verify is given by precomposition with $\gamma_y$.

It remains to show that our orthogonal factorization system is stable. Consider a pullback diagram
\begin{equation*}
\begin{tikzcd}
A' \arrow[d,swap,"k"] \arrow[r,"f"] & A \arrow[d,"l"] \\
B' \arrow[r,swap,"g"] & B
\end{tikzcd}
\end{equation*}
in which $l$ is in $\mathcal{L}$. By the pasting lemma for pullbacks, it
follows that $\hfib{k}{b}=\hfib{l}{g(b)}$ for each $b:B'$. Thus, it follows that
$k$ is in $\mathcal{L}$.
\end{proof}

\subsubsection{Connected maps}%
\label{sec:connected-maps}

The $\modal$-connected maps introduced in \cref{defn:connected} have a number of other useful properties.
Most of these are stated in~\cite[\S7.5]{TheBook} for the special case of the $n$-truncation modality, but essentially the same proofs work for any modality.

In fact, most of these properties are true about an arbitrary reflective subuniverse, although a few of the proofs must be different.
Thus, for this subsection, let $\modal$ be a reflective subuniverse, not in general $\Sigma$-closed.

\begin{lemma}%
\label{lem:connected-map-equiv-truncation}
If $f : A \to B$ is $\modal$-connected, then it induces an equivalence
$\modal f : \eqv{\modal{A}}{\modal{B}}$.
\end{lemma}
\begin{proof}
  To define an inverse $g:\modal B \to \modal A$, by the universal property of $\modal B$, it suffices to define a map $B\to \modal A$.
  But given $b:B$, we have a map $\proj1 : \hfib{f}{b} \to A$, hence $\modal\proj1 : \modal \hfib{f}{b} \to \modal A$.
  And $\modal \hfib{f}{b}$ is contractible since $f$ is $\modal$-connected, so it has a point $c_b$, and we define $g(\modalunit[B](b)) = \modal \proj1(c_b)$.

  Now by the universal property of $\modal A$ and $\modal B$, it suffices to show that the composites $g\circ \modal f \circ \modalunit[A]$ and $\modal f\circ g \circ \modalunit[B]$ are equal to $\modalunit[A]$ and $\modalunit[B]$ respectively.
  In the first case, for $a:A$ we have
  \begin{align*}
    g(\modal f(\modalunit[A](a)))
    &= g(\modalunit[B](f(a)))\\
    &= \modal \proj1(c_{f(a)})\\
    &= \modal \proj1(\modalunit[\hfib f b](a,\refl{f(a)}))\\
    &= \modalunit[A](\proj1(a,\refl{f(a)}))\\
    &= \modalunit[A](a),
  \end{align*}
  using in the third line the fact that $\modal(\hfib f b)$ is contractible.
  And in the second case, for $b:B$ we have
  \begin{align*}
    \modal f(g(\modalunit[B](b)))
    &= \modal f(\modal \proj1(c_b))\\
    &= \modal(f\circ \proj1)(c_b)\\
    &= \modal(\lam{u:\hfib f b} b)(c_b)\\
    &= \modal(\lam{u:\unit} b)(\modalunit[\unit](\ttt))\\
    &= \modalunit[B](b)
  \end{align*}
  where in the last two lines we use the commutativity of the following diagram:
  \[
    \begin{tikzcd}
      \hfib f b \ar[d] \ar[r] \ar[rr,bend left,"{\lam{u:\hfib f b} b}"] & \unit \ar[r,"b"] \ar[d,"{\modalunit[\unit]}"] \ar[dl,"{c_b}"] & B \ar[d,"{\modalunit[B]}"] \\
      \modal(\hfib f b) \ar[r] \ar[rr,bend right,"{\modal (\lam{u:\hfib f b} b)}"'] & \modal \unit \ar[r] & \modal B
    \end{tikzcd}
  \]
  and the fact that $\modal\unit$ is contractible.
\end{proof}

The converse of \cref{lem:connected-map-equiv-truncation} is false in general, even for modalities; we will see in \cref{thm:lex-modalities} that it holds exactly when $\modal$ is lex.

Recall that $\modaltype$ denotes the universe of modal types.
Note that the projection $\proj1 : (\sm{x:A} P(x)) \to A$ is $\modal$-modal if and only if $P$ factors through $\modaltype$.
The following generalizes the unique elimination property of $\modalunit$ to arbitrary $\modal$-connected maps.

\begin{lemma}\label{prop:nconnected_tested_by_lv_n_dependent types}
For $f:A\to B$ and $P:B\to\modaltype$, consider the following function:
\begin{equation*}
\lam{s} s\circ f :\Parens{\prd{b:B} P(b)}\to\Parens{\prd{a:A}P(f(a))}.
\end{equation*}
For a fixed $f$, the following are equivalent.
\begin{enumerate}
\item $f$ is $\modal$-connected.\label{item:conntest1}
\item For every $P:B\to \modaltype$, the map $\lam{s} s\circ f$ is an equivalence.\label{item:conntest2}
\item For every $P:B\to \modaltype$, the map $\lam{s} s\circ f$ has a section.\label{item:conntest3}
\end{enumerate}
\end{lemma}

\begin{proof}
First suppose $f$ is $\modal$-connected and let $P:B\to\modaltype$. Then:
\begin{align*}
  \prd{b:B} P(b) & \eqvsym \prd{b:B} \Parens{\modal{\hfib{f}b} \to P(b)}
  \tag{since $\modal{\hfib{f}b}$ is contractible}\\
  & \eqvsym \prd{b:B} \Parens{\hfib{f}b\to P(b)}
  \tag{since $P(b)$ is modal}\\ 
  & \eqvsym \prd{b:B}{a:A}{p:f(a)= b} P(b)\\
  & \eqvsym \prd{a:A} P(f(a))
\end{align*}
and the composite equivalence is indeed composition with $f$.
Thus, \ref{item:conntest1}$\Rightarrow$\ref{item:conntest2}, and clearly \ref{item:conntest2}$\Rightarrow$\ref{item:conntest3}. 
To show \ref{item:conntest3}$\Rightarrow$\ref{item:conntest1}, let 
$P(b)\defeq \modal{\hfib{f}b}$.
Then~\ref{item:conntest3} yields a map $c:\prd{b:B} \modal{\hfib{f}b}$ with
$c(f(a))=\modalunit{\pairr{a,\refl{f(a)}}}$. To show that each $\modal{\hfib{f}b}$ is contractible, we will show that $c(b)=w$ for any $b:B$ and $w:\modal{\hfib{f}b}$.
In other words, we must show that the identity function $\modal{\hfib{f}b} \to \modal{\hfib{f}b}$ is equal to the constant function at $c(b)$.
By the universal property of $\modal{\hfib{f}b}$, it suffices to show that they become equal when precomposed with $\modalunit[\hfib{f}b]$, i.e.\ we may assume that $w = \modalunit\pairr{a,p}$ for some $a:A$ and $p:f(a)=b$.
But now path induction on $p$ reduces our goal to the given $c(f(a))=\modalunit{\pairr{a,\refl{f(a)}}}$.
\end{proof}

\begin{corollary}\label{thm:nconn-to-ntype-const}\label{connectedtotruncated}
A type $A$ is $\modal$-connected if and only if the ``constant functions'' map
$
  B \to (A\to B)
$
is an equivalence for every modal type $B$.\qed%
\end{corollary}

Dually, we will prove in \cref{thm:detect-right-by-fibers} that when $\modal$ is a modality, if this holds for all $\modal$-connected $A$ then $B$ is modal.

\begin{lemma}\label{lem:nconnected_to_leveln_to_equiv}
Let $B$ be a modal type and let $f:A\to B$ be a function.
If $f$ is $\modal$-connected, then the induced function $g:\modal A\to B$ is an equivalence; the converse holds if $\modal$ is $\Sigma$-closed.
\end{lemma}

\begin{proof}
By \cref{lem:connected-map-equiv-truncation}, if $f$ is $\modal$-connected then $\modal f$ is an equivalence.
But $g$ is the composite ${\modalunit[B]}^{-1}\circ \modal f$, hence also an equivalence.

Conversely, by \cref{thm:ssrs-characterize}, $\modalunit$ is $\modal$-connected.
Thus, since $f = g\circ \modalunit[A]$, if $g$ is an equivalence then $f$ is also $\modal$-connected.
\end{proof}

\begin{lemma}\label{lem:nconnected_postcomp_variation}
Let $f:A\to B$ be a function and $P:A\to\type$ and $Q:B\to\type$ be type families. Suppose that $g:\prd{a:A} P(a)\to Q(f(a))$
is a family of $\modal$-connected functions.
If $f$ is also $\modal$-connected, then so is the function
\begin{align*}
\varphi &:\Parens{\sm{a:A} P(a)}\to\Parens{\sm{b:B} Q(b)}\\
\varphi(a,u) &\defeq \pairr{f(a),g_a(u)}.
\end{align*}
Conversely, if $\varphi$ and each $g_a$ are $\modal$-connected, and moreover $Q$ is fiberwise merely inhabited (i.e.\ we have $\brck{Q(b)}$ for all $b:B$), then $f$ is $\modal$-connected.
\end{lemma}

\begin{proof}
For any $b:B$ and $v:Q(b)$ we have
{\allowdisplaybreaks%
\begin{align*}
\modal{\hfib{\varphi}{\pairr{b,v}}} & \eqvsym \modal{\sm{a:A}{u:P(a)}{p:f(a)= b} \trans{p}{g_a(u)}= v}\\
& \eqvsym \modal{\sm{w:\hfib{f}b}{u:P(\proj1(w))} g_{\proj 1 w}(u)= \trans{\opp{\proj2(w)}}{v}}\\
& \eqvsym \modal{\sm{w:\hfib{f}b} \hfib{g(\proj1 w)}{\trans{\opp{\proj 2(w)}}{v}}}\\
& \eqvsym \modal{\sm{w:\hfib{f}b} \modal{\hfib{g(\proj1 w)}{\trans{\opp{\proj 2(w)}}{v}}}}\\
& \eqvsym \modal{\hfib{f}b}
\end{align*}}%
where the transportations along $f(p)$ and ${f(p)}^{-1}$ are with respect to $Q$, and we use \cref{lem:sum_idempotent} on the penultimate line.
Therefore, if either of $\modal{\hfib{\varphi}{\pairr{b,v}}}$ or $\modal{\hfib{f}b}$ is contractible, so is the other.

In particular, if $f$ is $\modal$-connected, then $\modal{\hfib{f}b}$ is contractible for all $b:B$, and hence so is $\modal{\hfib{\varphi}{\pairr{b,v}}}$ for all $(b,v):\sm{b:B} Q(b)$.
On the other hand, if $\varphi$ is $\modal$-connected, then $\modal{\hfib{\varphi}{\pairr{b,v}}}$ is contractible for all $(b,v)$, hence so is $\modal{\hfib{f}b}$ for any $b:B$ such that there exists some $v:Q(b)$.
Finally, since contractibility is a mere proposition, it suffices to merely have such a $v$.
\end{proof}

\begin{lemma}\label{prop:nconn_fiber_to_total}
Let $P,Q:A\to\type$ be type families and $f:\prd{a:A} \Parens{P(a)\to Q(a)}$.
Then $\total f: \sm{a:A}P(a) \to \sm{a:A} Q(a)$ is $\modal$-connected if and only if each $f(a)$ is $\modal$-connected.
\end{lemma}
\begin{proof}
We have
$\hfib{\total f}{\pairr{x,v}}\eqvsym\hfib{f(x)}v$
for each $x:A$ and $v:Q(x)$. Hence \[\modal{\hfib{\total f}{\pairr{x,v}}}\] is contractible if and only if
$\modal{\hfib{f(x)}v}$ is contractible.
\end{proof}

Of course, the ``if'' direction of \cref{prop:nconn_fiber_to_total} is a special case of \cref{lem:nconnected_postcomp_variation}.
This suggests a similar generalization of the ``only if'' direction of \cref{prop:nconn_fiber_to_total}, which would be a version of \cref{lem:nconnected_postcomp_variation} asserting that if $f$ and $\varphi$ are $\modal$-connected then so is each $g_a$.
However, this is not true in general; we will see in \cref{thm:lex-modalities} that it holds if and only if the modality is lex.

Finally, we note that the $\modal$-modal and $\modal$-connected maps are classified.
More generally, we prove the following generalization of~\cite[Thm~3.31]{RijkeSpitters:Sets}.

\begin{thm}
Let $P:\UU\to\prop$ be a predicate on the universe, let $\UU_P\defeq
\sm{X:\UU}P(x)$ and ${(\UU_P)}_\bullet\defeq\sm{X:\UU_P}X$. The projection
$\proj1:{(\UU_P)}_\bullet\to\UU_P$ \emph{classifies} the maps whose fibers satisfy $P$, in the sense that these are exactly the maps that occur as pullbacks of it.
\end{thm}

\begin{proof}
  The fiber of $\proj1:{(\UU_P)}_\bullet\to\UU_P$ over $X:\UU_P$ is $X$, which satisfies $P$ by definition.
  Thus all fibers of this map satisfy $P$, hence so do all fibers of any of its pullbacks.

Conversely, let $f:Y\to X$ be any map into $X$. Then $\hfibfunc{f}:X\to\UU$ factors through
$\UU_P$ if and only if all the fibers of $f$ satisfy $P$. Let us write
$P(f)$ for $\prd{x:X}P(\hfib{f}{x})$. Then we see that the equivalence
$\chi$ of Theorem 4.8.3 of~\cite{TheBook} restricts to an
equivalence
\begin{equation*}
\chi^P:(\sm{Y:\UU}{f:Y\to X}P(f))\to(X\to\UU_P).
\end{equation*}
Now observe that the outer square and the square on the right in the diagram
\begin{equation*}
\begin{tikzcd}[column sep=6em]
Y \arrow[d,swap,"f"] \arrow[rr,"{\lam{y}\pairr{\hfib{f}{f(y)},\blank,\pairr{y,\refl{f(y)}}}}"] & & \pointed{(\UU_P)} \arrow[r] \arrow[d] & \pointed{\UU} \arrow[d] \\
X \arrow[rr,swap,"{\hfibfunc{f}}"] & & \UU_P \arrow[r] & \UU
\end{tikzcd}
\end{equation*}
are pullback squares. Hence the square on the left is a pullback square.
\end{proof}

\begin{corollary}
The $\modal$-modal maps are classified by the universe of $\modal$-modal types, and the $\modal$-connected maps are classified by the universe of $\modal$-connected types.\qed%
\end{corollary}

\subsection{Stable orthogonal factorization systems}%
\label{sec:sofs}

To complete \cref{sec:modal-refl-subun}, we will show that stable orthogonal factorization systems are also determined by their modal types, and give rise to higher modalities.

\subsubsection{Orthogonal factorization systems}%
\label{sec:ofs}

In classical category theory, orthogonal factorization systems are equivalently characterized by a unique lifting property.
We begin with the analogue of this in our context.

\begin{defn}\label{defn:fillers}
Let $(\mathcal{L},\mathcal{R})$ be an orthogonal factorization system, and
consider a commutative square
\begin{equation*}
\begin{tikzcd}
A \arrow[r,"f"] \arrow[d,swap,"l"] \ar[dr,phantom,"\scriptstyle S"] & X \arrow[d,"r"] \\
B \arrow[r,swap,"g"] & Y
\end{tikzcd}
\end{equation*}
(i.e.\ paths $S : r\circ f = g\circ l$)
for which $l$ is in $\mathcal{L}$ and $r$ is in $\mathcal{R}$. We define
$\fillers S$ to be the type of \define{diagonal fillers}
of the above diagram, i.e.~the type of tuples $(j,H_f,H_g,K)$ consisting of
$j:B\to X$, $H_f:j\circ l=f$ and $H_g:r\circ j=g$ and an equality $K : r\circ H_f = \ct S{(H_g \circ l)}$.
\end{defn}

The equality $K$ is required because of homotopy coherence: the commutativity of the given square and of the two triangles are not mere propositions but \emph{data} consisting of homotopies inhabiting those squares and triangles, so to actually have a ``filler'' in the homotopy coherent sense we need to know that the ``pasting composite'' of the two triangles is the given square.

\begin{lemma}\label{lem:diagonal_fillers}
Let $(\mathcal{L},\mathcal{R})$ be an orthogonal factorization system, and
consider a commutative square
\begin{equation*}
\begin{tikzcd}
A \arrow[r,"f"] \arrow[d,swap,"l"] \ar[dr,phantom,"\scriptstyle S"] & X \arrow[d,"r"] \\
B \arrow[r,swap,"g"] & Y
\end{tikzcd}
\end{equation*}
for which $l$ is in $\mathcal{L}$ and $r$ is in $\mathcal{R}$. Then the type
$\fillers S$ of diagonal fillers is contractible.
\end{lemma}

\begin{proof}
By the fact that every morphism factors uniquely as a left map followed by a
right map, we may factorize $f$ and $g$ in $(\mathcal{L},\mathcal{R})$ as $H_f : f = f_\cR \circ f_\cL$ and $H_g : g = g_\cR \circ g_\cL$, obtaining the diagram
\begin{equation*}
\begin{tikzcd}
A \arrow[r,"f_{\mathcal{L}}"] \arrow[d,swap,"l"] & \im(f) \arrow[r,"f_{\mathcal{R}}"] & X \arrow[d,"r"] \\
B \arrow[r,swap,"g_{\mathcal{L}}"] & \im(g) \arrow[r,swap,"g_{\mathcal{R}}"] & Y.
\end{tikzcd}
\end{equation*}
Now both $(r\circ f_{\mathcal{R}})\circ f_{\mathcal{L}}$ and
$g_{\mathcal{R}}\circ(g_{\mathcal{L}}\circ l)$ are factorizations
of the same function $r\circ f:A\to Y$.
Since $\fact_{\mathcal{L},\mathcal{R}}(r\circ f)$ is contractible, so is its identity type
\[ (\im(f), f_\cL, r\circ f_\cR, r\circ H_f) = (\im(g), g_\cL \circ l, g_\cR, \ct{S}{(H_g\circ l)}). \]
This identity type is equivalent to
\begin{multline*}
\sm{e:\im(f) \simeq \im(g)}{H_\cL : g_\cL \circ l = e\circ f_\cL}{H_\cR : r\circ f_\cR = g_\cR\circ e}\\
(\ct{(r\circ H_f)}{(H_\cR \circ f_\cL)} = \ct S{\ct{(H_g \circ l)}{(g_\cR \circ H_\cL)}})
\end{multline*}
Now since $\fact_{\cL,\cR}(f)$ and $\fact_{\cL,\cR}(g)$ are also contractible, we can sum over them to get that the following type is contractible:
\begin{multline*}
  \sm{\im(f):\UU}{f_\cL : A \to \im(f)}{f_\cR : \im(f) \to X}{H_f : f = f_\cR \circ f_\cL}\\
  \sm{\im(g):\UU}{g_\cL : B \to \im(g)}{g_\cR : \im(g) \to Y}{H_g : g = g_\cR \circ g_\cL}\\
\sm{e:\im(f) \simeq \im(g)}{H_\cL : g_\cL \circ l = e\circ f_\cL}{H_\cR : r\circ f_\cR = g_\cR\circ e}\\
(\ct{(r\circ H_f)}{(H_\cR \circ f_\cL)} = \ct S{\ct{(H_g \circ l)}{(g_\cR \circ H_\cL)}})
\end{multline*}
(omitting the hypotheses that $f_\cL,g_\cL\in\cL$ and $f_\cR,g_\cR\in\cR$).
Reassociating and removing the contractible type $\sm{\im(g):\UU}(\im(f) \simeq \im(g))$, and renaming $\im(f)$ as simply $I$, this is equivalent to
\begin{multline*}
  \sm{I:\UU}{f_\cL : A \to I}{f_\cR : I \to X}{H_f : f = f_\cR \circ f_\cL}\\
  \sm{g_\cL : B \to I}{g_\cR : I \to Y}{H_g : g = g_\cR \circ g_\cL}{H_\cL : g_\cL \circ l = f_\cL}{H_\cR : r\circ f_\cR = g_\cR}\\
(\ct{(r\circ H_f)}{(H_\cR \circ f_\cL)} = \ct S{\ct{(H_g \circ l)}{(g_\cR \circ H_\cL)}})
\end{multline*}
Removing the contractible $\sm{f_\cL : A \to I} (g_\cL \circ l = f_\cL)$ and $\sm{g_\cR : I \to Y} (r\circ f_\cR = g_\cR)$, this becomes
\begin{multline*}
  \sm{I:\UU}{f_\cR : I \to X}{g_\cL : B \to I}{H_f : f = f_\cR \circ g_\cL \circ l}{H_g : g = r\circ f_\cR \circ g_\cL}
(r\circ H_f = \ct S{(H_g \circ l)})
\end{multline*}
Inserting a contractible $\sm{j:B\to X} (f_\cR \circ g_\cL = j)$, and reassociating some more, we get
\begin{multline*}
  \sm{j:B\to X}{I:\UU}{f_\cR : I \to X}{g_\cL : B \to I}{H_j:f_\cR \circ g_\cL = j}\\
  \sm{H_f : f = f_\cR \circ g_\cL \circ l}{H_g : g = r\circ f_\cR \circ g_\cL}
  (r\circ H_f = \ct S{(H_g \circ l)})
\end{multline*}
But now $\sm{I:\UU}{f_\cR : I \to X}{g_\cL : B \to I}{H_j:f_\cR \circ g_\cL = j}$ is just $\fact_{\cL,\cR}(j)$, hence contractible.
Removing it, we get
\begin{equation*}
  \sm{j:B\to X}{H_f : f = j \circ l}{H_g : g = r\circ j}(r\circ H_f = \ct S{(H_g \circ l)})
\end{equation*}
which is just $\fillers S$.
Therefore, this is also contractible.
\end{proof}

\begin{defn}\label{defn:orthogonal}
For any class $\mathcal{C}:\prd*{A,B:\UU}(A\to B)\to\prop$ of maps, we define
\begin{enumerate}
\item $^{\bot}\mathcal{C}$ to be the class of maps with \define{(unique) left lifting
property} with respect to all maps in $\mathcal{C}$: the mere proposition
$({}^\bot\mathcal{C})(l)$ asserts that for every commutative square
\begin{equation*}
\begin{tikzcd}
A \arrow[r,"f"] \arrow[d,swap,"l"] \ar[dr,phantom,"S"] & X \arrow[d,"r"] \\
B \arrow[r,swap,"g"] & Y
\end{tikzcd}
\end{equation*}
with $r$ in $\mathcal{C}$, the type $\fillers S$ of diagonal fillers is contractible.
\item $\mathcal{C}^\bot$ to be the class of maps with the dual \define{(unique) right lifting
property} with respect to all maps in $\mathcal{C}$.
\item $l\perp r$ to mean $r\in {\{l\}}^\perp$ (equivalently, $l\in {}^{\perp}\{r\}$).
\end{enumerate}
\end{defn}

\begin{lemma}\label{lem:ofs_lifting}
In an orthogonal factorization system $(\mathcal{L},\mathcal{R})$, one has
$\mathcal{L}={^\bot\mathcal{R}}$ and $\mathcal{L}^\bot=\mathcal{R}$.
\end{lemma}

\begin{proof}
We first show that $\mathcal{L}={^\bot\mathcal{R}}$, i.e.~we show that
$\mathcal{L}(f)\leftrightarrow {^\bot\mathcal{R}}(f)$ for any map $f$. Note
that the implication $\mathcal{L}(f)\to {^\bot\mathcal{R}}(f)$ follows from
\autoref{lem:diagonal_fillers}.

Let $f:A\to B$ be a map in ${^\bot\mathcal{R}}$.
We wish to show that $\mathcal{L}(f)$. Consider the factorization
$(f_{\mathcal{L}},f_{\mathcal{R}})$ of $f$. Then the square
\begin{equation*}
\begin{tikzcd}
A \arrow[r,"f_{\mathcal{L}}"] \arrow[d,swap,"f"] & \mathsf{im}_{\mathcal{L},\mathcal{R}}(f) \arrow[d,"f_{\mathcal{R}}"] \\
B \arrow[r,swap,"\idfunc"] & B
\end{tikzcd}
\end{equation*}
commutes. Since $f$ has the left lifting property, the type of diagonal fillers
of this square is contractible. Thus we have a section $j$ of $f_{\mathcal{R}}$.
The map $j\circ f_\mathcal{R}$ is then a diagonal filler of the square
\begin{equation*}
\begin{tikzcd}
A \arrow[r,"f_{\mathcal{L}}"] \arrow[d,swap,"f_{\mathcal{L}}"] & \mathsf{im}_{\mathcal{L},\mathcal{R}}(f) \arrow[d,"f_{\mathcal{R}}"] \\
\mathsf{im}_{\mathcal{L},\mathcal{R}}(f) \arrow[r,swap,"f_{\mathcal{R}}"] & B.
\end{tikzcd}
\end{equation*}
Of course, the identity map $\idfunc[\mathsf{im}_{\mathcal{L},\mathcal{R}}(f)]$
is also a diagonal filler for this square, so the fact that the type of
such diagonal fillers is contractible implies that $j\circ f_{\mathcal{R}}=\idfunc$.
Thus, $j$ and $f_\cR$ are inverse equivalences, and so the pair $(B,f)$ is equal to the pair $(\mathsf{im}_{\mathcal{L},\mathcal{R}}(f),f_\cL)$.
Hence $f$, like $f_\cL$, is in $\cL$.

Similarly, \autoref{lem:diagonal_fillers} also implies that $\mathcal{R}(f)\to \mathcal{L}^\bot(f)$
for any map $f$, while we can prove $\mathcal{L}^\bot(f)\to\mathcal{R}(f)$ analogously to ${^\bot\mathcal{R}}(f)\to\mathcal{L}(f)$.
\end{proof}

\begin{corollary}\label{lem:sofs_req}
The data of two orthogonal factorization systems $(\mathcal{L},\mathcal{R})$ and
$(\mathcal{L}',\mathcal{R}')$ are identical if and only if
$\mathcal{R}=\mathcal{R}'$.
\end{corollary}
\begin{proof}
  ``Only if'' is obvious.
  Conversely, if $\mathcal{R}=\mathcal{R}'$, then by \cref{lem:ofs_lifting} we have $\cL = \cL'$, and the remaining data of an orthogonal factorization system is a mere proposition.
\end{proof}

\begin{lemma}\label{lem:ofs_rightstable}
Let $(\mathcal{L},\mathcal{R})$ be an orthogonal factorization system. Then
the class $\mathcal{R}$ is stable under pullbacks.
\end{lemma}

\begin{proof}
Consider a pullback diagram
\begin{equation*}
\begin{tikzcd}
A \arrow[d,swap,"k"] \arrow[r,"g"] & X \arrow[d,"h"] \\
B \arrow[r,swap,"f"] & Y
\end{tikzcd}
\end{equation*}
where $h:X\to Y$ is assumed to be in $\mathcal{R}$, and let $k=k_{\mathcal{R}}\circ k_\mathcal{L}$ be a factorization of $k$.
Then the outer rectangle in the diagram
\begin{equation*}
\begin{tikzcd}
A \arrow[r,equals] \arrow[d,swap,"k_{\mathcal{L}}"] & A \arrow[d,swap,"k"] \arrow[r,"g"] & X \arrow[d,"h"] \\
\im_{\mathcal{L},\mathcal{R}}(k) \arrow[r,swap,"k_{\mathcal{R}}"] & B \arrow[r,swap,"f"] & Y
\end{tikzcd}
\end{equation*}
commutes, so by \cref{lem:diagonal_fillers} there is a diagonal lift $j:\im_{\mathcal{L},\mathcal{R}}(k)\to X$ with $i \circ k_{\cL} = g$ and $h\circ i = f \circ k_{\cR}$.
Then by the universal property of pullbacks, we obtain a map $j:\im_{\mathcal{L},\mathcal{R}}(k)\to A$ with $g\circ j = i$ and $k\circ j=k_{\mathcal{R}}$.
And since $g\circ j \circ k_{\cL} = i\circ k_{\cL} = g$ and $k\circ j\circ k_{\cL} = k_{\cR}\circ k_{\cL} = k$ (by homotopies coherent with the pullback square), the uniqueness aspect of the pullback gives $j\circ k_{\mathcal{L}}=\idfunc$.

It suffices to show that $k_{\mathcal{L}}$ is an equivalence, and since we already have that $j\circ k_{\mathcal{L}}=\idfunc$ we only need to show that $k_{\mathcal{L}}\circ j=\idfunc$.
We do this using the contractibility of the type of diagonal fillers. Consider the square
\begin{equation*}
\begin{tikzcd}
A \arrow[r,"k_{\mathcal{L}}"] \arrow[d,swap,"k_{\mathcal{L}}"] & \im_{\mathcal{L},\mathcal{R}}(k) \arrow[d,"k_{\mathcal{R}}"] \\
\im_{\mathcal{L},\mathcal{R}}(k) \arrow[r,swap,"k_{\mathcal{R}}"] & B,
\end{tikzcd}
\end{equation*}
for which $\idfunc:\im_{\mathcal{L},\mathcal{R}}(k)\to \im_{\mathcal{L},\mathcal{R}}(k)$ (with the trivial homotopies) is a diagonal filler. However, we also have the homotopies $k_{\mathcal{L}}\circ j\circ k_{\mathcal{L}} \htpy k_{\mathcal{L}}$ and $k_{\mathcal{R}}\circ k_{\mathcal{L}}\circ j\htpy k\circ j\htpy k_{\mathcal{R}}$. This shows that we have a second diagonal filler, of which the underlying map is $k_{\mathcal{L}}\circ j$. Since the type of diagonal fillers is contractible, it follows that $k_{\mathcal{L}}\circ j=\idfunc$, as desired.
\end{proof}

\subsubsection{Stable orthogonal factorization systems}%
\label{sec:sofs2}

\begin{lemma}\label{lem:fill_compute}
Given $l,r,f,g$ and a homotopy $S : r \circ f = g  \circ l$, consider as $b:B$ varies all the diagrams of the form
\begin{equation*}
\begin{tikzcd}
\hfib{l}{b} \arrow[r,"\proj1"] \arrow[d,"!"'] & A \arrow[d,swap,"l"] \arrow[r,"f"] \ar[dr,phantom,"S"] & X \arrow[d,"r"] \\
\unit \arrow[r,swap,"b"] & B \arrow[r,swap,"g"] & Y
\end{tikzcd}
\end{equation*}
and write $S_b : r \circ (f \circ \proj1) = (g\circ b) \circ \mathord !$ for the induced commutative square. 
Then the map
\begin{equation*}
\fillers{S} \to \prd{b:B}\fillers{S_b},
\end{equation*}
defined by precomposition with $b$, is an equivalence.
\end{lemma}

\begin{proof}
  The domain and codomain of the map in question are by definition
  \[
    \sm{j:B\to X}{H_f :j\circ l=f}{H_g:r\circ j=g} r\circ H_f = \ct{S}{(H_g\circ l)}
  \]
  and
  \begin{equation*}
  \prd{b:B}\sm{j_b:\unit\to X}{H_{f,b} : j_b\circ \mathord{!}=f\circ \proj1}{H_{g,b}: r\circ j_b=g\circ b} r\circ H_{f,b} = \ct{S_b}{(H_{g,b}\circ \mathord!)}.
  \end{equation*}
  The latter is equivalent (using function extensionality and contractibility of $\unit$) to
  \begin{multline*}
    \prd{b:B}\sm{j_b:X}{H_{f,b} : \prd{u:\hfib l b} j_b=f(\proj1(u))}{H_{g,b}: r(j_b)=g(b)}\\
    \prd{u:\hfib l b} r(H_{f,b}(u)) = \ct{S_b}{H_{g,b}}.
  \end{multline*}
  and thereby to
  \begin{multline*}
    \sm{j:B\to X}{H_{f} : \prd{b:B}\prd{u:\hfib l b} j(b)=f(\proj1(u))}{H_{g}: \prd{b:B} r(j(b))=g(b)}\\
    \prd{b:B}\prd{u:\hfib l b} r(H_{f}(b,u)) = \ct{S_b}{H_{g}(b)}.
  \end{multline*}
  Modulo these equivalences, the desired map acts as the identity on $j:B\to X$.
  Moreover, its action on the remaining parts is given by the equivalences
  \begin{align*}
    (j\circ l = f)
    &\eqvsym \prd{a:A} j(l(a)) = f(a)\\
    &\eqvsym \prd{a:A}{b:B}{p:l(a)=b} j(l(a)) = f(a)\\
    &\eqvsym \prd{b:B}{a:A}{p:l(a)=b} j(b) = f(a)\\
    &\eqvsym \prd{b:B} \prd{u:\hfib l b} j(b) = f(\proj1(u))
  \end{align*}
  and
  \begin{equation*}
    (r\circ j = g)
    \eqvsym \prd{b:B} r(j(b)) = g(b)
  \end{equation*}
  and
  \begin{align*}
    (r\circ H_f = \ct{S}{(H_g\circ l)})
    &\eqvsym \prd{a:A} r(H_f(a)) = \ct{S(a)}{H_g(l(a))}\\
    &\eqvsym \prd{a:A}{b:B}{p:l(a)=b} r(H_f(a)) = \ct{S(a)}{H_g(l(a))}\\
    &\eqvsym \prd{b:B}{a:A}{p:l(a)=b} r(H_f(a)) = \ct{S(a)}{H_g(b)}\\
    &\eqvsym \prd{b:B}{u:\hfib l b} r(H_f(b,u)) = \ct{S_b}{H_g(b)}
  \end{align*}
  hence the whole thing is an equivalence.
\end{proof}

\begin{corollary}
In any orthogonal factorization system
$(\mathcal{L},\mathcal{R})$, if
$l:A\to B$ is a map such that $\hfib{l}{b} \to \unit$ is in $\cL$ for each $b:B$, then also $l$ itself is in $\cL$.
\end{corollary}
\begin{proof}
  By \cref{lem:ofs_lifting}, $l$ is in $\cL$ iff $\fillers S$ is contractible for each $r\in\cR$ and $S$ as in \cref{lem:fill_compute}, while similarly $\hfib{l}{b} \to \unit$ is in $\cL$ iff $\fillers {S_b}$ is contractible.
  But the product of contractible types is contractible.
\end{proof}

\begin{corollary}\label{thm:detect-right-by-fibers}
  In any stable orthogonal factorization system, if $l\perp r$ for all maps $l\in\cL$ of the form $l:A\to \unit$, then $r\in\cR$.
  In particular, for any modality $\modal$, if $X\to (A\to X)$ is an equivalence for all $\modal$-connected types $A$, then $X$ is modal.
\end{corollary}
\begin{proof}
  By \cref{lem:fill_compute}, for any $l\in\cL$ and commutative square $S$ from $l$ to $r$, we have $\fillers{S} \eqvsym \prd{b:B}\fillers{S_b}$.
  Since $(\cL,\cR)$ is stable, each map $\mathord{!}_b:\hfib{l}{b}\to \unit$ is also in $\cL$, so that $\mathord{!}_b\perp r$ by assumption.
  Thus $\fillers{S_b}$ is contractible for all $b$, hence so is $\fillers{S}$.

  For the second statement, the type $A\to X$ is equivalent to the type of commutative squares
  \[
  \begin{tikzcd}
    A \ar[r,"f"] \ar[d] & X \ar[d] \\ \unit\ar[r] & \unit
  \end{tikzcd}
  \]
  and the type of fillers for such a square is equivalent to the type of $x:X$ such that $f(a) = x$ for all $a:A$, i.e.\ the fiber of $X\to (A\to X)$ over $f$.
  Thus, the assumption ensures that all such types of fillers are contractible, i.e.\ $l\perp r$ for all $\modal$-connected maps of the form $l:A\to \unit$, so the first statement applies.
\end{proof}

\begin{lemma}\label{lem:sofs_rfib}
Let $(\mathcal{L},\mathcal{R})$ be a stable orthogonal factorization system.
Then a map $r:X\to Y$ is in $\mathcal{R}$ if and only if $\hfib{r}{y}$
is $(\mathcal{L},\mathcal{R})$-modal for each $y:Y$.
\end{lemma}

\begin{proof}
The class of right maps is stable under pullbacks by \autoref{lem:ofs_rightstable},
so it suffices to show that any map with modal fibers is in $\mathcal{R}$.

Let $r:X\to Y$ be a map with modal fibers. Our goal is to show that
$r$ is in $\mathcal{R}$. By \autoref{lem:ofs_lifting} it suffices to show that
$r$ has the right lifting property with respect to the left maps.
Consider a diagram of the form
\begin{equation*}
\begin{tikzcd}
A \arrow[d,swap,"l"] \arrow[r,"f"] & X \arrow[d,"r"] \\
B \arrow[r,swap,"g"] & Y
\end{tikzcd}
\end{equation*}
in which $l$ is a map in $\mathcal{L}$.
We wish to show that the type of diagonal fillers is contractible.
By \autoref{lem:fill_compute}, the type of diagonal fillers of the above diagram
is equivalent to the dependent product of the types of fillers of
\begin{equation*}
\begin{tikzcd}
\hfib{l}{b} \arrow[d] \arrow[r,"f\circ i_b"] & X \arrow[d,"r"] \\
\unit \arrow[r,swap,"g(b)"] & Y
\end{tikzcd}
\end{equation*}
indexed by $b:B$. Thus, it suffices that the type of diagonal fillers for this
square is contractible for each $b:B$. Since any filler factors uniquely through
the pullback $\unit\times_Y X$, which is $\hfib{r}{g(b)}$, the type of diagonal
fillers of the above square is equivalent to the type of diagonal fillers of the
square
\begin{equation*}
\begin{tikzcd}
\hfib{l}{b} \arrow[d] \arrow[r,densely dotted] & \hfib{r}{g(b)} \arrow[d] \\
\unit \arrow[r,equals] & \unit
\end{tikzcd}
\end{equation*}
where the dotted map is the uniqe map into the pullback $\hfib{r}{g(b)}$. In
this square, the left map is in $\mathcal{L}$ because $\mathcal{L}$ is assumed
to be stable under pullbacks, and the right map is in $\mathcal{R}$ by assumption,
so the type of diagonal fillers is contractible.
\end{proof}

\begin{thm}\label{thm:subuniv-sofs}
Any two stable orthogonal factorization systems with the same modal types are
equal.
\end{thm}

\begin{proof}
By \autoref{lem:sofs_req} it follows that any orthogonal factorization system
is completely determined by the class of right maps.
By \autoref{lem:sofs_rfib} it follows that in a stable orthogonal factorization
system, the class of right maps is completely determined by the modal types.
\end{proof}

\begin{thm}\label{thm:highermod_from_sofs}
Any stable orthogonal factorization system determines a higher modality with
the same modal types.
\end{thm}

\begin{proof}
For every type $X$ we have the $(\cL,\cR)$-factorization $X\to\modal X\to\unit$ of the
unique map $X\to\unit$. This determines the modal unit
$\modalunit:X\to\modal X$ which is in $\mathcal{L}$, and the
unique map $\modal X\to\unit$ is in $\mathcal{R}$, i.e.\ $\modal X$ is $(\cL,\cR)$-modal.

To show the induction principle, let $P:\modal X\to\UU$ and $f:\prd{x:X} \modal(P(\eta(x)))$.
Then we have a (judgmentally) commutative square
\begin{equation*}
\begin{tikzcd}
X \arrow[r,"f"] \arrow[d,swap,"\modalunit"] & \sm{z:\modal X}\modal(P(z)) \arrow[d,"\proj1"] \\
\modal X \arrow[r,equals] & \modal X.
\end{tikzcd}
\end{equation*}
Note that by \autoref{lem:sofs_rfib},
the projection $\proj1:(\sm{z:\modal X}\modal(P(z)))\to\modal X$ is in $\mathcal{R}$
because its fibers are modal. Also, the modal unit
$\modalunit:X\to\modal X$ is in $\mathcal{L}$.
Thus, by \cref{defn:orthogonal}, the type of fillers of this square is contractible.
Such a filler consists of a function $s$ and homotopies filling the two triangles
\begin{equation*}
\begin{tikzcd}
X \arrow[r,"f"] \arrow[d,swap,"\modalunit"] & \sm{z:\modal X}\modal(P(z)) \arrow[d,"\proj1"] \\
\modal X \arrow[r,equals] \arrow[ur,densely dotted] & \modal X
\end{tikzcd}
\end{equation*}
whose composite is reflexivity, i.e.\ the type
\begin{multline*}
\sm{s:\modal X \to \sm{z:\modal X}\modal(P(z))}{H:\prd{z:\modal X} \proj1(s(z))=z}{K:\prd{x:X} s(\modalunit(x))=f(x)}\\
\prd{x:X} \proj1(K(x)) = H(\modalunit(x)).
\end{multline*}
If we decompose $s$, $f$, and $K$ by their components, we get
\begin{multline*}
\sm{s_1:\modal X \to \modal X}{s_2:\prd{z:\modal X} \modal(P(s_1(z)))}{H:\prd{z:\modal X} s_1(z)=z}\\
\sm{K_1:\prd{x:X} s_1(\modalunit(x))=f_1(x)}{K_2 :\prd{x:X} s_2(\modalunit(x)) =_{K_1(x)} f_2(x)}\\
\prd{x:X} K_1(x) = H(\modalunit(x)).
\end{multline*}
Now we can contract $s_1$ and $H$, and also $K_1$ with the final unnamed homotopy, to get
\begin{equation*}
\sm{s_2:\prd{z:\modal X} \modal(P(z))}  \prd{x:X} s_2(\modalunit(x)) = f_2(x).
\end{equation*}
But this is just the type of extensions of $f$ along $\modalunit$, i.e.\ the fiber of precomposition by $\modalunit$.
Thus, precomposition by $\modalunit$ is an equivalence, so in fact we have a uniquely eliminating modality.
By \cref{lem:rs_idstable}, the identity types of $\modal X$ are modal, so we have a higher modality as well.
\end{proof}

\section{Localization}\label{sec:localization}

Localization is the process of inverting a specified class of maps.
In category theory, the localization of a category $\mathcal{C}$ at a family of maps $F$ is obtained by adding formal inverses to those maps freely, obtaining a category $\mathcal{C}[F^{-1}]$ with a universal functor $\mathcal{C}\to \mathcal{C}[F^{-1}]$ sending each map in $F$ to an isomorphism.
In good situations, this universal functor is equivalent to the reflection onto a reflective subcategory of $\mathcal{C}$, which consists of the \emph{$F$-local objects}: those that ``see each map in $F$ as an isomorphism''.
We will not be concerned here with the universal property of the localized category; instead we are interested in constructing reflective subcategories of local objects.
We can do this with a higher inductive type, giving a general construction of reflective subuniverses and modalities.

\subsection{Local types and null types}%
\label{sec:local-types}

\begin{defn}\label{defn:Flocal}
Consider a family $F:\prd{a:A}B(a)\to C(a)$ of maps. We say that a type $X$
is \define{$F$-local} if the function
\begin{equation*}
\lam{g}g\circ F_a : (C(a)\to X)\to (B(a)\to X)
\end{equation*}
is an equivalence for each $a:A$. 
\end{defn}

In other words, $X$ is $F$-local if every $f:B(a)\to X$ extends uniquely to a map $\bar{f}:C(a)\to X$, along the map $F_a:B(a)\to C(a)$, as indicated in the diagram
\begin{equation*}
\begin{tikzcd}
B(a) \arrow[r,"f"] \arrow[d,swap,"F_a"] & X. \\
C(a) \arrow[ur,densely dotted,swap,"\bar{f}"]
\end{tikzcd}
\end{equation*}
Thus, one might say that a type $X$ is $F$-local if it is (right) orthogonal to the maps $F_a$, or that it ``thinks each map $F_a$ is an equivalence''.
In \autoref{thm:localization_rs} we will see that the $F$-local types determine a reflective subuniverse.

In most of our examples $C$ will be the constant family $\unit$, giving the following specialization.

\begin{defn}\label{defn:Bnull}
Let $B:A\to \UU$ be a type family. A type $X$ is said to be \define{$B$-null} if the map
\begin{equation*}
\lam{x}\lam{b}x : X \to (B(a) \to X)
\end{equation*}
is an equivalence for each $a:A$. 
\end{defn}

In other words, $X$ is $B$-null if and only if any map $f:B(a)\to X$ has a unique extension to a map $\unit\to X$, as indicated in the diagram
\begin{equation*}
\begin{tikzcd}
B(a) \arrow[r,"f"] \arrow[d] & X. \\
\unit \arrow[ur,densely dotted]
\end{tikzcd}
\end{equation*}
Thus, a type $X$ is $B$-null if it is (right) orthogonal to the types $B(a)$, or that it ``thinks each type $B(a)$ is contractible''.
In \autoref{thm:nullification_modality} we will see that the $B$-null types determine a modality.

\begin{egs}\label{egs:locality}\
\begin{enumerate}
\item The unit type is local for any family of maps.
\item Since $\emptyt\to X$ is contractible for any type $X$, a type is $\emptyt$-null if and only if it is contractible.
\item Any type is $\unit$-null.
\item A type $X$ is $\bool$-null if and only if $X$ is a mere proposition. To see this, recall that a mere proposition is a type for which any two points can be identified. A map of type $\bool\to X$ is equivalently specified by two points in $X$. If $X$ is assumed to be $\bool$-null, and $x,y:X$ are points in $X$, then it follows that there is a (unique) point $z:X$ such that $x=z$ and $y=z$. In particular it follows that $x=y$, so we conclude that $X$ is a mere proposition.
\item More generally, a type is $\Sn^{n+1}$-null if and only if it is $n$-truncated.
  This follows from~\cite[Theorem 7.2.9 and Lemma 6.5.4]{TheBook}.
\item If $Q$ is a mere proposition, then the $Q$-null types are exactly the $\open Q$-modal types (see \cref{eg:open}).
\end{enumerate}
\end{egs}

\begin{rmk}
We choose to consider the notion of being local at a \emph{family} of maps, rather than
as a \emph{class} of maps (i.e.~a subtype of $\sm{X,Y:\UU}X\to Y$). A family of maps (indexed by a type $A$ in $\UU$) is intrinsically small with respect to $\UU$, whereas a class
is not. By localizing at a small family of maps, we obtain a small type constructor.
Nevertheless, one can show that for any family $F$ of maps, a type is $F$-local
if and only if it is local at the class $\im(F)$, when $\im(F)$ is regarded
as a subtype of $\sm{X,Y:\UU}X\to Y$. A similar relation holds for
set-quotients in~\cite{RijkeSpitters:Sets}.
\end{rmk}

\subsection{Localizing at a family of maps}%
\label{sec:localizing}

In this subsection we introduce the localization operation and show that it determines a reflective subuniverse, which is a modality in the case of nullification.
We define a modal operator $\localization{F}:\UU\to\UU$ called \define{localization at $F$}, via a construction involving higher inductive types.
The idea is that one of the point constructors will be the modal unit $\modalunit[X]$ and the other constructors build in exactly the data making each $\lam{g}g\circ F_a$ an equivalence.

For this to be homotopically well-behaved, we have to choose a ``good'' notion of equivalence such as those in~\cite[Chapter 4]{TheBook}.
Any such choice is possible, but some are easier than others.
Of those in~\cite{TheBook}, ``bi-invertibility'' is easiest because it allows us to avoid 2-path constructors.
However, the following notion of equivalence, which doesn't appear in~\cite{TheBook}, is easier still.
As we will see, this is because although it does include 2-path constructors, the four data it comprises can be broken into two pairs that can be treated ``uniformly'' despite occuring at ``different dimensions''; thus we only need to deal explicitly with one point constructor and one path constructor (and no 2-path constructors).

For $f:A\to B$ we write
\begin{equation*}
  \mathsf{rinv}(f) \defeq \sm{g:B\to A} (f\circ g = \idfunc[B])
\end{equation*}
and for $x,y:A$ we write $\apfunc{f}^{x,y} : (x=y) \to (fx=fy)$ for the action of $f$ on identities.

\begin{defn}
  We say that $f$ is \textbf{path-split} if we have an inhabitant of the following type:
  \[ \mathsf{pathsplit}(f) \defeq \mathsf{rinv}(f) \times \prd{x,y:A} \mathsf{rinv}(\apfunc{f}^{x,y}). \]
\end{defn}

\begin{thm}
  For any $f$ we have $\eqv{\mathsf{pathsplit}(f)}{\isequiv(f)}$.
\end{thm}
\begin{proof}
  If $f$ is path-split, to show that it is an equivalence it suffices to show that its right inverse $g$ is also a left inverse, i.e.\ that $gfx=x$ for all $x:A$.
  But $fgfx = fx$ since $f\circ g = \idfunc[B]$, and $\apfunc{f} : (gfx=x) \to (fgfx=fx)$ has a right inverse, so $gfx=x$.

  This gives a map $\mathsf{pathsplit}(f) \to \isequiv(f)$; to show that it is an equivalence, we may assume that its codomain is inhabited.
  But if $f$ is an equivalence, then so is $\apfunc{f}^{x,y}$, and hence $\mathsf{rinv}(f)$ and $\mathsf{rinv}(\apfunc{f}^{x,y})$ are both contractible.
  So in this case $\mathsf{pathsplit}(f)$ and $\isequiv(f)$ are both contractible, hence equivalent.
\end{proof}

Now let $F:\prd{a:A} B(a) \to C(a)$ be a family of functions and $X:\UU$.
As a ``first approximation'' to the localization $\localization{F}(X)$, let $\localhit{F}{X}$ be the higher inductive type with the following constructors:
\begin{itemize}
\item $\alpha_X : X \to \localhit{F}{X}$
\item $\mathsf{ext} : \prd*{a:A} (B(a) \to \localhit{F}{X}) \to (C(a) \to \localhit{F}{X})$
\item $\mathsf{isext} : \prd*{a:A}{f:B(a)\to\localhit{F}{X}}{b:B(a)}\id{\mathsf{ext}(f)(F_a(b))}{f(b)}$.
\end{itemize}
The induction principle of $\localhit{F}{X}$ is that for any type family $P:\localhit{F}{X}\to \UU'$, if there are terms
\begin{small}
\begin{align*}
N & : \prd{x:X}P(\alpha_X(x))\\
R & : \prd*{a:A}{f:B(a)\to\localhit{F}{X}}(\prd{b:B(a)}P(f(b)))\to\prd{c:C(a)} P(\mathsf{ext}(f,c)) \\
S & : \prd*{a:A}{f:B(a)\to\localhit{F}{X}}{f':\prd{b:B(a)}P(f(b))}{b:B(a)}\dpath{P}{\mathsf{isext}(f,b)}{R(f')(F_a(b))}{f'(b)},
\end{align*}\end{small}%
then there is a section $s:\prd{x:\localhit{F}{X}}P(x)$ such that $s\circ \alpha_X= N$.
(The section $s$ also computes on $\mathsf{ext}$ and $\mathsf{isext}$, but we will not need those rules.)
Note that the family $P$ does not have to land in the same universe $\UU$ that contains our types $A,B,C,X$; this will be important in \cref{sec:nullification}.

This approximation $\localhit{F}{X}$ behaves like we expect $\localization{F}(X)$ to behave when mapping into local types:

\begin{lemma}\label{thm:appx-loc}
  If $Y$ is $F$-local (and $X$ is arbitrary), then precomposition with $\alpha_X$
  \[ (-\circ \alpha_X) : (\localhit{F}{X} \to Y) \to (X\to Y) \]
  is an equivalence.
\end{lemma}
\begin{proof}
  We will show that this map is path-split.

  First we have to construct a right inverse to it, i.e.\ given $g:X\to Y$ we must extend it to $\localhit{F}{X}$.
  We will apply the induction principle using the constant family $Y$ over $\localhit{F}{X}$ and $N\defeq g$, so that the computation rule shows that what we get is an extension of $g$.
  To construct the cases of $R$ and $S$, let $f:B(a)\to \localhit{F}{X}$, and let $f':B(a)\to Y$.
  Our goal is to construct $R(f,f'):C(a)\to Y$ together with a witness $S(f,f')$ that the triangle
  \begin{equation*}
    \begin{tikzcd}[column sep=large]
      B(a) \arrow[dr,"{f'}"] \arrow[d,swap,"F_a"] \\
      C(a) \arrow[r,swap,"{R(f,f')}"] & Y
    \end{tikzcd}
  \end{equation*}
  commutes.
  But $Y$ is $F$-local, so the map
  \[ (-\circ F_a) : (C(a) \to Y) \to (B(a)\to Y) \]
  is an equivalence, and hence in particular has a right inverse; applying this right inverse to $f'$ gives $R$ and $S$.

  Second, we must suppose given $g,h:\localhit{F}{X} \to Y$ and construct a right inverse to
  \[ \apfunc{(-\circ \alpha_X)} : (g=h) \to (g\circ \alpha_X = h\circ \alpha_X). \]
  Thus, suppose we have $K : \prd{x:X} g(\alpha_X(x)) = h(\alpha_X(x))$; we must extend $K$ to a homotopy $\tilde{K} : \prd{z:\localhit{F}{X}} g(z)=h(z)$ such that $\tilde{K}(\alpha_X(x)) = K(x)$.
  We will apply the induction principle using the family $P:\localhit{F}{X} \to \UU$ defined by $P(z) \defeq (g(z)=h(z))$, and $N\defeq K$.
  To construct the cases of $R$ and $S$, let $f:B(a)\to \localhit{F}{X}$ and $f':\prd{b:B(a)} gfb = hfb$.
  Our goal is to construct $R(f,f'):\prd{c:C(a)} g(\mathsf{ext}(f,c))=h(\mathsf{ext}(f,c))$ together with a witness $S(f,f')$ that for any $b:B(a)$ we have
  \begin{equation}
    R(f,f')(F_a(b)) = \ct{\ap{g}{\mathsf{isext}(f,b)}}{\ct{f'(b)}{\ap{h}{\mathsf{isext}(f,b)}^{-1}}}.\label{eq:locpsRS}
  \end{equation}
  However, once again, since $Y$ is $F$-local, the map
  \[ (-\circ F_a) : (C(a) \to Y) \to (B(a)\to Y) \]
  is an equivalence, and hence in particular
  \begin{equation}
    \apfunc{(-\circ F_a)} :
  (g\circ \mathsf{ext}(f) = h\circ \mathsf{ext}(f)) \to (g\circ \mathsf{ext}(f) \circ F_a = h\circ \mathsf{ext}(f) \circ F_a)\label{eq:locpsap}
  \end{equation}
  has a right inverse.
  But the right-hand side of~\eqref{eq:locpsRS} inhabits the codomain of~\eqref{eq:locpsap}, so applying this right inverse gives $R$ and $S$.
\end{proof}

In general, $\localhit{F}{X}$ is not $F$-local: its constructors only ensure that each map
\[ (-\circ F_a) : (C(a) \to \localhit{F}{X}) \to (B(a) \to \localhit{F}{X}) \]
has a right inverse, not that it is an equivalence.
(In fact, $\localhit{F}{X}$ is the ``free algebraically $F$-injective type on $X$'', cf.~\cite{bourke:alginj}.)

\begin{rmk}\label{rmk:localhit-notlocal}
  However, it does happen in many common cases that $\localhit{F}{X}$ \emph{is} already $F$-local (and hence the $F$-localization of $X$).
  Specifically, this happens whenever each $(-\circ F_a)$ \emph{already} has a left inverse, which happens whenever each $F_a : B(a) \to C(a)$ has a \emph{right} inverse.
  For instance, if $C(a)\defeq\unit$ for all $a$ (so that we are talking about $B$-nullification), then this happens whenever all the types $B(a)$ are inhabited (i.e.\ we have $\prd{a:A}B(a)$); cf.~\cite[Lemma 8.7]{shulman:bfp-realcohesion}.

  In particular, this occurs for $\Sn^{n+1}$-nullification for $n\ge -1$, which as we saw in \cref{egs:locality} coincides with $n$-truncation.
  In this case $\localhit{F}{X}$ essentially reduces to the ``hub and spoke'' construction of truncations from~\cite[\S7.3]{TheBook}.

  A concrete example where $\localhit{F}{X}$ is \emph{not} yet $F$-local is $\emptyset$-nullification, where $\localhit{F}{X} = X+\unit$, but only contractible types are $\emptyset$-null.
  Note that $\emptyset = \Sn^{-1}$, so this is equivalently $(-2)$-truncation.
\end{rmk}

To modify $\localhit{F}{X}$ to become $F$-local using bi-invertibility or half-adjoint equivalences, we would need to add two more constructors to $\localhit{F}{X}$ corresponding to the additional two pieces of data in those definitions of equivalence, and then add two more cases to the proof of \cref{thm:appx-loc} to deal with those constructors.
Moreover, these additional cases are rather more difficult than the ones we gave, since they involve homotopies ``on the other side''.

Fortunately, with path-splitness, we can instead use a simple trick.
Given any map $f:B\to C$, let $\Delta_f : B\to B\times_C B$ be its \textbf{diagonal} and $\nabla_f : C +_B C \to C$ its \textbf{codiagonal}.

\begin{lemma}\label{thm:precomp-codiag}
  For any $f:B\to C$ and any $X$, we have a commuting triangle
\begin{equation*}
\begin{tikzcd}[column sep=-2em]
\phantom{(C\to X) \times_{(B\to X)} (C\to X)} & (C\to X) \arrow[dl,swap,"(-\circ \nabla_f)"] \arrow[dr,"\Delta_{(-\circ f)}"] \\
(C +_B C \to X) \arrow[rr,"\sim"] & & (C\to X) \times_{(B\to X)} (C\to X)
\end{tikzcd}
\end{equation*}
in which the bottom map is an equivalence.
\end{lemma}
\begin{proof}
  By the universal property of the pushout.
\end{proof}

\begin{lemma}\label{thm:pathsplit-delta}
  For any $f:B\to C$, we have
  \[ \mathsf{pathsplit}(f) \eqvsym \mathsf{rinv}(f) \times \mathsf{rinv}(\Delta_f). \]
\end{lemma}
\begin{proof}
  Decomposing $B\times_C B$ and its identity types into $\Sigma$-types, we have
  \begin{align*}
    \mathsf{rinv}(\Delta_f)
    &\eqvsym \prd{x,y:B}{p:fx=fy}\sm{z:B}{q:x=z}{r:z=y} \ct{\apfunc{f}^{x,z}(q)}{\apfunc{f}^{z,y}(r)} = p\\
    &\eqvsym \prd{x,y:B}{p:fx=fy}\sm{r:x=y} \apfunc{f}^{x,y}(r) = p\\
    &\eqvsym \prd{x,y:B} \mathsf{rinv}(\apfunc{f}^{x,y}).\qedhere
  \end{align*}
\end{proof}

\begin{lemma}\label{thm:pathsplit-local}
  For $f:B\to C$, a type $X$ is $f$-local if and only if both maps
  \begin{align*}
    (-\circ f) &: (C\to X) \to (B\to X) \\
    (-\circ \nabla_f) &: (C\to X) \to (C +_B C \to X)
  \end{align*}
  have right inverses, and if and only if both of these maps are equivalences.
\end{lemma}
\begin{proof}
  By \cref{thm:pathsplit-delta}, $X$ is $f$-local if and only if $(-\circ f)$ and $\Delta_{(-\circ f)}$ have right inverses, but by \cref{thm:precomp-codiag} the latter is equivalent to $(-\circ \nabla_f)$.
  The second statement follows since the diagonal of an equivalence is an equivalence.
\end{proof}

\cref{thm:pathsplit-local} implies that for $F$-locality it suffices for precomposition with each $F_a$ and $\nabla_{F_a}$ to have right inverses.
But $\localhit{F}{X}$ is the universal way to make precomposition with each $F_a$ have right inverses, so to localize we just need to add all the morphisms $\nabla_{F_a}$ to $F$.

Specifically, for any $F:\prd{a:A} B(a) \to C(a)$, define $\hat B,\hat C : A+A \to \UU$ and a family $\hat F: \prd{a:A+A} \hat B(a) \to \hat C(a)$ by
\begin{alignat*}{3}
  \hat B(\inl(a)) &\defeq B(a) &\quad \hat C(\inl(a)) &\defeq C(a) &\quad \hat F(\inl(a)) &\defeq F_a\\
  \hat B(\inr(a)) &\defeq C(a) +_{B(a)} C(a) &\quad \hat C(\inr(a)) &\defeq C(a) &\quad \hat F(\inr(a)) &\defeq \nabla_{F_a}.
\end{alignat*}

\begin{defn}
  For any $X:\UU$, the \textbf{localization of $X$ at $F$} is $\localization{F}(X) \defeq \localhit{\hat F}{X}$, and $\modalunit[X] : X\to \localization{F}(X)$ is $\alpha^{\hat F}_X$.
\end{defn}

\begin{eg}
  As noted in \cref{rmk:localhit-notlocal}, a simple example where $\localhit{F}{X}$ is not yet $F$-local is $\emptyset$-nullification, where $F$ is the single map $\emptyset\to\unit$.
  In this case $\hat F$ consists of $\emptyset\to\unit$ and the fold map $\nabla : \unit+\unit \to \unit$.
  The constructors of $\localhit{\hat F}{X}$ corresponding to the former give it a point, and those corresponding to the latter make it a mere proposition (in fact they are the constructors of $(-1)$-truncation, i.e.\ $\Sn^{0}$-nullification).
  Thus, $\localhit{\hat F}{X}$ is contractible, i.e.\ $\emptyset$-local.
\end{eg}

\begin{lemma}\label{lem:islocal_localization}
  For any $F:\prd{a:A} B(a) \to C(a)$, the type $\localization{F}(X)$ is $F$-local.
\end{lemma}
\begin{proof}
  The constructors of $\localization{F}(X)$ as $\localhit{\hat F}{X}$ say that the precomposition maps
  \[ (-\circ \hat F_a) : (\hat C(a) \to \localhit{\hat F}{X}) \to (\hat B(a) \to \localhit{\hat F}{X}) \]
  have right inverses for all $a:A+A$.
  But by definition of $\hat F$, these maps consist of precomposition with each $F_a$ and $\nabla_{F_a}$.
  Thus, by \cref{thm:pathsplit-local}, $\localhit{\hat F}{X}$ is $F$-local.
\end{proof}

\begin{lemma}\label{thm:local-ump}
  If $Y$ is $F$-local (and $X$ is arbitrary), then precomposition with $\modalunit[X]$
  \[ (-\circ \modalunit[X]) : (\localization{F}(X) \to Y) \to (X\to Y) \]
  is an equivalence.
\end{lemma}
\begin{proof}
  By the second clause of \cref{thm:pathsplit-local}, any $F$-local type is also $\hat F$-local; so this follows from \cref{thm:appx-loc}.
\end{proof}

\begin{thm}\label{thm:localization_rs}
  The subuniverse of $F$-local types in $\UU$ is a reflective subuniverse, with modal operator $\localization{F}$.
\end{thm}
\begin{proof}
  By \cref{lem:islocal_localization,thm:local-ump}.
\end{proof}

\subsection{Nullification and accessibility}%
\label{sec:nullification}

A general localization is only a reflective subuniverse, but there is a convenient sufficient condition for it to be a modality: if each $C(a)=\unit$.
A localization modality of this sort is called \emph{nullification}.

\begin{thm}\label{thm:nullification_modality}
  If $F:\prd{a:A} B(a) \to C(a)$ is such that each $C(a)=\unit$, then localization at $F$ is a modality, called \define{nullification at $B$}.
\end{thm}
\begin{proof}
  It suffices to show that for any $B:A\to\UU$, the $B$-null types are $\Sigma$-closed.
  Thus, let $X:\UU$ and $Y:X\to \UU$ be such that $X$ and each $Y(x)$ are $B$-null.
  Then for any $a:A$ we have
  \begin{align*}
    (B(a)\to \sm{x:X} Y(x))
    &\eqvsym \sm{g:B(a)\to X} \prd{b:B(a)} Y(g(b)) \\
    &\eqvsym \sm{x:X} B(a) \to Y(x) \\
    &\eqvsym \sm{x:X} Y(x)
  \end{align*}
  with the inverse equivalence being given by constant maps.
  Thus, $\sm{x:X} Y(x)$ is $B$-null.
\end{proof}

Of course, it might happen that $\localization{F}$ is a modality even if $F$ doesn't satisfy the condition of \cref{thm:nullification_modality}.
For instance, if $B:A\to \UU$ has a section $s:\prd{a:A} B(a)$, then localizing at the family $s' : \prd{a:A} \unit \to B(a)$ is equivalent to nullifying at $B$, since in a section-retraction pair the section is an equivalence if and only if the retraction is.
However, we can say the following.

\begin{lemma}\label{thm:acc-modal}
  If $F:\prd{a:A} B(a)\to C(a)$ is such that $\localization{F}$ is a modality, then there exists a family $E:D\to \UU$ such that $\localization{F}$ coincides with nullification at $E$.
\end{lemma}
\begin{proof}
  Write $\modal\defeq\localization{F}$ and $\modalunit$ for its modal unit.
  Define $D = \sm{a:A} (\modal (B(a)) + \modal(C(a)))$, and $E:D\to \UU$ by
  \begin{align*}
    E(a,\inl(b)) &\defeq \hfib{\modalunit[B(a)]}{b}\\
    E(a,\inr(c)) &\defeq \hfib{\modalunit[C(a)]}{c}.
  \end{align*}
  Then since $\modalunit$ is $\modal$-connected, each $E(d)$ is $\modal$-connected, and hence every $F$-local type is $E$-null.

  On the other hand, suppose $X$ is an $E$-null type.
  Each $\modalunit[B(a)]$ and $\modalunit[C(a)]$ is $\localization{E}$-connected, since their fibers are $\localization{E}$-connected (by definition); thus $X$ is also $\modalunit[B(a)]$-local and $\modalunit[C(a)]$-local.
  But we have the following commutative square:
  \[
  \begin{tikzcd}[column sep=large]
    B(a) \ar[r,"{\modalunit[B(a)]}"] \ar[d,"F_a"'] & \modal(B(a)) \ar[d,"{\modal(F_a)}"]\\
    C(a) \ar[r,"{\modalunit[C(a)]}"'] & \modal(C(a))
  \end{tikzcd}
  \]
  and ${\modal(F_a)}$ is an equivalence; thus $X$ is also $F_a$-local.
  So the $F$-local types coincide with the $E$-null types.
\end{proof}

This shows that the following pair of definitions are consistent.

\begin{defn}\label{defn:accessible}
A reflective subuniverse on $\UU$ is said to be \define{accessible} if it is the localization at a family of maps in $\UU$, indexed by a type in $\UU$.
Similarly, a modality $\modal$ on $\UU$ is said to be \define{accessible} if it is the nullification at a family of types in $\UU$, indexed by a type in $\UU$.

Explicitly, a \define{presentation} of a reflective subuniverse $\modal$ of $\UU$ consists of a family of maps $F : \prd{a:A} B(a) \to C(a)$, where $A:\UU$ and $B,C:A\to\UU$, such that $\modal = \localization{F}$.
Similarly, a \define{presentation} of a modality $\modal$ consists of a family of types $B: A\to\UU$, where $A:\UU$, such that $\modal = \localization{\lam{a} B(a)\to \unit}$.
\end{defn}

\begin{rmk}
Note that being accessible is structure; different families can present the same reflective subuniverse or modality.
As a trivial example, note that localizing at the empty
type, and localizing at the type family on $\bool$ defined by
$\bfalse\mapsto \emptyt$ and $\btrue\mapsto \unit$ both map all types to contractible types.

However, we are usually only interested in properties of presentations insofar as they determine properties of subuniverses.
For instance, by \cref{thm:acc-modal}, a reflective subuniverse is a modality exactly when it has a presentation in which each $C(a)=\unit$.
Similarly, in \cref{sec:lex-top-cotop} we will define a modality to be ``topological'' if it has a presentation in which each $C(a)=\unit$ and each $B(a)$ is a mere proposition.
\end{rmk}

\begin{eg}\label{thm:trunc-acc}
The trivial modality $\truncf{(-2)}$ is presented by $\emptyt$, while the propositional truncation modality $\truncf{(-1)}$ is presented by $\bool$.  More generally, the
$n$-truncation modality $\truncf{n}$ is presented by the $(n+1)$-sphere $\Sn^{n+1}$.
\end{eg}

\begin{eg}\label{thm:open-acc}
For every mere proposition $P$, the open modality $\open P (X) \defeq (P\to X)$ from \cref{eg:open} is
presented by the singleton type family $P$.
To see this, note that $\modalunit[X] : X \to (P\to X)$ is the same as the map in the definition of locality, so that $X$ is modal for the open modality on $P$ if and only if it is $P$-null.
(If $P$ is not a mere proposition, however, then $X\mapsto (P\to X)$ is not a modality, and in particular does not coincide with localization at $P$.)
\end{eg}

\begin{eg}\label{thm:closed-acc}
  The closed modality $\closed P$ from \cref{eg:closed} associated to a mere proposition $P$ is presented by the type family $\lam{x} \emptyt : P \to \UU$.
  For by definition, $A$ is null for this family if and only if for any $p:P$ the map $A \to (\emptyt \to A)$ is an equivalence.
  But $\emptyt \to P$ is contractible, so this says that $P\to\iscontr(A)$, which was the definition of $\closed P$-modal types from \cref{eg:closed}.
\end{eg}

One of the main uses of accessibility is when passing between universes.
Our definitions of reflective subuniverses and modalities are relative to a \emph{particular} universe $\UU$, but most examples are ``uniform'' or ``polymorphic'' and apply to types in all universes (or all sufficiently large universes) simultaneously.
Accessibility is one technical condition which ensures that this holds and that moreover these modal operators on different universes ``fit together'' in a convenient way.
For instance, we have:

\begin{lemma}\label{thm:acc-extend}
  If $\modal$ is an accessible reflective subuniverse on a universe $\UU$, and $\UU'$ is a larger universe containing $\UU$, then there is a reflective subuniverse $\modal'$ on $\UU'$ such that:
  \begin{enumerate}
  \item If $\modal$ is a modality, so is $\modal'$.\label{item:ae5}
  \item A type $X:\UU$ is $\modal'$-modal if and only if it is $\modal$-modal.\label{item:ae1}
  \item For $X:\UU$, the induced map $\modal' X \to \modal X$ is an equivalence.\label{item:ae2}
  \item A type $X:\UU'$ is $\modal'$-modal if and only if $(\blank\circ f) : (B\to X) \to (A\to X)$ is an equivalence for any map $f:A\to B$ in $\UU$ such that $\modal(f)$ is an equivalence.\label{item:ae3}
  \item $\modal'$ depends only on $\modal$, not on a choice of presentation for it.\label{item:ae4}
  \end{enumerate}
\end{lemma}
\begin{proof}
  Since $\modal$ is accessible, it is generated by some family $F:\prd{a:A} B(a) \to C(a)$.
  Define $\modal':\UU'\to\UU'$ to be the higher inductive localization at the same family $F$, which lives in $\UU'$ as well since $\UU'$ is larger than $\UU$.
  If $\modal$ is a modality, we can take each $C(a)=\unit$ so that $\modal'$ is also a modality, giving~\ref{item:ae5}.

  The notion of $F$-locality for a type $X$ is independent of what universe $X$ lives in, giving~\ref{item:ae1}.
  Moreover, because the induction principle for a higher inductive localization allows us to eliminate into any type in any universe, \cref{thm:local-ump} applies no matter what universe the target lives in.
  Thus, if $X:\UU$ then $\modal X$ and $\modal' X$ have the same universal property, hence are canonically equivalent, giving~\ref{item:ae2}.

  To prove~\ref{item:ae3}, note first that certainly each $\modal (F_a)$ is an equivalence, so any type with the stated property is $F$-local.
  Conversely, if $X$ is $F$-local, hence $\modal'$-modal, then $(B\to X) \to (A\to X)$ is certainly an equivalence for any map $f$ such that $\modal'(f)$ is an equivalence; but $\modal'$ and $\modal$ coincide on $\UU$.
  Thus~\ref{item:ae3} holds; and this implies~\ref{item:ae4} since a reflective subuniverse is determined by its modal types.
\end{proof}

We refer to the $\modal'$ constructed in \cref{thm:acc-extend} as the \textbf{canonical accessible extension} of $\modal$ to $\UU'$.

\begin{egs}
  Our characterizations of the truncation and open and closed modalities in \cref{thm:trunc-acc,thm:open-acc,thm:closed-acc} made no reference to the ambient universe.
  Thus, when these modalities are defined in the standard ways on $\UU$ and $\UU'$ respectively, their $\UU'$-version is the canonical accessible extension of their $\UU$-version.
\end{egs}

\begin{eg}
  By contrast, the double-negation modality $\neg\neg$ \emph{is} defined in a polymorphic way on all universes, but in general there seems no reason for it to be accessible on any of them.
  However, if propositional resizing holds, then it is the nullification at $\bool$ together with all propositions $P$ such that $\neg\neg P$ holds, and hence accessible.

  Whether or not any inaccessible modalities remain after imposing propositional resizing may depend on large-cardinal principles.
  It is shown in~\cite{css:large-cardinal} that this is the case for the analogous question about reflective sub-$(\infty,1)$-categories of the $(\infty,1)$-category of $\infty$-groupoids.
\end{eg}

\begin{eg}
  Suppose that all types in $\UU$ are 0-types.
  We have tacitly assumed that all universes are closed under all higher inductive types, so (assuming univalence) this is not actually possible, but to get a feeling for what else could in principle go wrong suppose we drop that assumption.
  Then if $F$ is a family such that the higher inductive type $\localization{F}$ does not preserve 0-types, we might (depending on what we assume about closure under higher inductive types) still be able to define a modality on $\UU$ by $\modal X = \trunc0{\localization{F}X}$.
  But if $\UU'$ is a larger universe containing non-0-types, then this $\modal$ would not eliminate into types in $\UU'$, and if we define $\modal'$ by localizing at $F$ in $\UU'$ then the canonical map $\modal' X \to \modal X$ would be the 0-truncation rather than an equivalence.
  So \cref{thm:acc-extend} is not as trivial as it may seem.
\end{eg}

\begin{rmk}\label{rmk:extend-oops}
  It is tempting to think that \emph{any} reflective subuniverse $\modal$ on $\UU$ could be extended to an accessible one on $\UU'$ by localizing at the family of \emph{all} functions in $\UU$ that are inverted by $\modal$ (or nullifying at the family of all $\modal$-connected types in $\UU$, in the case of modalities), which is a $\UU'$-small family though not a $\UU$-small one.
  This does produce an accessible reflective subuniverse $\modal'$ of $\UU'$ such that the $\modal'$-modal types in $\UU$ coincide with the $\modal$-modal ones, but there seems no reason why the modal \emph{operators} $\modal'$ and $\modal$ should agree on types in $\UU$.
\end{rmk}

\begin{rmk}
  Reflective subuniverses and modalities defined by localization have another convenient property: their eliminators have a strict judgmental computation rule (assuming that our higher inductive localization type has a judgmental computation rule on point-constructors, which is usually assumed).
  This will be useful in \cref{thm:subtopos-model}.
\end{rmk}

\subsection{Non-stable factorization systems}%
\label{sec:nonstable-factsys}

We have seen in \cref{sec:modal-refl-subun} that $\Sigma$-closed reflective subuniverses are equivalent to stable orthogonal factorization systems.
Without $\Sigma$-closedness and stability, this equivalence fails.
However, we can still say:

\begin{lemma}
  Any orthogonal factorization system has an underlying reflective subuniverse, consisting of those types $X$ such that $X\to\unit$ is in $\cR$.
\end{lemma}
\begin{proof}
  If $Y$ is modal in this sense, then by applying orthogonality to squares of the form
  \[
  \begin{tikzcd}
    A \ar[d,"f"'] \ar[r] & Y \ar[d] \\ B \ar[r] & \unit
  \end{tikzcd}
  \]
  we see that if $f:A\to B$ lies in $\cL$, then precomposition
  \[ (-\circ f) : (B\to Y) \to (A\to Y) \]
  is an equivalence.
  Thus, it suffices to show that for every $X$ there is an $\cL$-map $X\to \modal X$ where $\modal X\to \unit$ is in $\cR$; but this is just an $(\cL,\cR)$-factorization of the map $X\to\unit$.
\end{proof}

Conversely, in classical category theory there are various ways of extending a reflective subcategory to a factorization system.
One canonical one is considered in~\cite{chk:reflocfact}, but this is harder to reproduce homotopy-theoretically.
(It is possible in what is there called the ``simple'' case, hence also the ``semi-left-exact'' case --- which includes all modalities, as the case of ``stable units'' --- but we will not investigate that construction here.)
Instead, if we have an \emph{accessible} reflective subuniverse presented by localization at a family of maps, we can generalize the construction of localization to produce a factorization system (though in general the result will depend on the choice of presentation, not just on the reflective subuniverse we started with).

To avoid too much wrangling with witnesses of commutative squares, we will factorize dependent types rather than functions.
In this case, right orthogonality (\cref{defn:orthogonal}) can be expressed in the following way.

\begin{defn}
  Given $l:A\to B$ and $X:Y\to\UU$, and functions $g:B\to Y$ and $f:\prd{a:A} X(g(l(a)))$ forming a judgmentally commutative square
  \begin{equation}
    \begin{tikzcd}[column sep=large]
      A \ar[d,"l"'] \ar[r,"{(g\circ l,f)}"] & \sm{y:Y}X(y) \ar[d,"\proj1"] \\ B \ar[r,"g"'] & Y
    \end{tikzcd}%
    \label{eq:dfill-sq}
  \end{equation}
  a \textbf{dependent filler} in this square consists of a morphism ${j:\prd{b:B} X(g(b))}$ and a homotopy $j\circ l \sim f$.
  That is, the type of dependent fillers is
  \begin{equation}
    \dfill{l,X,g,f} \defeq \sm{j:\prd{b:B} X(g(b))} \prd{a:A} j(l(a)) = f(a).\label{eq:dep-fillers}
  \end{equation}
\end{defn}

Recall that for a map $f:B\to C$, we denote by $\Delta_f : B\to B\times_C B$ its diagonal and $\nabla_f : C +_B C \to C$ its codiagonal.
We have the following dependent generalization of \cref{thm:precomp-codiag}:

\begin{lemma}\label{thm:dep-precomp-codiag}
  Let $f:B\to C$ and $X:Y\to\UU$ and $g:C\to Y$; then we have a commuting triangle
  \begin{small}
\begin{equation*}
\begin{tikzcd}
& \prd{c:C} X(g(c)) \arrow[dl,swap,"(-\circ \nabla_f)"] \arrow[d,"\Delta_{(-\circ f)}"] \\
\prd{z:C +_B C} X(g'(z)) \arrow[r,"\sim"'] &
\Big(\prd{c:C} X(g(c))\Big) \times_{(\prd{b:B} X(g(f(b))))} \Big(\prd{c:C} X(g(c))\Big)
\end{tikzcd}%
\end{equation*}
\end{small}%
where $g':C+_{B} C \to Y$ is induced by $g$ on both copies of $C$, and the bottom map is an equivalence.
\end{lemma}
\begin{proof}
  Like the non-dependent case \cref{thm:precomp-codiag}, this follows from the universal property of the pushout.
\end{proof}

And similarly for \cref{thm:pathsplit-local}:

\begin{lemma}\label{thm:orth-dep}
  For $l:B\to C$ and $X:Y\to\UU$, the following are equivalent.
  \begin{enumerate}
  \item The map $\proj1 : (\sm{y:Y}X(y)) \to Y$ is right orthogonal to $l$.\label{item:od1}
  \item For every $g:C\to Y$ and $f:\prd{b:B} X(g(l(b)))$, the type $\dfill{l,X,g,f}$ of dependent fillers in~\eqref{eq:dfill-sq} is contractible.\label{item:od2}
  \item For every $g:C\to Y$, the precomposition map
    \begin{equation}
      (-\circ l) : \Big(\prd{c:C} X(g(c))\Big) \to \Big(\prd{b:B} X(g(l(b)))\Big)\label{eq:dfill-eqv}
    \end{equation}
    is an equivalence.\label{item:od3}
  \item For every $g:C\to Y$, the precomposition maps\label{item:od4}
    \begin{align*}
      (-\circ l) &: \Big(\prd{c:C} X(g(c))\Big) \to \Big(\prd{b:B} X(g(l(b)))\Big)\\
      (-\circ \nabla_l) &: \Big(\prd{c:C} X(g(c))\Big) \to \Big(\prd{z:C+_{B} C} X(g'(z))\Big)
    \end{align*}
    have right inverses.
  \item For every $g:C\to Y$, the maps in~\ref{item:od4} are equivalences.\label{item:od5}
  \end{enumerate}
\end{lemma}
\begin{proof}
  The equivalence of~\ref{item:od2} and~\ref{item:od3} is immediate, since $\dfill{l,X,g,f}$ is the fiber of~\eqref{eq:dfill-eqv} over $f$.
  And as in \cref{thm:pathsplit-local},~\ref{item:od3} is equivalent to~\ref{item:od4} and~\ref{item:od5} using \cref{thm:dep-precomp-codiag,thm:pathsplit-delta}.

  Finally, regarding~\ref{item:od1}, if we have any commutative square
  \[
  \begin{tikzcd}
    B \ar[d,"l"'] \ar[r,"f'"] \ar[dr,phantom,"S"] & \sm{y:Y}X(y) \ar[d,"\proj1"] \\ C \ar[r,"g"'] & Y
  \end{tikzcd}
  \]
  witnessed by $S:\proj1 \circ f'=g\circ l$, we can define $f(b) \defeq \trans{S(b)}{\proj2(f'(b))}$ to get an equivalent and judgmentally commutative square as in~\eqref{eq:dfill-sq}.
  Thus,~\ref{item:od1} is equivalent to its restriction to such squares.
  But given such a square, the type of ordinary diagonal fillers (\cref{defn:fillers}) is equivalent to
  \[ \sm{j:C\to \sm{y:Y} X(y)}{H_f : j\circ l = (g\circ l,f)}{H_g : \proj1 \circ j = g} \proj1 \circ H_f = H_g \circ l \]
  and thereby to
  \begin{multline*}
    \sm{j_1:C\to Y}{j_2 : \prd{c:C} X(j_1(c))}\\
    \sm{H_{f1} : j_1 \circ l = g\circ l}{H_{f2} : \dpath{X}{H_{f1}}{j_2\circ l}{f}}{H_g : j_1 = g} H_{f1} = H_g \circ l.
  \end{multline*}
  But now we can contract two based path spaces (combining $j_1$ with $H_g$, and $H_{f1}$ with the final unnamed equality $H_{f1} = H_g\circ l$) to get the type~\eqref{eq:dep-fillers} of dependent fillers.
\end{proof}

Let $F:\prd{a:A} B(a) \to C(a)$ and let $X:Y\to\UU$ be a type family.
We define an indexed higher inductive type $\factorhit{F}{Y}{X} : Y\to \UU$ by the following constructors:
\begin{align*}
\beta_X &: \prd{y:Y} X(y) \to \factorhit{F}{Y}{X}(y)\\
\mathsf{lift} &: \prd*{a:A}{g:C(a) \to Y}{f:\prd{b:B(a)} \factorhit{F}{Y}{X}(g(F_a(b)))}{c:C(a)} \factorhit{F}{Y}{X}(g(c))\\
\mathsf{islift} &
\!\begin{multlined}[t]
: \prd*{a:A}{g:C(a) \to Y}{f:\prd{b:B(a)} \factorhit{F}{Y}{X}(g(F_a(b)))}{b:B(a)}\\
\mathsf{lift}(g,f,F_a(b)) = f(b).
\end{multlined}
\end{align*}
Diagrammatically, $\mathsf{lift}$ and $\mathsf{islift}$ comprise a specified dependent filler for any judgmentally commutative square as follows:
\[
  \begin{tikzcd}
    B(a) \ar[d,"{F_a}"'] \ar[r,"f"] & \sm{y:Y} \factorhit{F}{Y}{X}(y) \ar[d,"\proj1"] \\
    C(a) \ar[ur,dotted] \ar[r,"g"'] & Y.
  \end{tikzcd}
\]
The induction principle of $\factorhit{F}{Y}{X}$ says that for any $P:\prd{y:Y} \factorhit{F}{Y}{X}(y) \to \UU$ with
\begin{align*}
N &: \prd{y:Y}{x:X(y)} P(y,\beta_X(y,x))\\
R &
\!\begin{multlined}[t]
  : \prd{a:A}{g:C(a) \to Y}{f:\prd{b:B(a)} \factorhit{F}{Y}{X}(g(F_a(b)))}\\
  \prd{f':\prd{b:B(a)} P(g(F_a(b)),f(b))}{c:C(a)} P(g(c),\mathsf{lift}(g,f,c))
\end{multlined}
\\
S &
\!\begin{multlined}[t]
  : \prd{a:A}{g:C(a) \to Y}{f:\prd{b:B(a)} \factorhit{F}{Y}{X}(g(F_a(b)))}\\
  \prd{f':\prd{b:B(a)} P(g(F_a(b)),f(b))}{b:B(a)} \dpath{P}{\mathsf{islift}(g,f,b)}{R(g,f,f',F_a(b))}{f'(b)}
\end{multlined}
\end{align*}
there is a section $s:\prd{y:Y}{w:\factorhit{F}{Y}{X}(y)} P(y,w)$ such that $s \circ \beta_X = N$ (plus two more computation rules we ignore).
Note that by transporting along $\mathsf{islift}$, the types of $R$ and $S$ are equivalent to
\begin{align*}
  R' &
  \!\begin{multlined}[t]
    : \prd{a:A}{g:C(a) \to Y}{f:\prd{b:B(a)} \factorhit{F}{Y}{X}(g(F_a(b)))}\\
    \prd{f':\prd{b:B(a)} P(g(F_a(b)),\mathsf{lift}(g,f,F_a(b)))}{c:C(a)} P(g(c),\mathsf{lift}(g,f,c))
  \end{multlined}
  \\
  S' &
  \!\begin{multlined}[t]
    : \prd{a:A}{g:C(a) \to Y}{f:\prd{b:B(a)} \factorhit{F}{Y}{X}(g(F_a(b)))}\\
    \prd{f':\prd{b:B(a)} P(g(F_a(b)),\mathsf{lift}(g,f,F_a(b)))}{b:B(a)} \id{R(g,f,f',F_a(b))}{f'(b)}.
  \end{multlined}
\end{align*}
With this modification, the inputs of the induction principle are a judgmentally commutative square
\begin{equation}
  \begin{tikzcd}
    \sm{y:Y} X(y) \ar[d,"{(\idfunc[Y],\beta_X)}"'] \ar[r,"N"] & \sm{y:Y}{w:\factorhit{F}{Y}{X}(y)} P(y,w) \ar[d,"\proj1"] \\
    \sm{y:Y} \factorhit{F}{Y}{X}(y) \ar[r,equals] &\sm{y:Y} \factorhit{F}{Y}{X}(y)
  \end{tikzcd}%
  \label{eq:Nsq}
\end{equation}
together with a specified dependent filler for each judgmentally commutative square of the form
\[
  \begin{tikzcd}[column sep=huge]
    B(a) \ar[rr,"{(g\circ F_a,\mathsf{lift}(g,f,F_a(-)),f')}"] \ar[d,"{F_a}"'] && \sm{y:Y}{w:\factorhit{F}{Y}{X}(y)} P(y,w) \ar[d,"\proj1"] \\
    C(a) \ar[rr,"{(g,\mathsf{lift}(g,f,-))}"'] && \sm{y:Y} \factorhit{F}{Y}{X}(y),
  \end{tikzcd}
\]
while the output of the induction principle is a dependent filler in~\eqref{eq:Nsq}.

\begin{lemma}\label{thm:appx-factsys}
  If $P:\prd{y:Y} \factorhit{F}{Y}{X}(y) \to \UU$ is such that
  \[\proj1 : (\sm{y:Y}{w:\factorhit{F}{Y}{X}} P(y,w)) \to \sm{y:Y} \factorhit{F}{Y}{X}\]
  is right orthogonal to $F$, then
  \[(-\circ \beta_X) : \Big(\prd{y:Y}{w:\factorhit{F}{Y}{X}(y)} P(y,w)\Big) \to \Big(\prd{y:Y}{x:X(y)} P(y,\beta_X(x))\Big) \]
  is an equivalence.
\end{lemma}
\begin{proof}
  As in \cref{thm:appx-loc}, we will show that it is path-split using the induction principle of $\factorhit{F}{Y}{X}$.

  First, given $h:\prd{y:Y}{x:X(y)} P(y,\beta_X(x))$, we take $P(y,w) \defeq P(y,w)$ and $N\defeq h$.
  To give the remaining data $R,S$, suppose given
  \begin{mathpar}
  a:A
  \and
  g:C(a) \to Y
  \and
  f:\prd{b:B(a)} \factorhit{F}{Y}{X}(g(F_a(b)))
  \and
  f':\prd{b:B(a)} P(g(F_a(b)),f(b))
  \end{mathpar}
  Now we can apply \cref{thm:orth-dep} with $l\defeq F_a$ and $f\defeq f'$: an inhabitant of~\eqref{eq:dep-fillers} consists exactly of the desired $R$ and $S$.

  Second, given $h,k:\prd{y:Y} (\factorhit{F}{Y}{X}(y) \to P(y))$ and $p:h\circ \beta_X = k\circ \beta_X$, we take $P(y,x) \defeq (h(y,x)=k(y,x))$ and $N\defeq p$.
  To give $R,S$, suppose given $a:A$, $g:C(a) \to Y$, $f:\prd{b:B(a)} \factorhit{F}{Y}{X}(g(F_a(b)))$, and
  \[f':\prd{b:B(a)} h(g(F_a(b)),f(b))=k(g(F_a(b)),f(b)).\]
  Define
  \begin{align*}
    j(c) &\defeq h(g(c),\mathsf{lift}(g,f,c))\\
    j'(c) &\defeq k(g(c),\mathsf{lift}(g,f,c))\\
    q(b) &\defeq k(g(F_a(b)),f(b)).
  \end{align*}
  Then we can apply \cref{thm:orth-dep} to the square
  \[
  \begin{tikzcd}
    B(a) \ar[d,"F_a"'] \ar[r,"q"] & \sm{y:Y} P(y) \ar[d,"\proj1"] \\
    C(a) \ar[r,"g"'] & Y.
  \end{tikzcd}
  \]
  We have
  \[ j'(F_a(b)) \jdeq k(g(F_a(b)),\mathsf{lift}(g,f,F_a(b))) = k(g(F_a(b)),f(b)) \jdeq q(b) \]
  and
  \begin{multline*}
    j(F_a(b)) \jdeq h(g(F_a(b)),\mathsf{lift}(g,f,F_a(b))) = h(g(F_a(b)),f(b)) \overset p= k(g(F_a(b)),f(b)) \jdeq q(b),
  \end{multline*}
  giving two inhabitants $(j,\nameless)$ and $(j',\nameless)$ of~\eqref{eq:dep-fillers}, which are therefore equal.
  This equality consists of an equality $j=j'$, which gives precisely $R$, and an equality between the above two paths, which gives precisely $S$.
\end{proof}

\begin{thm}\label{thm:ofs}
  Given $F:\prd{a:A} B(a) \to C(a)$, define $\cR = F^{\perp}$ and $\cL = {}^{\perp}\cR$, and let $\hat F$ be as in \cref{sec:localizing} and $\factorhit{\hat F}{Y}{X}$ constructed as above for $\hat F$.
  Then for any $X:Y\to\UU$, the composite
  \[ \Big(\sm{y:Y} X(y)\Big) \to \Big(\sm{y:Y} \factorhit{\hat F}{Y}{X}(y)\Big) \to Y \]
  is an $(\cL,\cR)$-factorization.
  Therefore, $(\cL,\cR)$ is an orthogonal factorization system.
\end{thm}
\begin{proof}
  By \cref{thm:dep-precomp-codiag}, if $\proj1$ is right orthogonal to $F$, then it is also right orthogonal to $\hat F$.
  Since every function is equivalent to one of the form $\proj1$, we have $F^{\perp} = {\hat F}^{\perp}$.
  Thus, since applying \cref{thm:appx-factsys} to $\hat F$ shows that the first factor of this factorization is in ${}^{\perp}({\hat F}^{\perp})$, it is also in ${}^{\perp}({F}^{\perp}) = \cL$.

  On the other hand, the constructors $\mathsf{lift}$ and $\mathsf{islift}$ show that the second factor $\proj1 : \big(\sm{y:Y} \factorhit{\hat F}{Y}{X}(y)\big) \to Y$ of this factorization satisfies \cref{thm:orth-dep}\ref{item:od4} for $F$, since the fibers of these maps are the types of dependent fillers against morphisms in $\hat F$.
  Thus, this second factor is in $\cR$.

  Finally, in \cref{sec:modal-refl-subun} we defined orthogonal factorization systems by the uniqueness of factorizations and proved from this the orthogonality of the two classes of maps; but it is easy to show that, as in classical category theory, orthogonality implies the uniqueness of factorizations when they exist, since any two factorizations must lift uniquely against each other.
\end{proof}

\section{Left exact modalities}\label{sec:left-exact-modal}

We have seen that the modal operator of any reflective subuniverse preserves products, but even for a modality it does not generally preserve pullbacks.
If it does, we call the modality ``left exact'' or just ``lex''.

In higher topos theory, lex modalities coincide with reflective sub-toposes.
We can construct them by nullifying any family of propositions (\autoref{thm:prop-loc-lex}); these correspond categorically to the ``topological'' localizations (in 1-topos theory, every subtopos is topological).

\subsection{Lex, topological, and cotopological modalities}%
\label{sec:lex-top-cotop}

\begin{thm}\label{thm:lex-modalities}
  For a modality $\modal$, the following are equivalent.
  \begin{enumerate}
  \item If $A$ is $\modal$-connected, then so is $(x=y)$ for any $x,y:A$.\label{item:mu0}
  \item Whenever $A$ and $\sm{x:A}B(x)$ are $\modal$-connected, then so is $B(x)$ for all $x:A$.\label{item:mu1}
  \item Any map between $\modal$-connected types is $\modal$-connected.\label{item:mu1a}
  \item Any $\modal$-modal function between $\modal$-connected types is an equivalence.\label{item:mu1b}
  \item If $f:A\to B$ is $\modal$-connected, and $g:\prd{a:A} P(a) \to Q(f(a))$ is such that $\total g:(\sm{x:A} P(x)) \to (\sm{y:B} Q(y))$ is $\modal$-connected, then $g_a:P(a)\to Q(fa)$ is also $\modal$-connected for each $a:A$.\label{item:mu3b}
  \item Given a commutative square
    \begin{equation}
      \begin{tikzcd}
        B \ar[r,"h"] \ar[d,"g"'] & A \ar[d,"f"] \\
        D \ar[r,"k"'] & C
      \end{tikzcd}%
      \label{eq:lex-commsq}
    \end{equation}
    in which $f$ and $g$ are $\modal$-connected, then for any $a:A$ the induced map $\hfib{h}{a} \to \hfib{k}{f(a)}$ is $\modal$-connected.\label{item:mu3c}
  \item Any commutative square~\eqref{eq:lex-commsq} in which $f$ and $g$ are $\modal$-connected and $h$ and $k$ are $\modal$-modal is a pullback.\label{item:mu3d}
  \item For any $f:A\to B$ and $b:B$, the evident map $\hfib{f}{b} \to \hfib{\modal f}{\modalunit b}$ is $\modal$-connected.\label{item:mu3a}
  \item For any $A$ and $x,y:A$, the induced map $\modal(x=y) \to (\modalunit[A](x) = \modalunit[A](y))$ is an equivalence.\label{item:mu6}
  \item The functor $\modal$ preserves pullbacks.\label{item:mu3}
  \item $\modal$-connected maps satisfy the 2-out-of-3 property.\label{item:mu4}
  \item If $\modal f: \modal A\to \modal B$ is an equivalence, then $f$ is $\modal$-connected.\label{item:mu5}
  \item For any $\modal$-connected type $A$ and any $P:A\to \modaltype$, there is a $Q:\modaltype$ such that $P(a)\eqvsym Q$ for all $a:A$.\label{item:mu2}
  \end{enumerate}
  When they hold, we say that $\modal$ is \define{lex}.
\end{thm}

\begin{proof}
  The equivalence~\ref{item:mu1}$\Leftrightarrow$\ref{item:mu1a} is easy, using the definition of $\modal$-connected maps and the fact that any function is equivalent to a fibration.
  And~\ref{item:mu0}$\Rightarrow$\ref{item:mu1a} since
  \[
    \hfib f b \jdeq \sm{a:A} (f(a)=b)
\]
and $\modal$-connected types are closed under $\Sigma$ (since $\modal$-connected maps are closed under composition, being the left class of a factorization system).

  Condition~\ref{item:mu1b} is a special case of~\ref{item:mu1a}, since a function that is both modal and connected is an equivalence.
  But assuming~\ref{item:mu1b}, if $f:A\to B$ is any function between $\modal$-connected types, then in its $(\cL,\cR)$-factorization $A\xrightarrow{e} I\xrightarrow{m} B$ the type $I$ is also connected by right cancellation.
  Thus~\ref{item:mu1b} implies that $m$ is an equivalence; thus $f$, like $e$, is $\modal$-connected, giving~\ref{item:mu1a}.

  Assuming~\ref{item:mu1a}, in the situation of~\ref{item:mu3b} the $3\times 3$ lemma for fiber sequences allows us to identify the fiber of $g_a$ over $q:Q(f(a))$ with the fiber over $(a,\refl{f(a)})$ of the induced map $\hfib{\total{g}}{(f(a),q)} \to \hfib{f}{f(a)}$:
  \begin{equation}
  \begin{tikzcd}[column sep=large]
  \bullet \arrow[r] \arrow[d]
    & P(a) \arrow[r] \arrow[d]
    & Q(f(a)) \arrow[d] \\
  \hfib{\total{g}}{(f(a),q)} \arrow[r] \arrow[d]
    & \sm{x:A}P(x) \arrow[r,"{\total{g}}"] \arrow[d,swap,"{\proj1}"]
    & \sm{y:B}Q(y) \arrow[d,"{\proj1}"] \\
  \hfib{f}{f(a)} \arrow[r]
    & A \arrow[r,swap,"f"]
    & B.
  \end{tikzcd}%
  \end{equation}
  Since $f$ and $\total g$ are $\modal$-connected by assumption, their fibers are $\modal$-connected, and hence by~\ref{item:mu1a} so is this fiber; thus~\ref{item:mu3b} holds.

  Now assuming~\ref{item:mu3b}, we can deduce~\ref{item:mu3c} by replacing the maps $h$ and $k$ by equivalent dependent projections.
  If in addition $h$ and $k$ are $\modal$-modal, then $\hfib{h}{a} \to \hfib{k}{fa}$ is a function between $\modal$-modal types, hence itself $\modal$-modal as well as $\modal$-connected and thus an equivalence; thus~\ref{item:mu3c}$\Rightarrow$\ref{item:mu3d}.
  On the other hand, the special case of~\ref{item:mu3d} in which $f$ and $g$ have codomain $\unit$ reduces to~\ref{item:mu1b}.

  Applying~\ref{item:mu3c} instead to the commutative square
  \begin{equation}
  \begin{tikzcd}
  A \arrow[r,"{\modalunit[A]}"] \arrow[d,swap,"f"]
    & \modal(A) \arrow[d,"\modal(f)"] \\
  B \arrow[r,swap,"{\modalunit[B]}"]
    & \modal(B)
  \end{tikzcd}
  \end{equation}
  for any $f:A\to B$ yields~\ref{item:mu3a}.
  And as a special case of~\ref{item:mu3a}, if $A\defeq \unit$ and $B$ is $\modal$-connected, we find that $\apfunc{\modalunit}$ is $\modal$-connected.
  Since $\modal$-connected maps are inverted by $\modal$, this implies~\ref{item:mu6}.
  Conversely, if~\ref{item:mu6} holds, if $A$ is $\modal$-connected then $(\modalunit(x)=\modalunit(y))$ is contractible, hence $(x=y)$ is $\modal$-connected, giving~\ref{item:mu0}.
  Thus~\ref{item:mu0} through~\ref{item:mu6} are equivalent.

  Assuming these equivalent conditions, for a cospan $A\xrightarrow{f}C \xleftarrow{g} B$ the map of pullbacks
  \begin{equation}
    \sm{a:A}{b:B} (fa=gb) \longrightarrow \sm{x:\modal A}{y:\modal B} ((\modal f)(x) = (\modal g)(y))\label{eq:pbpres}
  \end{equation}
  is equivalent to the map on total spaces induced by $\modalunit[A]:A\to\modal A$ and the fiberwise transformation
  \[ h : \prd{a:A} \left(\hfib{g}{fa} \to \hfib{\modal g}{(\modal f)(\modalunit a)}\right). \]
  But since $(\modal f)(\modalunit a) = \modalunit(fa)$, by~\ref{item:mu3a} each $h_a$ is $\modal$-connected.
  Since $\modalunit[A]$ is also $\modal$-connected, by \cref{lem:nconnected_postcomp_variation} so is~\eqref{eq:pbpres}.
  Hence the induced map
  \[ \modal\left(\sm{a:A}{b:B} (fa=gb)\right) \longrightarrow \sm{x:\modal A}{y:\modal B} ((\modal f)(x) = (\modal g)(y))\]
  (which exists since the codomain is $\modal$-modal) is an equivalence, yielding~\ref{item:mu3}.

  On the other hand, if~\ref{item:mu3}, then $\modal$ preserves any pullback
  \begin{equation}
  \begin{tikzcd}
  (x=y) \arrow[r] \arrow[d]
    & \unit \arrow[d,"x"] \\
  \unit \arrow[r,"y"']
    & A
  \end{tikzcd}
  \end{equation}
  yielding~\ref{item:mu6}.

  For~\ref{item:mu4}, two-thirds of the 2-out-of-3 property holds for any modality, so it remains to show that for $f:A\to B$ and $g:B\to C$, if $g\circ f$ and $g$ are $\modal$-connected, so is $f$.
  However, the unstable octahedral axiom~(\cite[ex4.4]{TheBook}) implies that for any $b:B$, the fiber $\hfib f b$ is equivalent to the fiber of the induced map $\hfib{g\circ f}{gb} \to \hfib{g}{gb}$.
  These two types are $\modal$-connected since $g\circ f$ and $g$ are; thus~\ref{item:mu1a}$\Rightarrow$\ref{item:mu4}.
  Conversely,~\ref{item:mu1a} is clearly a special case of~\ref{item:mu4}.

  Since $\modalunit[A] : A\to \modal A$ is $\modal$-connected, easily~\ref{item:mu4}$\Rightarrow$\ref{item:mu5}.
  On the other hand, if $g\circ f$ and $g$ are $\modal$-connected, then they are both inverted by $\modal$, and hence so is $f$; thus~\ref{item:mu5}$\Rightarrow$\ref{item:mu4}.

  Next we assume~\ref{item:mu4} and show~\ref{item:mu2}.
  Suppose $A$ is $\modal$-connected and $P:A\to\modaltype$, and define
  \[ Q \defeq \modal\left(\sm{a:A} P(a)\right),\]
  and $g:\prd{a:A} P(a) \to Q$ by $g(a,u) \defeq \modalunit(a,u)$.
  We will show $g$ to be a family of equivalences.

  Since $P(a)$ and $Q$ are both $\modal$-modal, for $g_a$ to be an equivalence, it suffices for it to be $\modal$-connected.
  We will prove this by showing that the induced map $\total g:(\sm{a:A} P(a)) \to (\sm{a:A} Q)$ is $\modal$-connected.
  By the assumed 2-out-of-3 property, for this it suffices to show that the other two maps in the following commutative triangle are $\modal$-connected:
  \begin{equation}
  \begin{tikzcd}[column sep=large]
  \sm{a:A} P(a) \arrow[r,"{\total g}"] \arrow[dr,swap,"{\modalunit}"]
    & \sm{a:A} Q \mathrlap{\;\jdeq A\times Q} \arrow[d,"\proj2"] \\
  & Q
  \end{tikzcd}
  \end{equation}
  But the right-hand vertical map is $\modal$-connected since its fiber is the $\modal$-connected type $A$,
  and the diagonal map is $\modal$-connected since it is simply $\modalunit$.
  This completes the proof of~\ref{item:mu4}$\Rightarrow$\ref{item:mu2}.

  Finally, we prove~\ref{item:mu2}$\Rightarrow$\ref{item:mu0}.
  Suppose $A$ is $\modal$-connected and $x:A$.
  Then $\lam{y} \modal(x=y) : A \to \modaltype$ so there is a $Q_x:\modaltype$ such that $\modal(x=y)\eqvsym Q_x$ for all $y:A$.
  It follows that transport in the type family $\lam{y} \modal(x=y)$ is constant, i.e.\ if $p,q:y=z$ and $u:\modal(x=y)$ then $\trans p u = \trans q u$.
  Now for any $p:x=y$, we have $\trans p {\modalunit(\refl x)} = \modalunit(p)$; hence for any $p,q:x=y$ we have $\modalunit(p)=\modalunit(q)$.
  By $\mathsf{ind}^\modal$, it follows that for any $u,v:\modal(x=y)$ we have $u=v$, i.e.\ $\modal(x=y)$ is a mere proposition.
  But $\modal(x=x)$ is inhabited by $\modalunit(\refl x)$, hence $Q_x$ is also inhabited, and thus so is $\modal(x=y)$ for all $y$; thus it is contractible.
\end{proof}

Note that~\ref{item:mu3a} and~\ref{item:mu3} both imply that a lex modality preserves fibers: given $f:A\to B$ and $b:B$, the map $\modal(\hfib{f}{b}) \to \hfib{\modal f}{\modalunit b}$ is an equivalence.
In fact, this property (and hence also~\ref{item:mu3}) characterizes lex modalities even among reflective subuniverses.

\begin{thm}\label{thm:rsu-lex}
  If $\modal$ is a reflective subuniverse such that for any $f:A\to B$ and $b:B$, the map $\modal(\hfib{f}{b}) \to \hfib{\modal f}{\modalunit b}$ is an equivalence, then $\modal$ is $\Sigma$-closed (and hence a lex modality).
\end{thm}
\begin{proof}
  Suppose $A$ and each $B(a)$ are $\modal$-modal.
  We have a commutative square
  \[
  \begin{tikzcd}
    \sm{a:A}B(a) \ar[r,"\modalunit"] \ar[d,"\proj1"'] &
    \modal(\sm{a:A}B(a)) \ar[d,"\modal \proj1"] \\
    A \ar[r,"\modalunit"',"\sim"] & \modal A
  \end{tikzcd}
  \]
  in which the bottom map is an equivalence.
  Thus, to show that the top map is an equivalence it suffices to show that the induced map on each fiber $B(a) \to \hfib{\modal \proj1}{\modalunit a}$ is an equivalence.
  But this map factors through the equivalence $B(a) \eqvsym \modal B(a)$ by the map $\modal B(a) \to \hfib{\modal \proj1}{\modalunit a}$, which is an equivalence by assumption.
\end{proof}

A particularly useful corollary of \cref{thm:lex-modalities} is the following.

\begin{corollary}\label{modaln-truncated}
  A lex modality preserves $n$-truncated maps for all $n$.
\end{corollary}
\begin{proof}
  We first argue by induction on $n$ that a lex modality $\modal$ preserves $n$-types for all $n$.
  The base case is \cref{lem:modal-pres-prop}.
  For the inductive step, suppose $\modal$ is lex and preserves $n$-types, and $A$ is an $(n+1)$-type.
  Then for $u,v:\modal A$ the proposition that $u=v$ is an $n$-type is $\modal$-modal, since it is constructed inductively using $\Sigma$, $\Pi$, and identity types.
  Thus, we can prove it by $\modal$-induction on $u,v$.
  But for $x,y:A$ the type $\modalunit(x)=\modalunit(y)$ is equivalent to $\modal(x=y)$ by \cref{thm:lex-modalities}\ref{item:mu6}, hence is an $n$-type by the inductive hypothesis.

  Now if $f:A\to B$ is $n$-truncated, to show that $\modal f$ is $n$-truncated we must show that $\hfib{\modal f}{y}$ is an $n$-type for all $y:\modal B$.
  Again, by $\modal$-induction we can reduce to the case $y\defeq \modalunit(b)$ for some $b:B$, in which case \cref{thm:lex-modalities}\ref{item:mu3a} implies that $\hfib{\modal f}{\modalunit(b)} \eqvsym \modal(\hfib f b)$, which is an $n$-type since $f$ is $n$-truncated and $\modal$ preserves $n$-types.
\end{proof}

Not every modality satisfying \cref{modaln-truncated} is lex.
For instance, the $m$-truncation modality preserves $n$-types for all $n$, but is not lex for $m\ge -1$.
(To see that it is not lex, consider an Eilenberg--MacLane space $K(G,m+1)$~\cite{lf:emspaces}; this is $m$-connected, but its loop space is $K(G,m)$ which is not $m$-connected. 
Alternatively, we can use \cref{thm:acc-lex} below together with the fact that the universe of $m$-types in the $m^{\mathrm{th}}$ universe is not an $m$-type~\cite{ks:u-not-ntype}.)

We do know at least one example of a lex modality.

\begin{eg}
  For any mere proposition $P$, the open modality $\open P \defeq \lam{X} (P\to X)$ is lex.
  This is easy to see since mapping out of $P$ is a right adjoint, hence preserves all limits, including pullbacks.
\end{eg}

However, constructing lex modalities in general, such as by localization, is somewhat tricky.
Unlike the characterization of modalities as $\Sigma$-closed reflective subuniverses, which refers only to the \emph{modal types} and hence was easy to prove in \cref{thm:nullification_modality}, all the characterizations of lex-ness refer explicitly or implicitly to the \emph{modal operator} $\modal$, and not just by way of its ``mapping out'' universal property but saying something about its identity types.
In general, saying anything about the identity types of a higher inductive type (such as localization) requires some amount of univalence, and the present case is no exception (although we do not need a full ``encode-decode'' type argument).

\begin{thm}\label{thm:acc-lex}
  Let $\modal$ be an accessible modality; the following are equivalent.
  \begin{enumerate}
  \item $\modal$ is lex.\label{item:al1}
  \item $\modal$ has a presentation $B:A\to \UU$ such that for any $a:A$ and any $P:B(a)\to \modaltype$, there is a $Q:\modaltype$ such that $P(b)\eqvsym Q$ for all $b:B(a)$.\label{item:al2}
  \item The universe $\modaltype \defeq \setof{A:\type | A \text{ is $\modal$-modal}}$ of modal types is $\modal'$-modal, where $\modal'$ is the canonical accessible extension of $\modal$ to a universe $\UU'$ containing $\UU$, as in \cref{thm:acc-extend}.\label{item:al3}
  \end{enumerate}
\end{thm}
\begin{proof}
  Assuming~\ref{item:al1}, condition~\ref{item:al2} holds for \emph{any} presentation: it is just a special case of \cref{thm:lex-modalities}\ref{item:mu2}, since each $B(a)$ is $\modal$-connected.

  Now assume~\ref{item:al2} for some presentation $B:A\to\UU$.
  By definition of $\modal'$, it suffices to show that $\modaltype$ is $B(a)$-null for all $a:A$, i.e.\ that the ``constant functions'' map
  \[ \modaltype \to (B(a) \to \modaltype) \]
  is an equivalence for all $a:A$.
  The assumption~\ref{item:al2} says that this map has a section, and hence in particular is surjective.
  Thus, it suffices to show it is an embedding, i.e.\ that for any $X,Y:\modaltype$ the map
  \[ (X=Y) \to ((\lam{b} X)= (\lam{b} Y)) \]
  is an equivalence.
  But by univalence and function extensionality, this map is equivalent to
  \[ (X\eqvsym Y) \to (B(a) \to (X\eqvsym Y)), \]
  which is an equivalence by \cref{connectedtotruncated} since $X\eqvsym Y$ is $\modal$-modal and $B(a)$ is $\modal$-connected.

  Finally, if we assume~\ref{item:al3}, then for any $\modal$-connected type $A:\UU$ the map
  \[ \modaltype \to (A\to\modaltype) \]
  is an equivalence.
  In particular, it has a section, proving \cref{thm:lex-modalities}\ref{item:mu2}.
\end{proof}

In particular, we have the following general result.

\begin{corollary}\label{thm:prop-loc-lex}
  Let $B:A\to\prop$ be a family of mere propositions.
  Then nullification at $B$ is a lex modality.
\end{corollary}
\begin{proof}
  We prove condition~\ref{item:al2} of \cref{thm:acc-lex}.
  Given $P:B(a) \to \modaltype$, define $Q \defeq \prd{b:B(a)} P(b)$.
  This lies in $\modaltype$ since modal types are always closed under dependent function types.
  And if we have any $b:B(a)$, then $B(a)$ is an inhabited proposition and hence contractible,
  and a product over a contractible type is equivalent to any of the fibers.
\end{proof}

\begin{defn}
  A (necessarily lex) modality that can be presented by nullification at a family of mere propositions is called \textbf{topological}.
\end{defn}

The term ``topological'' is from~\cite{lurie2009higher}.
Presumably it comes from the fact that by~\cite[Proposition 6.2.2.17]{lurie2009higher}, topological localizations of a presheaf $\infty$-topos correspond to ``Grothendieck topologies'' on the domain, as defined there.

\begin{eg}
  For any mere proposition $Q$, the closed modality $\closed Q \defeq \lam{X} Q\ast X$ is topological, since it is presented by the family $\lam{x:P} \emptyt$.
  Thus, by \cref{thm:prop-loc-lex}, it is lex.
\end{eg}

Topological modalities may seem very special, since very few types are mere propositions.
But in fact, if we allow ourselves to assume rather than conclude lex-ness, then it doesn't matter what truncation level we take the generating family at, as long as it is finite:

\begin{thm}\label{thm:acc-ntypes-tpl}
  If $\modal$ is an accessible lex modality with a presentation $B:A\to\UU$ for which each $B(a)$ is an $n$-type (for some fixed $n$ independent of $a$), then $\modal$ is topological.
\end{thm}
\begin{proof}
  We will prove that under the given hypotheses, if $n\ge 0$ then $\modal$ also has a presentation $D:C\to \UU$ for which each $D(c)$ is an $(n-1)$-type.
  By induction, this will prove the theorem.
  The argument is a modification of~\cite[Lemma 7.5.11]{TheBook}.

  Let $C\defeq A + \sm{a:A} B(a)\times B(a)$, and define
  \begin{align*}
    D(\inl(a)) &\defeq \strunc{-1}{B(a)}\\
    D(\inr(a,x,y)) &\defeq (x=_{B(a)} y).
  \end{align*}
  Clearly each $D(c)$ is an $(n-1)$-type (here is where we use the assumption $n\ge 0$).
  Since $\modal$ is lex and each $B(a)$ is $\modal$-connected, each $D(\inr(a,x,y))$ is also $\modal$-connected.
  To show that $D(\inl(a))$ is also $\modal$-connected, since $\modal$ preserves mere propositions, the proof of \cref{prop:nconnected_tested_by_lv_n_dependent types} implies that it suffices to show that $Z\to (\brck{B(a)} \to Z)$ is an equivalence for any $\modal$-modal mere proposition $Z$.
  But in this case $(\brck{B(a)} \to Z) \eqvsym (B(a) \to Z)$, and the latter is equivalent to $Z$ since $B(a)$ is $\modal$-connected and $Z$ is $\modal$-modal.

  Thus each type $D(c)$ is $\modal$-connected, so every $\modal$-modal type is $D$-null.
  For the converse, suppose $X$ is $D$-null and let $a:A$.
  We want to show that $X$ is $B(a)$-null, i.e.\ that the ``constant functions'' map $c : X\to (B(a) \to X)$ is an equivalence.
  Let $f:B(a) \to X$; we will show that $\hfib{c}{f}$ is contractible.

  Now $X$ and $B(a)\to X$ are both $D$-null, hence so is $\hfib{c}{f}$, and hence so is the proposition ``$\hfib{c}{f}$ is contractible''.
  Thus, we may assume in proving it that we have $\brck{B(a)}$.
  But it is also a proposition, so we may furthermore assume that we have some $b:B(a)$.

  If we also write $b$ for the induced map $\unit \to B(a)$, then for any $u:B(a)$ we have $\hfib{b}{u} \simeq (b=u)$, which belongs to $D$.
  Thus $b:\unit\to B(a)$ is $D$-connected.

  We construct a point in $\hfib{c}{f}$ by taking $f(b)$ and constructing a path \[p:\prd{u:B(a)} f(u)=f(b).\]
  To give $p$, note that since $X$ is $D$-modal, so is the type $f(u)=f(b)$.
  Thus, by \cref{prop:nconnected_tested_by_lv_n_dependent types}, since $b:\unit\to B(a)$ is $D$-connected, it suffices to prove $f(b)=f(b)$, which is of course trivial.

  Finally, suppose we have some other point $(x,q) : \hfib{c}{f}$, i.e.\ an $x:X$ with $q:\prd{u:B(a)} f(u)=x$.
  Then $q_b : f(b) = x$, so it remains to show that for any $u:B(a)$ we have $q_b = \ct{p_u^{-1}}{q_u}$.
  But since this is an iterated equality type in $X$, it is $D$-modal, so using again the fact that $b:\unit\to B(a)$ is $D$-connected it suffices to prove it when $u=b$.
  But $p_b = \refl{f(b)}$ by definition, so in this case the goal reduces to $q_b = q_b$, which is trivial.
\end{proof}

Thus, a topological modality could equivalently be defined as a lex modality that admits a generating family of bounded homotopy type.
Moreover, \emph{every} lex modality is ``almost topological'' in the following sense.

\begin{thm}\label{thm:lex-ntypes-prop}
  If $\modal$ is a lex modality and $A$ is an $n$-type for $n<\infty$, then $A$ is $\modal$-modal if and only if it is $P$-null for any $\modal$-connected mere proposition $P$.
\end{thm}
\begin{proof}
  ``Only if'' is trivial, so we prove the converse.
  By induction on $n$.
  The base case $n=-2$ is trivial.
  Thus, suppose $A$ is an $(n+1)$-type that is $P$-null for every $\modal$-connected proposition $P$.
  Then for any $x,y:A$, we have a commutative triangle
  \[
  \begin{tikzcd}
    & x=y \ar[dl] \ar[dr] \\
    (P\to (x=y)) \ar[rr] && (\lam{\nameless}x =_{P\to A} \lam{\nameless}y)
  \end{tikzcd}
  \]
  in which the bottom map is an equivalence by function extensionality, and the right-hand diagonal map is an equivalence since it is the action on equalities of the equivalence $A\eqvsym (P\to A)$.
  Thus, the left-hand diagonal map is also an equivalence, so $(x=y)$ is also $P$-null.
  By the inductive hypothesis, therefore, $(x=y)$ is $\modal$-modal.
  Hence by \cref{thm:lex-modalities}\ref{item:mu6}, the map $\modalunit[A]:A\to \modal A$ is an embedding; thus it suffices to show that it is surjective.

  Now given $z:\modal A$, since $\modalunit[A]$ is an embedding, its fiber $\hfib{\modalunit[A]}{z}$ is a mere proposition; and it is $\modal$-connected since $\modalunit[A]$ is connected.
  Thus, by assumption $A \to (\hfib{\modalunit[A]}{z} \to A)$ is an equivalence.
  But we have $\proj1 : \hfib{\modalunit[A]}{z} \to A$, so there exists an $x:A$ such that $\proj1 = \lam{\nameless}x$, i.e.\ for any $y:A$ with $\modalunit (y)=z$ we have $y = x$.

  We claim that $\modalunit(x) = z$.
  This is a modal type, since it is an equality in $\modal A$.
  Thus, since $\hfib{\modalunit[A]}{z}$ is $\modal$-connected, when proving $\modalunit(x) = z$ we may assume that $\hfib{\modalunit[A]}{z}$, i.e.\ we have $y:A$ with $\modalunit (y) = z$.
  But then $y=x$ as shown above, so that $\modalunit (x) = z$ as well.
\end{proof}

Thus, if an accessible lex modality is not topological, it must be generated by a family including $n$-types for arbitrarily high $n$ (or else at least one type that is not an $n$-type for any finite $n$), and moreover its failure to be topological will only be visible to types that are not $n$-types for any finite $n$.
This means that it is rather hard to give examples of lex modalities that are not topological.

Semantically, it is known that not all subtoposes of $(\infty,1)$-toposes are topological, so by the results of \cref{sec:semantics} non-topological lex modalities do exist in some models.
The basic example is the \emph{hypercompletion}.
We do not know how to construct hypercompletion inside type theory, but we can show that if it exists then it is lex, and not topological unless it is trivial.
We begin with definitions.

\begin{defn}
  A type $A$ or a function $f:A\to B$ is \define{$\infty$-connected} if it is $n$-connected for all $n$.
\end{defn}

Recall that if $f$ is $n$-connected for fixed $n$, then $\trunc n f$ is an equivalence, but the converse may not hold.
However, $\trunc {n+1} f$ being an equivalence is sufficient for $f$ to be $n$-connected, and so $f$ is $\infty$-connected if and only if $\trunc n f$ is an equivalence for all $n$.
Similarly, a type $A$ is $\infty$-connected if and only if $\trunc n A$ is contractible for all $n$.
Note that since a map is $n$-connected if and only if all its fibers are, a map is likewise $\infty$-connected if and only if all its fibers are.

\begin{defn}
  A type $Z$ is \define{$\UU$-$\infty$-truncated} or \define{$\UU$-hypercomplete} if it is local with respect to all $\infty$-connected maps in $\UU$, i.e.\ if $(\blank\circ f):(C\to Z) \to (B\to Z)$ is an equivalence whenever $f:B\to C$ is $\infty$-connected with $B,C:\UU$.
\end{defn}

In general, it is not clear to what extent the notion of $\UU$-$\infty$-truncatedness depends on $\UU$.
However, if $Z$ is an $n$-type for some $n<\infty$, then $(\blank\circ f)$ is equivalent to $(\blank\circ \trunc nf)$, which is an equivalence if $f$ is $\infty$-connected.
Thus, any $n$-type is $\infty$-truncated independent of universe level.
In particular, this implies:

\begin{lemma}\label{thm:infconn}
  Given $B,C:\UU$ and $f:B\to C$, the following are equivalent.
  \begin{enumerate}
  \item $f$ is $\infty$-connected.\label{item:ic1}
  \item $(-\circ f):(C\to Z) \to (B\to Z)$ is an equivalence for all $\UU$-$\infty$-truncated $Z:\UU$.\label{item:ic2}
  \item $(-\circ f):(C\to Z) \to (B\to Z)$ is an equivalence for all $n$-types $Z:\UU$.\label{item:ic3}
  \end{enumerate}
\end{lemma}
\begin{proof}
  We have~\ref{item:ic1}$\Rightarrow$\ref{item:ic2} by definition of ``$\UU$-$\infty$-truncated'', and~\ref{item:ic2}$\Rightarrow$\ref{item:ic3} by the above remarks.
  Now assuming~\ref{item:ic3}, the universal property of $n$-truncation tells us that
  \[ (\blank\circ \trunc nf):(\trunc n C\to Z) \to (\trunc nB\to Z) \]
  is an equivalence for any $n$-type $Z$.
  By the Yoneda lemma, this implies that $\trunc n f$ is an equivalence for all $n$; hence $f$ is $\infty$-connected.
\end{proof}

The closure of $\infty$-connectedness under fibers also implies:

\begin{lemma}\label{thm:trunc-null}
  A type $Z$ is $\UU$-$\infty$-truncated if and only if it is null with respect to all $\infty$-connected types in $\UU$, i.e.\ if $Z \to (B\to Z)$ is an equivalence whenever $B:\UU$ is $\infty$-connected.
\end{lemma}
\begin{proof}
  ``Only if'' is clear, so suppose the given condition holds and let $f:A\to B$ be $\infty$-connected with $A,B:\UU$.
  Then we have
  \begin{align*}
    (A \to Z)
    &\simeq (\sm{b:B}\hfib{f}{b}) \to Z\\
    &\simeq \prd{b:B} (\hfib{f}{b} \to Z)\\
    &\simeq \prd{b:B} Z\\
    &\jdeq (B\to Z).\qedhere
  \end{align*}
\end{proof}

Now, we can certainly localize at all the $\infty$-connected maps in $\UU$ to obtain a reflective subuniverse of any \emph{larger} universe $\UU'$ whose modal types are the $\UU$-$\infty$-truncated ones.
However, hypercompletion should really be a modality on $\UU$ \emph{itself} whose modal types are the $\UU$-$\infty$-truncated ones.
A local presentability argument in~\cite[Prop.~6.5.2.8]{lurie2009higher} shows that in any Grothendieck $\infty$-topos there exists a small family that generates such a modality by localization.
But in type theory, the best we can do at present is show that \emph{if} such a modality exists, then it behaves as expected.

\begin{thm}\label{thm:hypercompletion}
  Suppose $\modal$ is a reflective subuniverse on $\UU$ whose modal types are precisely the $\UU$-$\infty$-truncated ones.
  Then:
  \begin{enumerate}
  \item $\modal$ is a lex modality.\label{item:hc1}
  \item The $\modal$-connected maps are precisely the $\infty$-connected ones.\label{item:hc2}
  \item $\modal$ is topological if and only if every type is $\modal$-modal, i.e.\ every type is $\UU$-$\infty$-truncated, i.e.\ ``Whitehead's principle''~\cite[\S8.6]{TheBook} holds for $\UU$.\label{item:hc3}
  \end{enumerate}
  If such a modality exists, we call it \textbf{hypercompletion}.
\end{thm}
\begin{proof}
  The proof of \cref{thm:nullification_modality} shows that the $B$-null types for any type family $B$ are $\Sigma$-closed, regardless of whether or not $B$ is small.
  Thus, \cref{thm:trunc-null} shows that the $\UU$-$\infty$-truncated types are $\Sigma$-closed, hence $\modal$ is a modality.

  Next we prove~\ref{item:hc2}.
  By \cref{thm:infconn}, any $\modal$-connected map is $\infty$-connected.
  Conversely, if $f:A\to B$ is $\infty$-connected, then any fiber $\hfib{f}{b}$ is also $\infty$-connected.
  Thus for any $\modal$-modal type $Z$ we have $Z \eqvsym (\hfib{f}{b}\to Z)$; hence $\hfib{f}{b}$ is $\modal$-connected, and thus so is $f$.

  This shows~\ref{item:hc2}.
  Now the lex-ness of $\modal$ follows from the fact that $\infty$-connected maps satisfy the 2-out-of-3 property, since $f$ is $\infty$-connected if and only if each $\trunc n f$ is an equivalence, and equivalences satisfy the 2-out-of-3 property.

  Finally, if $\modal$ is topological, then there is a family $B:A\to \prop$ of mere propositions that generates it.
  In particular, each $B(a)$ must then be $\modal$-connected, and hence $\infty$-connected.
  But a mere proposition is a $(-1)$-type, hence also $\infty$-truncated.
  Thus each $B(a)$ is contractible, so that every type is $\modal$-modal.
\end{proof}

More generally, we have the following analogue of~\cite[Proposition 6.5.2.16]{lurie2009higher}:

\begin{thm}\label{thm:cotop}
  For a lex modality $\modal$, the following are equivalent:
  \begin{enumerate}
  \item Every $\modal$-connected mere proposition is contractible.\label{item:ct1}
  \item Every $\modal$-connected map is $\infty$-connected.\label{item:ct2}
  \item Every $\UU$-$\infty$-truncated type is $\modal$-modal.\label{item:ct3}
  \end{enumerate}
  In this case we say $\modal$ is \textbf{cotopological}.
\end{thm}
\begin{proof}
  Using \cref{thm:infconn}, we have~\ref{item:ct2}$\Leftrightarrow$\ref{item:ct3}.
  And an $\infty$-connected mere proposition is contractible, so~\ref{item:ct2}$\Rightarrow$\ref{item:ct1}.
  Conversely, assuming~\ref{item:ct1}, by \cref{thm:lex-ntypes-prop} every $n$-type is $\modal$-modal; hence \cref{thm:infconn} yields~\ref{item:ct2}.
\end{proof}

\begin{rmk}\label{thm:nontop-lex}
  At the time this paper was written, we did not know any ``small'' condition on a family $B:A\to\UU$ ensuring that the modality it generates is lex and such that every lex modality can be generated by such a family.
  (\cref{thm:acc-lex}\ref{item:al2} is not ``small'' because it refers to arbitrary families of modal types.)
  However, as we were preparing it for final publication,~\cite{abfj:lexloc} found two such conditions:
  \begin{enumerate}[label={(\alph*)}]
  \item For all $a:A$ and $x,y:B(a)$ the type $x=y$ is $\modal$-connected (a relative version of \cref{thm:lex-modalities}\ref{item:mu0}).\label{item:abfj1}
  \item For all $a:A$ and $x,y:B(a)$ there is an $a':A$ with $B(a') \eqvsym (x=y)$.\label{item:abfj2}
  \end{enumerate}
  Clearly~\ref{item:abfj2} implies~\ref{item:abfj1}, while any $B$ satisfying~\ref{item:abfj1} can be enhanced to one satisfying~\ref{item:abfj2} by closing it up under path spaces.
  The nontrivial part is showing that~\ref{item:abfj1} implies \cref{thm:acc-lex}\ref{item:al2}.

  In particular, this characterization implies that if $\modal$ is an accessible lex modality on $\UU$, then its canonical accessible extension $\modal'$ to a larger universe $\UU'$ from \cref{thm:acc-extend} is again lex, since whether a generating family satisfies~\ref{item:abfj1}--\ref{item:abfj2} is independent of universe level.
  Without such a characterization, we could only conclude this when $\modal$ is topological.
\end{rmk}

\begin{rmk}\label{thm:subtopos-model}
  The modal types for an accessible lex modality are closed under identity types (by \cref{lem:rs_idstable}), $\Pi$-types (by \cref{lem:modal-Pi}), $\Sigma$-types (since it is a modality), and universes (by \cref{thm:acc-lex}).
  Thus, they are in their own right a model of the fragment of homotopy type theory containing only these type operations (the internal language of a subtopos).

  The modal types are not closed under other type formers like $\emptyt$, $A+B$, the natural numbers, and more general inductive and higher inductive types.
  However, if $F$ is a presentation of $\modal$, then we can construct a version of any higher inductive type $\mathsf{H}$ that is $\modal$-modal and satisfies the induction principle with respect to other modal types, by adding the second two constructors of $\localization{F}$ to the given constructors of $\mathsf{H}$, yielding a new higher inductive type that is ``$F$-local by definition''.
  (This is a sort of internal version of the algebraic fibrant replacement used semantically in~\cite{ls:hits}.)
  The fact that localization modalities have judgmental computation rules ensures that these ``local higher inductive types'' do too.
  Thus, the subtopos model inherits higher inductive types as well.

  In principle, this sort of construction could reduce the problem of modeling homotopy type theory with strict univalent universes in all $(\infty,1)$-toposes to the problem of modeling it in presheaf $(\infty,1)$-toposes, since every $(\infty,1)$-topos is (by one definition) an accessible left exact localization of a presheaf $(\infty,1)$-topos.
  However, in order for this to work we need strict univalent universes that are strictly closed under the modality, and in general we do not know how to ensure this semantically; see \cref{rmk:universes}.
  A similar construction of a subtopos model, but using a strict monad, can be found in~\cite{Stacks,Coquand:stack}.
\end{rmk}

\subsection{Meets and joins of modalities}%
\label{sec:poset-modalities}

Let $\rsu[\UU]$ denote the type of reflective subuniverses of a universe $\UU$, and similarly $\mdl[\UU]$, $\lex[\UU]$, and $\tpl[\UU]$ the types of modalities, lex modalities, and topological modalities, while $\accrsu$, $\accmdl$, and $\acclex$ consist of accessible ones.
Each of these is partially ordered by inclusion, i.e.\ $\modal \le \lozenge$ means that every $\modal$-modal type is $\lozenge$-modal, and we have full inclusions
\[
\begin{tikzcd}
  \tpl \ar[r] \ar[dr] & \acclex \ar[r] \ar[d] & \accmdl \ar[r] \ar[d] & \accrsu \ar[d] \\
  & \lex \ar[r] & \mdl \ar[r] & \rsu.
\end{tikzcd}
\]
The poset $\rsu$ has both a bottom element (the zero modality, for which only $\unit$ is modal) and a top element (the trivial modality, for which all types are modal), which both happen to lie in $\tpl$ and hence all of these other posets.
It is natural to wonder whether these posets have other lattice structure.
We do not have a complete answer, but there are some things we can say.

\begin{thm}\label{thm:meet-join}
  Suppose given any family $\modal_i$ of reflective subuniverses.
  \begin{enumerate}
  \item If there is a reflective subuniverse $\lozenge$ such that a type is $\lozenge$-modal if and only if it is $\modal_i$-modal for all $i$, then $\lozenge$ is the meet $\bigwedge_i \modal_i$ in $\rsu$.
    Moreover, if each $\modal_i$ is a modality, then so is $\lozenge$, and it is also the meet in $\mdl$.\label{item:mj1}
  \item If each $\modal_i$ is a modality, and there is a modality $\lozenge$ such that a type is $\lozenge$-connected if and only if it is $\modal_i$-connected for all $i$, then $\lozenge$ is the join $\bigvee_i \modal_i$ in $\mdl$.\label{item:mj2}
  \item If there is a reflective subuniverse $\lozenge$ such that for any function $f:A\to B$, we have that $\lozenge(f)$ is an equivalence if and only if $\modal_i(f)$ is an equivalence for all $i$, then $\lozenge$ is the join $\bigvee_i \modal_i$ in $\rsu$.\label{item:mj3}
  \end{enumerate}
\end{thm}
\begin{proof}
  The first part of statement~\ref{item:mj1} follows from the fact that the ordering on reflective subuniverses is determined by inclusion of the universes of modal types.
  The second follows since $\Sigma$-closure of such universes is inherited by intersections.

  The other two statements are instances of a general fact about Galois connections.
  Suppose $G: \mathcal{B}^{\mathrm{op}} \leftrightarrows \mathcal{A} : H$ is a contravariant adjunction between posets, i.e.\ $G$ and $H$ are contravariant functors and $b \le G a \iff a \le H b$.
  Then $(G,H)$ restricts to a contravariant isomorphism between the posets of fixed points $\mathcal{A}^{GH}$ and $\mathcal{B}^{HG}$ for the monads $G H$ and $H G$.
  Moreover, any meets in $\mathcal{B}$ are inherited by $\mathcal{B}^{HG}$, hence also by ${(\mathcal{A}^{GH})}^{\mathrm{op}}$, i.e.\ are joins in $\mathcal{A}^{GH}$.

  In the simpler case of~\ref{item:mj2}, let $\mathcal{A}$ and $\mathcal{B}$ both be the set $\UU \to \prop$ of subtypes of the universe, let $G(\mathcal{E})$ be the set of types $A$ such that $A \to (B\to A)$ is an equivalence for all $B\in \mathcal{E}$, and likewise let $H(\mathcal{M})$ be the set of types $B$ such that $A \to (B\to A)$ is an equivalence for all $A\in \mathcal{M}$.
  Then by \cref{connectedtotruncated,thm:detect-right-by-fibers}, the $\modal$-modal types for any modality are a fixed point of $GH$, and the $\modal$-connected types are the corresponding fixed point of $HG$.
  Not every such fixed point is a modality, but it does follow that if a meet in $\mathcal{A}^{HG}$, i.e.\ an intersection of the universes of $\modal_i$-connected types, is the $\lozenge$-connected types for some modality $\lozenge$, then it is a join in the dual poset of modalities.

  Case~\ref{item:mj3} is similar, using the same $\mathcal{A}$ but taking $\mathcal{B}$ to be the set $\prd{X,Y:\UU} (X\to Y) \to \prop$ of subtypes of the type of all functions in the universe, letting $G(\mathcal{E})$ be the set of types $X$ such that $(\blank\circ f) : (B\to X) \to (A\to X)$ is an equivalence for all $f:A\to B$ in $\mathcal{E}$, and dually $H(\mathcal{M})$ the set of functions $f:A\to B$ such that $(\blank\circ f) : (B\to X) \to (A\to X)$ is an equivalence for all $X\in \mathcal{M}$.
  Then the $\modal$-modal types for any reflective subuniverse are a fixed point of $GH$, since the universal property of $\modal$ tells us that $(\blank\circ f) : (B\to X) \to (A\to X)$ is an equivalence for all modal $X$ if and only if $\modal f$ is an equivalence, and \cref{thm:rsu-galois} tells us that we can detect modal types by mapping out of such functions.
  The same argument then applies to the dual classes of $\modal$-inverted functions.
\end{proof}

When the conditions of \cref{thm:meet-join}\ref{item:mj1} hold, we say that $\lozenge$ is the \textbf{canonical meet} of the $\modal_i$'s, and dually in cases~\ref{item:mj2} and~\ref{item:mj3} we say that $\lozenge$ is their \textbf{canonical join}.
We have no reason to believe that all meets and joins in $\rsu$ and $\mdl$ are canonical, but we do not know of any that are not.

\begin{eg}\label{eg:prop-meet}
  If $P$ and $Q$ are two propositions, we claim that $\open{P\times Q}$ is the canonical meet of $\open P$ and $\open Q$.
  To prove this, note that $(P\times Q \to X) \eqvsym (P\to (Q\to X))$, and we have a commutative square
  \[
  \begin{tikzcd}
    X \ar[r] \ar[d] & P\to X \ar[d] \\ Q\to X \ar[r] & P\to (Q\to X).
  \end{tikzcd}
  \]
  If $X$ is $\open P$-modal, then the top function is an equivalence, and if $X$ is $\open Q$-modal, then the left-hand function is an equivalence, hence so is the right-hand one.
  Thus, in this case the diagonal is also an equivalence, so $X$ is $\open{P\times Q}$-modal.
  Conversely, since the unit $X\to (P\times Q \to X)$ factors through $P\to X$ and $Q\to X$, if it has a retraction then so do they; thus if $X$ is $\open{P\times Q}$-modal it is both $\open P$-modal and $\open Q$-modal.
  In other words, the operation $\open{} : \prop_\UU \to \lex$ preserves finite meets (it obviously preserves the top element).
\end{eg}

\begin{eg}\label{eg:prop-join}
  Suppose $P:A\to\prop_\UU$ is a family of propositions indexed by a type $A:\UU$, and let $Q \defeq \brck{\sm{a:A} P(a)}$.
  Then $Q$ is the join (i.e.\ disjunction) of all the $P(a)$'s in $\prop_\UU$.
  Now recall from \cref{eg:closed} that a type $X$ is $\closed Q$-modal if and only if $Q\to\iscontr(X)$, and note that
  \[ (Q \to \iscontr(X)) \eqvsym \prd{a:A} (P(a) \to \iscontr(X)). \]
  Thus, $X$ is $\closed Q$-modal if and only if it is $\closed{P(a)}$-modal for all $a:A$, and hence $\closed Q$ is the canonical meet of the $\closed{P(a)}$'s.

  We saw in \cref{eg:closed-connected} that the same condition $Q\to\iscontr(X)$ also characterizes the $\open Q$-connected types; thus $\open Q$ is the canonical join of the $\open{P(a)}$'s.
  In other words, the operation $\open{} : \prop_\UU \to \lex$ preserves joins (indexed by types in $\UU$).
\end{eg}

\begin{eg}
  The hypercompletion modality from \cref{thm:hypercompletion}, if it exists, is the canonical join $\bigvee_n \truncmod{n}$ of all the $n$-truncation modalities $\truncmod{n}$.
\end{eg}

We can construct meets in a fair amount of generality:

\begin{thm}\label{thm:meets}
  Any family ${(\modal_i)}_{i:I}$ of accessible reflective subuniverses (indexed by a type $I$ in $\UU$) has a canonical meet, which is again accessible, and is a modality or topological if each $\modal_i$ is.
\end{thm}
\begin{proof}
  By a ``family of accessible reflective subuniverses'' we mean that we have a family of generating families $F : \prd{i:I}\prd{a:A_i} B_i(a) \to C_i(a)$.
  Uncurrying $F$, we obtain a family $F : \prd{(i,a):\sm{i:I} A_i} B_i(a) \to C_i(a)$ indexed by $A \defeq \sm{i:I} A_i$, such that a type is $F$-local if and only if it is $F_i$-local for all $i$.
  Thus, $\localization{F}$ is the canonical meet.
  In the topological case we can take the $F_i$ to be topological generators with each $C_i(a)=\unit$ and each $B_i(a)$ a proposition, so that $F$ is also a topological generator.
\end{proof}

Thus, the posets $\tpl$, $\accmdl$, and $\accrsu$ have meets indexed by any type in $\UU$.
Using the result of~\cite{abfj:lexloc}, as in \cref{thm:nontop-lex}, we can show that $\acclex$ likewise has canonical meets.

However, these posets are not ``complete lattices'' as usually understood, since in general they are themselves large (i.e.\ not types in $\UU$), so we cannot use the usual argument to construct arbitrary joins from arbitrary meets.

There are also some cases in which we can identify the modal operator of a meet more explicitly:

\begin{thm}\label{thm:meets2}
  Let $\modal$ and $\lozenge$ be reflective subuniverses, and assume that $\modal$ preserves $\lozenge$-modal types.
  Then $\modal$ and $\lozenge$ have a canonical meet in $\rsu$, which is a modality, accessible, lex, or topological if $\modal$ and $\lozenge$ are.
\end{thm}
\begin{proof}
  If $Y$ is both $\modal$-modal and $\lozenge$-modal, for any $X$ we have
  \[ (X\to Y) \eqvsym (\lozenge X \to Y) \eqvsym (\modal \lozenge X \to Y) \]
  and $\modal \lozenge X$ is both $\modal$-modal and $\lozenge$-modal.
  Thus, the composite $\modal\circ\lozenge : \UU\to\UU$ is the modal operator for a canonical meet of $\modal$ and $\lozenge$.
  Preservation of modalities follows from \cref{thm:meet-join}, preservation of accessibility and topologicality follows from \cref{thm:meets} (using the different construction given there), while if $\modal$ and $\lozenge$ are both lex then so is their composite $\modal\circ \lozenge$.
\end{proof}

\begin{eg}
  By \cref{modaln-truncated}, if $\modal$ is lex then it preserves $n$-types.
  Thus the composite $\modal\circ\truncmod{n}$ is the meet $\modal\land\truncmod{n}$ of $\modal$ and the $n$-truncation modality $\truncmod{n}$.
\end{eg}

\begin{eg}\label{eg:strongly-disjoint}
  If every $\lozenge$-modal type is $\modal$-connected, then $\modal$ preserves $\lozenge$-modal types since it takes them all to $\unit$.
  Thus, the composite $\modal\lozenge$, which is the bottom element of $\rsu$, is also the meet $\modal\land\lozenge$.
  In this case we say that \textbf{$\modal$ is strongly disjoint from $\lozenge$} (note that this is an asymmetric relation).
  We will study this case further in \cref{sec:fracture}.
\end{eg}

\begin{eg}
  One special case in which \cref{thm:meets2} applies is if $\modal \lozenge \simeq \lozenge\modal$, since in that case for $\lozenge$-modal $X$ we have $\modal X \simeq \modal \lozenge X \simeq \lozenge \modal X$, so that $\modal X$ is also $\lozenge$-modal.
  For instance, \cref{eg:prop-meet} is an instance of this, since $(P\to (Q\to X)) \simeq (Q\to (P\to X))$.
  So is the binary case of \cref{eg:prop-join}, since join is associative and commutative: $P \ast (Q\ast X) \simeq (P \ast Q) \ast X \simeq (Q \ast P) \ast X \simeq Q \ast (P \ast X)$.
\end{eg}

\subsection{Lawvere-Tierney operators}%
\label{sec:ltop}

For any modality $\lozenge$, the slice poset $\rsu/\lozenge$ consists of the reflective subuniverses contained in $\UU_{\lozenge}$.
In other words, we have
\[ \rsu/\lozenge \eqvsym \rsu[\UU_{\lozenge}]. \]
Composing this with the universal property of meets, we obtain a partial adjunction
\[
\begin{tikzcd}
  \rsu \ar[r,phantom,"\scriptstyle\top"] \ar[r,dashed,bend left,"\blank\land\lozenge"] & \rsu/\lozenge \ar[l,bend left] \ar[r,equals] & \rsu[\UU_{\lozenge}]
\end{tikzcd}
\]
in which the right adjoint $\blank\land\lozenge$ is only known to be defined under the restrictions in \cref{thm:meets}.

One situation in which this is automatic is when $\lozenge$ is $\truncmod{-1}$, since every reflective subuniverse preserves mere propositions.
Thus we have a totally defined adjunction
\begin{equation}
\begin{tikzcd}
  \rsu \ar[r,phantom,"\scriptstyle\top"] \ar[r,bend left,"\blank\land\truncmod{-1}"] & \rsu/\truncmod{-1} \ar[l,bend left] \ar[r,equals] & \rsu[\prop].
\end{tikzcd}%
\label{eq:meet-brck}
\end{equation}
A reflective subuniverse of $\prop$, or more generally any universe $\Omega$ of mere propositions, is known as a \define{Lawvere-Tierney operator} or \define{local operator}.
It can equivalently be defined as a map $j:\Omega\to\Omega$ which is idempotent and preserves finite meets (including the top element):
\[ j(\top)=\top \qquad j(j(P)) = j(P) \qquad j(P\land Q) = j(P) \land j(Q). \]
This is equivalent to $j$ being order-preserving, inflationary, and idempotent:
\[ (P\to Q) \Rightarrow (j(P) \to j(Q)) \qquad P\to j(P) \qquad j(j(P)) = j(P) \]
and also to its being a monad on the poset $\Omega$.

In particular, such a monad automatically preserves meets, for the same reason that any modality preserves products; but since $\Omega$ is a poset, this makes it automatically left exact.
Moreover, we have:

\begin{lemma}\label{thm:rsu-prop}
  Every reflective subuniverse of a universe $\Omega$ of mere propositions is a lex modality.
\end{lemma}
\begin{proof}
  If $P=j(P)$ and $Q:P\to \Omega$ is such that $Q(x)=j(Q(x))$ for any $x:P$, then the projection $\proj1:(\sm{x:P} Q(x))\to P$ induces a map $j(\sm{x:P} Q(x))\to j(P) = P$.
  But as soon as we have $p:P$ then $(\sm{x:P} Q(x))\simeq Q(p)$ and so $j(\sm{x:P} Q(x)) \to j(Q(p)) = Q(p)$, hence $j(\sm{x:P} Q(x))\to (\sm{x:P} Q(x))$.
  Thus it is $\Sigma$-closed, hence a modality, and hence (as observed above) a lex modality.
\end{proof}

In other words, when $\Omega$ is a universe of mere propositions, we have
\[ \rsu[\Omega] = \mdl[\Omega] = \lex[\Omega]. \]

In general, the equivalence $\rsu/\lozenge \eqvsym \rsu[\UU_{\lozenge}]$ preserves $\Sigma$-closedness, since it preserves the modal types.
Thus the reflective subuniverse on $\UU$ corresponding to a Lawvere-Tierney operator $j$, which is defined by $A\mapsto j\brck{A}$, is always a modality.
However, it is not lex; in particular, $\truncmod{-1}$ itself is not lex.

A somewhat similar situation is when we have two universes $\UU:\UU'$.
Let $\rsu[\UU'/\UU]$ be the poset of pairs of reflective subuniverses $\modal'$ and $\modal$ on the universes $\UU'$ and $\UU$, respectively, such that a type in $\UU$ is $\modal$-modal if and only if it is $\modal'$-modal, and moreover for any $X:\UU$ the induced map $\modal' X \to \modal X$ is an equivalence.
There is an evident restriction functor $\rsu[\UU'/\UU] \to \rsu[\UU]$, and similarly for the other posets.

\begin{thm}\label{thm:acc-extend-adjt}
  The following functors have fully faithful right adjoints:
  \begin{mathpar}
    \accrsu[\UU'/\UU] \to \accrsu[\UU]\and
    \accmdl[\UU'/\UU] \to \accmdl[\UU]\and
    \tpl[\UU'/\UU] \to \tpl[\UU]
  \end{mathpar}
\end{thm}
\begin{proof}
  Given an accessible reflective subuniverse $\modal$ on $\UU$, we define $\modal'$ to be its canonical accessible extension to $\UU'$.
  As shown in \cref{thm:acc-extend}, this is a modality or topological if $\modal$ is, and it restricts to $\modal$ on $\UU$, so that $(\modal',\modal) : \accrsu[\UU'/\UU]$.

  We also need to show that this operation is functorial on $\rsu[\UU]$.
  If $\modal_1 \le \modal_2$, so that every $\modal_1$-modal type is $\modal_2$-modal, then the functor $\modal_1$ factors through the functor $\modal_2$, so that if $\modal_2 f$ is an equivalence then so is $\modal_1 f$.
  Therefore, by \cref{thm:acc-extend}\ref{item:ae3} every $\modal_1'$-modal type is $\modal_2'$-modal.

  The restriction of $(\modal',\modal)$ to $\UU$ is certainly $\modal$, so to have an adjunction it remains to show that for any $(\modal',\modal):\rsu[\UU'/\UU]$, the reflective subuniverse $\modal'$ is contained in the canonical accessible extension of $\modal$ to $\UU'$.
  But since $\modal'$ restricts to $\modal$ on $\UU$, it also inverts every map in $\UU$ inverted by $\modal$, so this follows from \cref{thm:acc-extend}\ref{item:ae3}.
\end{proof}

Using the result of~\cite{abfj:lexloc}, we can construct a similar adjoint to $\acclex[\UU'/\UU] \to \acclex[\UU]$ as in \cref{thm:nontop-lex}.

In general, we also do not know how to do without accessibility; the obvious thing to do is localize $\UU'$ at the class of \emph{all} maps in $\UU$ inverted by $\modal$, but as noted in \cref{rmk:extend-oops} there seems no reason why the resulting $\modal'$ would agree with $\modal$ on $\UU$.
However, there is one case in which this does work.

\begin{thm}\label{thm:tpl-extend}
  If propositional resizing holds for $\UU$, so that there is a universe $\Omega$ of mere propositions such that $\Omega:\UU$ and every mere proposition in $\UU$ is equivalent to one in $\Omega$, then the restriction functor
  \begin{equation}
    \rsu[\UU] \to \rsu[\Omega]\label{eq:restr}
  \end{equation}
  has a right adjoint $\shmod{}$, which lands inside $\tpl[\UU/\Omega]$ and induces an equivalence
  \[ \tpl[\UU] \eqvsym \rsu[\Omega]. \]
\end{thm}
\begin{proof}
  The restriction functor is defined on all of $\rsu[\UU]$ since any modal operator preserves mere propositions.
  Now given a reflective subuniverse of $\Omega$, i.e.\ a Lawvere-Tierney operator $j:\Omega\to\Omega$, we define $\shmod{j}$ to be the nullification of $\UU$ at all $j$-connected propositions (which are also called \textbf{$j$-dense}).
  Because any modality preserves mere propositions, if $P:\Omega$ then $\shmod{j}(P)$ is again a mere proposition, hence equivalent to some type in $\Omega$.
  Thus the universal properties of $j$ and $\shmod{j}$ do coincide for mapping into types in $\Omega$, so that $j(P) \eqvsym \shmod{j}(P)$.
  The rest of \cref{thm:acc-extend,thm:acc-extend-adjt} goes through without difficulty.

  Of course $\shmod{j}$ is topological by definition.
  Moreover, if $\modal$ is any topological modality on $\UU$, its generating family is equivalent to one lying in $\Omega$, hence contained in the family of all $j_\modal$-dense propositions (where $j_\modal$ is the restriction of $\modal$ to $\Omega$).
  Thus $\modal = \shmod{j_\modal}$, giving the stated equivalence.
\end{proof}

Note that the \emph{left} adjoint~\eqref{eq:restr} coincides with the \emph{right} adjoint in~\eqref{eq:meet-brck}.
That is, assuming propositional resizing, the forgetful operation $\rsu[\UU] \to \rsu[\Omega]$ has both adjoints: its left adjoint sends $j$ to $j\circ \truncf{-1}$, while its right adjoint is $\shmod{j}$.
The $\shmod j$-modal types are also called \textbf{$j$-sheaves}, with $\shmod j$ being \textbf{$j$-sheafification}.
(We remarked above that the $j$-connected propositions are called \textbf{$j$-dense}; the $j$-modal propositions are called \textbf{$j$-closed}.)

\begin{eg}
  For a proposition $P$, the \textbf{open Lawvere-Tierney operator} is defined by $o_P(Q) = P\Rightarrow Q$.
  This is the restriction to $\Omega$ of the open modality $\open P$, which is topological; hence $\shmod{o_P} = \open P$.
\end{eg}

\begin{eg}
  For a proposition $P$, the \textbf{closed Lawvere-Tierney operator} is defined by $c_P(Q) = P\lor Q$.
  Since $P \lor Q$ is equivalently the join $P \ast Q$ (see~\cite[Lemma 2.4]{joinconstruction}), this is the restriction to $\Omega$ of the closed modality $\closed P$, which is topological; hence $\shmod{c_P} = \closed P$.
\end{eg}

\begin{eg}\label{eg:dnsheaves}
  If $j = \neg\neg$ is the double negation operator, then by the usual arguments, the lattice of $\neg\neg$-closed elements of $\Omega$ is a Boolean algebra.
  Thus, the logic of the subtopos determined by $\shmod{\neg\neg}$ is Boolean.
  The $\shmod{\neg\neg}$-modal types are called \textbf{double-negation sheaves}.
\end{eg}

For a general reflective subuniverse $\modal$, the sheafification modality $\shmod{j_\modal}$ is far from equivalent to $\modal$.
We showed in \cref{thm:tpl-extend} that this is the case if $\modal$ is topological.
In classical 1-topos theory every lex modality is topological; in higher topos theory this is not the case, and $\modal$ can disagree with $\shmod{j_\modal}$ even when $\modal$ is lex, but at least we can say the following.

\begin{thm}\label{thm:lex-tpl}
  Assuming propositional resizing, the map $\shmod{j_\modal} A \to \modal A$ is an equivalence whenever $A$ is an $n$-type with $n<\infty$.
\end{thm}
\begin{proof}
  By \cref{thm:lex-ntypes-prop}, $A$ is $\modal$-modal if and only if it is $P$-null for any $\modal$-connected mere proposition $P$.
  But the latter condition exactly characterizes the $\shmod{j_\modal}$-modal types.
\end{proof}

At the other extreme, if $\modal$ is cotopological, then $\shmod{j_\modal}$ is the trivial modality.
For a general lex $\modal$, the restriction of $\modal$ to $\shmod{j_\modal}$ is cotopological, in the sense that any $\modal$-connected $\shmod{j_\modal}$-modal mere proposition is contractible.
That is, any lex modality ``decomposes'' into a topological part and a cotopological part, as in~\cite[Proposition 6.5.2.19]{lurie2009higher}.

\cref{thm:tpl-extend} also supplies additional structure on $\tpl$; the following proof is that of~\cite{wilson:frames}, as reproduced in~\cite[C1.1.15]{johnstone:elephant}.

\begin{corollary}
  Assuming propositional resizing, $\tpl$ is a coframe, i.e.\ a complete lattice in which finite joins distribute over arbitrary meets.
\end{corollary}
\begin{proof}
  Since $\tpl$ has canonical meets, the corresponding meets in $\rsu[\Omega]$ are also canonical, i.e.\ given by taking intersections of the sets of $j$-closed propositions.
  On the other hand, the ordering on modalities in $\Omega$ is the reverse of the pointwise ordering on Lawvere-Tierney operators $j:\Omega\to\Omega$, and any pointwise meet of Lawvere-Tierney operators is again a Lawvere-Tierney operator.

  Now suppose $j$ and ${(k_i)}_{i:I}$ are Lawvere-Tierney operators, and suppose $P$ is a $\bigwedge_i (j\lor k_i)$-closed proposition
  This means that $P$ is $(j\lor k_i)$-closed for each $i$, so that we have $P = j(P) \land k_i(P)$.
  Now
  \[ (j(P) \to P) =
  (j(P) \to j(P) \land k_i(P)) =
  (j(P) \to k_i(P)).
  \]
  Hence $j(P)\to P$ is $k_i$-closed for every $i$, so it is $\bigwedge_i k_i$-closed.
  Taking $Q \defeq j(P)$ and $R\defeq (j(P) \to P)$, and writing $k\defeq \bigwedge_i k_i$, we have
  \begin{multline*}
    (j\lor k)(Q\land R) = j(Q\land R) \land k(Q\land R) = j(Q)\land j(R) \land k(Q)\land k(R)\\
    = Q \land k(Q) \land R \land j(R) = Q\land R
  \end{multline*}
  so that $Q\land R$ is $(j\lor \bigwedge_i k_i)$-closed.
  But $Q\land R = (j(P) \land (j(P)\to P)) = P$.
\end{proof}

However, there seems no particular reason for the inclusions $\tpl \to \lex$ or $\tpl \to \mdl$ to preserve joins, and joins in $\lex$ and $\mdl$ in general seem difficult to construct.
In the next section we will consider one situation in which such joins can be constructed explicitly.

\subsection{A fracture and gluing theorem}%
\label{sec:fracture}

We end the paper by proving a general ``fracture and gluing'' theorem for a pair of modalities (\cref{thm:fracture-gluing}), which has as a special case the ``Artin gluing'' of a complementary closed and open subtopos.

\begin{defn}
  Let $\modal$ and $\lozenge$ be two modalities on a universe $\UU$.
  A $(\lozenge,\modal)$-\textbf{fracture square} consists of the following.
  \begin{itemize}
  \item An arbitrary type $A:\UU$.
  \item A $\modal$-modal type $B:\UU_\modal$.
  \item A $\lozenge$-modal type $C:\UU_\lozenge$.
  \item Functions $f:A\to B$ and $l:A\to C$ and $g:C\to \lozenge B$.
  \item A commutative square
    \[
    \begin{tikzcd}
      A \ar[r,"f"] \ar[d,"l"'] & B \ar[d,"{\modalunit^\lozenge_B}"] \\
      C \ar[r,"g"'] & \lozenge B.
    \end{tikzcd}
    \]
  \end{itemize}
  For any type $A$, the \textbf{canonical fracture square} associated to $A$ is the naturality square for $\modalunit^\lozenge$ at $\modalunit^\modal_A$:
  \begin{equation}
    \begin{tikzcd}
      A \ar[r,"\modalunit^\modal_A"] \ar[d,"\modalunit^\lozenge_A"'] & \modal A \ar[d,"{\modalunit^\lozenge_{\modal A}}"] \\
      \lozenge A \ar[r,"{\lozenge \modalunit^\modal_A}"'] & \lozenge \modal A.
    \end{tikzcd}%
    \label{eq:canonical-fracture}
  \end{equation}
  Given an arbitrary fracture square, we say it is \textbf{canonical} if it is equal to a canonical one in the type of fracture squares.
\end{defn}

\begin{lemma}\label{thm:canonical-fracture}
  A fracture square is canonical if and only if $f$ is $\modal$-connected and $l$ is $\lozenge$-connected.
\end{lemma}
\begin{proof}
  ``Only if'' is clear, so suppose $f$ is $\modal$-connected and $l$ is $\lozenge$-connected.
  Then by \cref{lem:reflective_uniqueness}, we have $(B,f) = (\modal A,\modalunit^\modal_A)$ and $(C,l) = (\lozenge A,\modalunit^\lozenge_A)$.
  And modulo these equivalences, $g$ and the commutative square are a factorization of $\modalunit^\lozenge_{\modal A} \circ \modalunit^\modal_A$ through $\modalunit^\lozenge_A$, hence inhabit a contractible type of which~\eqref{eq:canonical-fracture} is another element.
\end{proof}

\begin{thm}\label{thm:fracture}
  If $\lozenge$ is lex, then the canonical fracture square associated to $A$ is a pullback square if and only if $\modalunit^\modal_A$ is $\lozenge$-modal.
\end{thm}
\begin{proof}
  The maps $\modalunit^\lozenge_A$ and $\modalunit^\lozenge_{\modal A}$ are always $\lozenge$-connected, while $\lozenge \modalunit^\modal_A$ is a map between $\lozenge$-modal types and hence $\lozenge$-modal.
  Thus, if $\modalunit^\modal_A$ is $\lozenge$-modal then the square is a pullback by \cref{thm:lex-modalities}\ref{item:mu3d}.
  Conversely, if the square is a pullback then $\modalunit^\modal_A$ is a pullback of the $\lozenge$-modal map $\lozenge \modalunit^\modal_A$ and hence $\lozenge$-modal.
\end{proof}

\begin{corollary}
  If $\lozenge$ is lex and every $\modal$-connected type is $\lozenge$-modal, then every canonical fracture square is a pullback.
\end{corollary}
\begin{proof}
  The map $\modalunit^\modal_A$ is always $\modal$-connected, so the hypothesis ensures it is $\lozenge$-modal.
\end{proof}

Recall from \cref{eg:strongly-disjoint} that we say \textbf{$\modal$ is strongly disjoint from $\lozenge$} if every $\lozenge$-modal type is $\modal$-connected.

\begin{thm}\label{thm:cofracture}
  If $\modal$ is strongly disjoint from $\lozenge$, then every fracture square that is a pullback is canonical.
\end{thm}
\begin{proof}
  If a fracture square is a pullback, then $l$ must be $\lozenge$-connected since it is a pullback of $\modalunit^\lozenge_{\modal A}$, and similarly $f$ must be $\lozenge$-modal since it is a pullback of $g$.
  The assumption therefore ensures that $f$ is $\modal$-connected, so that \cref{thm:canonical-fracture} applies.
\end{proof}

Putting together \cref{thm:fracture,thm:cofracture} we can construct certain joins of modalities.

\begin{thm}\label{thm:join}
  If $\lozenge$ is a lex modality and $\modal$ is a modality is strongly disjoint from $\lozenge$, then the canonical join $\modal\lor\lozenge$ exists in $\rsu$.
  Moreover, the following are equivalent:
  \begin{enumerate}
  \item $A$ is $(\modal\lor\lozenge)$-modal.\label{item:j1}
  \item $\modalunit^\modal_A : A \to \modal A$ is $\lozenge$-modal.\label{item:j2}
  \item The canonical fracture square of $A$ is a pullback.\label{item:j3}
  \end{enumerate}
  And we have an equivalence of universes
  \begin{equation}
    \UU_{\modal\lor\lozenge} \simeq \sm{B:\UU_\modal}{C:\UU_\lozenge} (C \to \lozenge B).\label{eq:gluing0}
  \end{equation}
  Finally, if $\modal$ is also lex, then $\modal\lor\lozenge$ is a lex modality, and hence is the join in $\lex$.
\end{thm}
\begin{proof}
  The equivalence~\ref{item:j2}$\Leftrightarrow$\ref{item:j3} is by \cref{thm:fracture}, so we must show that such types form a reflective subuniverse.
  Given $A:\UU$, we define $(\modal\lor\lozenge)(A)$ to be the pullback of its canonical fracture square:
  \[
  \begin{tikzcd}
    A \ar[dr,"\modalunit^{\modal\lor\lozenge}_A" description] \ar[drr,bend left,"\modalunit^\modal_A"] \ar[ddr,bend right,"\modalunit^\lozenge_A"'] \\
    & (\modal\lor\lozenge)(A) \ar[r] \ar[d] \ar[dr,phantom,"\lrcorner" near start] & \modal A \ar[d,"\modalunit^\lozenge_{\modal A}"] \\
    & \lozenge A \ar[r,"\lozenge \modalunit^\modal_A"'] & \lozenge\modal A.
  \end{tikzcd}
  \]
  By \cref{thm:cofracture} this pullback square is a canonical fracture square, and thus $(\modal\lor\lozenge)(A)$ satisfies~\ref{item:j2} and~\ref{item:j3}.
  Now suppose we have some other $B$ satisfying~\ref{item:j2} and~\ref{item:j3}, hence a canonical fracture square that is a pullback:
  \[
  \begin{tikzcd}
    B \ar[r] \ar[d] \ar[dr,phantom,"\lrcorner" near start] & \modal B \ar[d,"\modalunit^\lozenge_{\modal B}"] \\
    \lozenge B \ar[r,"\lozenge \modalunit^\modal_B"'] & \lozenge\modal B.
  \end{tikzcd}
  \]
  Then we have equivalences
  \begin{align*}
    (A\to B)
    &\eqvsym (A\to \modal B) \times_{(A\to \lozenge \modal B)} (A\to\lozenge B)\\
    &\eqvsym (\modal A\to \modal B) \times_{(\lozenge A\to \lozenge \modal B)} (\lozenge A\to\lozenge B)
  \end{align*}
  in which the final pullback is of the two maps
  \begin{align*}
    (\lozenge \modalunit^\modal_B \circ \blank) &: (\lozenge A\to\lozenge B)\to (\lozenge A\to \lozenge \modal B)\\
    (\lam{h} \lozenge h \circ \lozenge \modalunit^\modal_A)&: (\modal A\to \modal B) \to (\lozenge A\to \lozenge \modal B).
  \end{align*}
  However, since the canonical fracture square of $A$ is also the canonical fracture square of $(\modal\lor\lozenge)(A)$, we also have
  \[ ((\modal\lor\lozenge)(A) \to B) \eqvsym (\modal A\to \modal B) \times_{(\lozenge A\to \lozenge \modal B)} (\lozenge A\to\lozenge B) \]
  and hence
  \[ (A\to B) \eqvsym ((\modal\lor\lozenge)(A) \to B)\]
  giving the desired universal property.

  To see that $\modal\lor\lozenge$ is the canonical join of $\modal$ and $\lozenge$, first note that if $A$ is $\modal$-modal, then $\modalunit^\modal_A$ and hence $\lozenge\modalunit^\modal_A$ are equivalences, so that its canonical fracture square is a pullback and so $A$ is $(\modal\lor\lozenge)$-modal.
  On the other hand, if $A$ is $\lozenge$-modal, then $\modalunit^\lozenge_A$ is an equivalence, while (since $\lozenge$ is strongly disjoint from $\modal$) $\modal A$ and hence $\lozenge \modal A$ are contractible; thus the canonical fracture square is again a pullback and so $A$ is $(\modal\lor\lozenge)$-modal.
  That is, any $\modal$-modal or $\lozenge$-modal type is $(\modal\lor\lozenge)$-modal, and hence any $(\modal\lor\lozenge)$-connected type is $\modal$-connected and $\lozenge$-connected.
  On the other hand, if $A$ is both $\modal$-connected and $\lozenge$-connected, then $(\modal\lor\lozenge)(A)$ is a pullback of a square of contractible types, hence contractible, so $A$ is also $(\modal\lor\lozenge)$-connected.

  As for~\eqref{eq:gluing0}, the left-to-right map sends $A$ to the bottom morphism in its canonical fracture square; while the right-to-left map sends $(B,C,g)$ to the pullback of $g$ and $\modalunit^\lozenge_B$, i.e.\ the vertex of the pullback fracture square with $g$ on the bottom.
  The two round-trip composites are the identity because a fracture square with $(\modal\lor\lozenge)$-modal vertex is a pullback if and only if it is canonical.

  Finally, suppose $\modal$ is also lex.
  To show that $\modal\lor\lozenge$ is a lex modality, by \cref{thm:rsu-lex} it suffices to show that $\modal\lor\lozenge$ preserves pullbacks.
  However, this follows from its construction as the pullback of the canonical fracture square, since $\lozenge$ and $\modal$ preserve pullbacks, and pullbacks commute with pullbacks.
  In somewhat more detail, given a cospan $B \to C \leftarrow D$, we have a $3\times 3$-diagram
  \[
  \begin{tikzcd}
    \modal B \ar[r] \ar[d] & \modal C \ar[d] & \modal D \ar[d] \ar[l] & \modal (B\times_C D) \ar[d] \\
    \lozenge \modal B \ar[r] & \lozenge \modal C & \lozenge \modal D \ar[l] & \lozenge\modal (B\times_C D) \\
    \lozenge B \ar[r] \ar[u] & \lozenge C \ar[u] & \lozenge D \ar[u] \ar[l] & \lozenge (B\times_C D) \ar[u]\\
    (\modal\lor\lozenge)(B) \ar[r] & (\modal\lor\lozenge)(C) & (\modal\lor\lozenge)(D) \ar[l] &
  \end{tikzcd}
  \]
  in which the limit of the rows gives the canonical fracture cospan for $B\times_C D$, whose pullback is $(\modal\lor\lozenge)(B\times_C D)$, whereas the limit of the columns gives $\modal\lor\lozenge$ of the given cospan.
  Thus, these two pullbacks agree, so $\modal\lor\lozenge$ preserves pullbacks, and hence is a lex modality.
\end{proof}

\begin{corollary}\label{thm:fracture-gluing}
  If $\lozenge$ is a lex modality, and $\modal$ a modality such that the $\lozenge$-modal types coincide with the $\modal$-connected types, then $\lozenge \lor \modal$ is the top element of $\lex$ (the trivial modality), and every canonical fracture square is a pullback.
  Moreover, we have an induced equivalence
  \begin{equation}
    \UU \simeq \sm{B:\UU_\modal}{C:\UU_\lozenge} (C \to \lozenge B).\label{eq:gluing}
  \end{equation}
\end{corollary}
\begin{proof}
  The additional assumption that $\modal$-connected types are $\lozenge$-modal means that a $(\modal\lor\lozenge)$-connected type must be both $\lozenge$-modal and $\lozenge$-connected, hence contractible.
  Thus, every type is $(\modal\lor\lozenge)$-modal, i.e.\ $(\modal\lor\lozenge)$ is the maximal modality.
  The equivalence~\eqref{eq:gluing} is just a specialization of~\eqref{eq:gluing0}.
\end{proof}

\begin{rmk}
  We call \cref{thm:fracture-gluing} a ``fracture theorem'' because it appears formally analogous to the fracture theorems for localization and completion at primes in classical homotopy theory~\cite{mp:more-concise}, or more generally for localization at complementary generalized homology theories~\cite{bauer:loc-hasse}.
  However, we do not know a precise relationship, because the classical fracture theorems either apply only to spectra (which do not form an $\infty$-topos) or to spaces with restrictions (such as nilpotence), and moreover the localizations appearing therein are not generally left exact (though they do often have some limit-preservation properties).
\end{rmk}

The equivalence~\eqref{eq:gluing} says informally that the universe of all types is equivalent to the ``comma category'' or ``gluing'' of the $\modal$-modal types with the $\lozenge$-modal types along the functor $\lozenge : \UU_\modal \to \UU_\lozenge$, as in the ``Artin gluing'' construction for toposes.
The paradigmatic example is the following.

\begin{eg}\label{eg:artin}
  Let $Q$ be a mere proposition.
  We have seen that both open and closed modalities $\open Q$ and $\closed Q$ are lex, and in \cref{eg:closed-connected} we noted that the $\open Q$-connected types coincide with the $\closed Q$-modal ones.
  Thus, these modalities satisfy the hypotheses of \cref{thm:fracture-gluing}.
  In particular, for any type $A$ we have a pullback square
  \begin{equation}
  \begin{tikzcd}
    A \ar[r] \ar[d] & Q\to A \ar[d] \\
    Q\ast A \ar[r] & Q\ast(Q\to A).
  \end{tikzcd}%
  \label{eq:propositional-fracture}
  \end{equation}
  To understand this better internally, suppose $Q$ is decidable, i.e.\ we have $Q+\neg Q$.
  Then we claim that $\eqv{Q\ast A}{\neg Q \to A}$.
  For if $Q$, then both are contractible, while if $\neg Q$, then both are equivalent to $A$.
  In particular, when $Q$ is decidable, $\eqv{(Q\ast (Q\to A))}{(\neg Q \land Q \to A)}$ and hence is contractible; so our above pullback square becomes
  \[
  \begin{tikzcd}
    A \ar[r] \ar[d] & Q\to A \ar[d] \\
    \neg Q\to A \ar[r] & \unit.
  \end{tikzcd}
  \]
  This is just the equivalence
  \[A \eqvsym ((Q+\neg Q) \to A) \eqvsym (Q\to A) \times (\neg Q \to A) \]
  that allows us to do case analysis on $Q$ to construct an element of any type $A$.

  Thus, the fracture square~\eqref{eq:propositional-fracture} can be viewed as a sort of ``constructive case analysis'': even if $Q$ is not decidable, we can construct an element of any type $A$ by constructing an element of $A$ assuming $Q$, then constructing an element of $Q\ast A$ (a sort of ``positive replacement'' for $\neg Q \to A$), then checking that they agree in $Q\ast (Q\to A)$.
  If $A$ is also a mere proposition, then $Q\ast A = Q\lor A$, so this reduces to the intuitionistic tautology
  \[ A \leftrightarrow (Q\lor A) \land (Q\to A). \]
  It is unclear to us whether the more general version has any applications.
\end{eg}

\section{Conclusion and outlook}%
\label{sec:conclusion}

The theory of lex and topological modalities can be viewed as a contribution to the program of giving an elementary (first order) definition of an $\infty$-topos as a purported model of homotopy type theory.
Specifically, lex and topological modalities are a higher-categorical analogue of the standard theory of Lawvere-Tierney operators in 1-topos theory, which are the usual way to internalize the notion of subtopos.
We thus expect that lex and topological modalities on universe objects will play a similar role in the theory of elementary $\infty$-toposes.

As mentioned in \cref{sec:fracture}, our fracture theorem can be viewed as an internal perspective on the gluing of higher toposes; an external perspective on gluing can be found in~\cite{shulman:invdia}.
We hope and expect that other topos-theoretic constructions, such as realizability, can also be extended to homotopy type theory.

The analogues of non-lex modalities and reflective subuniverses are not well-studied in 1-topos theory, perhaps because in the absence of a universe they cannot be internalized: as we have seen, any modality on a subobject classifier $\Omega$ is automatically lex.
However, the reflector into the quasitopos of \emph{separated} objects for a Lawvere-Tierney topology is an external analogue of a non-lex modality in our sense.
Notions of ``$\infty$-quasitopos'' relative to a factorization system are studied in~\cite{GepnerKock}; we expect there to be an internal analogue of this theory using modalities.

Localizations are, however, much better-studied in classical homotopy theory.
Modern calculational homotopy theory very often works in subuniverses that are localized at a prime number or a cohomology theory.
We therefore expect the theory of modalities and reflective subuniverses to be useful in extending such results to the synthetic setting of homotopy type theory, and thereby internalizing them in higher toposes.
Moreover, the homotopy-type-theoretic notion of modality has already proven fruitful in higher topos theory: in addition to the theory of $(\infty,1)$-quasitoposes in~\cite{GepnerKock}, a reworking of the synthetic Blakers--Massey theorem~\cite{ffll:blakers-massey} using general modalities in place of $n$-truncations has led to a new topos-theoretic generalization in~\cite{abfj:gen-blakers-massey}, with applications to Goodwillie calculus. 

\paragraph*{Acknowledgments}
This work was initiated when we were at the Institute for Advanced Study for the special year on Univalent Foundations.
We are thankful for the pleasant and active atmosphere during that year.


\appendix
\section{Semantics}%
\label{sec:semantics}

We now sketch something of how our syntactic description of modalities corresponds to semantic structures in higher category theory.
In the rest of the paper we used type universes, and some of our results require a universe; however strict universes are difficult to produce in categorical semantics, so for maximum generality here we consider modalities without universes.
Also, we will not concern ourselves with initiality theorems for syntax, instead working at the level of comprehension categories and their corresponding model categories.
Finally, in the interests of conciseness we will be sketchy about coherence theorems, although we expect that the methods of~\cite{lw:localuniv} will apply.

\subsection{Judgmental modalities}%
\label{sec:judgm-modal}

To avoid universes, in this appendix we will work with ``judgmentally specified'' modalities.
A judgmental modality acts on all types, not just those belonging to some universe, and makes sense even if there are no universes.
If our type theory does have universes, then to obtain a judgmental modality in this sense we need a consistent ``polymorphic'' family of modalities, one on each universe (or at least on all large enough universes).
But we have seen that practically any modality can be defined polymorphically, particularly those obtained by localization and nullification, so there is little loss of generality.

\cref{fig:jd-rsu} shows the judgmental rules for a reflective subuniverse, \cref{fig:jd-mod} augments it to a modality (a $\Sigma$-closed reflective subuniverse), and \cref{fig:jd-lex} to a lex modality (using \cref{thm:lex-modalities}\ref{item:mu1}), while \cref{fig:jd-gen} asserts that it is generated by a given family of maps.

Note that we have included an explicit rule that the predicate $\ismodal$ is invariant under equivalence; this is automatic if we have a univalent universe, but in general we should assert it explicitly.
Similar ``univalence-reductions'' must be made in various other places in the paper to work with judgmental modalities, manually replacing equalities between types by equivalences.
The only truly unavoidable use of univalence is in \cref{thm:acc-lex,thm:prop-loc-lex}.

\begin{figure}
  \centering
  \begin{mathpar}
    \inferrule{\Gamma\vdash A\;\mathsf{type}}{\Gamma\vdash \ismodal(A)\;\mathsf{type}}\and
    \inferrule{\Gamma\vdash A\;\mathsf{type}}{\Gamma\vdash \nameless:\mathsf{isprop}(\ismodal(A))}\and
    \inferrule{\Gamma\vdash A\;\mathsf{type}}{\Gamma\vdash \modal A\;\mathsf{type}}\and
    \inferrule{\Gamma\vdash A\;\mathsf{type}}{\Gamma\vdash \nameless:\ismodal(\modal A)}\and
    \inferrule{\Gamma\vdash A\;\mathsf{type} \\ \Gamma\vdash a:A}{\Gamma\vdash \modalunit[A](a):\modal A}\and
    \inferrule{\Gamma\vdash A\;\mathsf{type} \\ \Gamma\vdash B\;\mathsf{type} \\ \Gamma\vdash \nameless:\ismodal(B)}{\Gamma\vdash \nameless:\mathsf{isequiv}(\lam{f:\modal A \to B} f\circ \modalunit[A])}\and
    \inferrule{\Gamma\vdash A\;\mathsf{type} \\ \Gamma\vdash B\;\mathsf{type} \\ \Gamma\vdash f:\eqv A B \\ \Gamma\vdash \nameless:\ismodal(A)}{\Gamma\vdash \nameless:\ismodal(B)}\and
  \end{mathpar}
  \caption{A judgmental reflective subuniverse}%
  \label{fig:jd-rsu}
\end{figure}

\begin{figure}
  \centering
  \begin{mathpar}
    \inferrule{\Gamma\vdash A\;\mathsf{type} \\ \Gamma,a:A\vdash B\;\mathsf{type} \\\\ \Gamma\vdash \nameless:\ismodal(A) \\ \Gamma,a:A \vdash \nameless:\ismodal(B)}{\Gamma \vdash \nameless:\ismodal(\sm{a:A}B)}
  \end{mathpar}
  \caption{Judgmental $\Sigma$-closedness}%
  \label{fig:jd-mod}
\end{figure}

\begin{figure}[t]
  \centering
  \begin{mathpar}
    \inferrule{\Gamma\vdash A\;\mathsf{type} \\ \Gamma,x:A \vdash B\;\mathsf{type} \\
      \Gamma\vdash \nameless:\mathsf{iscontr}(\modal A) \\
      \Gamma\vdash \nameless:\mathsf{iscontr}(\modal (\sm{x:A} B)) \\
      \Gamma \vdash a:A
    }{\Gamma \vdash \nameless:\mathsf{iscontr}(\modal B)}
  \end{mathpar}
  \caption{Judgmental lexness}%
  \label{fig:jd-lex}
\end{figure}

\begin{figure}
  \centering
  \begin{mathpar}
    \text{(For some fixed $a:A\vdash F(a):B\to C$)}\\
    \inferrule{\Gamma\vdash X\;\mathsf{type}}{\Gamma\vdash \nameless: \ismodal(X) \simeq \prd{a:A}\mathsf{isequiv}(\lam{g:C(a)\to X} g\circ F(a))}
  \end{mathpar}
  \caption{Judgmental generation}%
  \label{fig:jd-gen}
\end{figure}

The corresponding definitions for comprehension categories are the following.
We state them in the ``weak stability'' style of~\cite{lw:localuniv} so that the ``local universes'' coherence method can be applied, although we will not write out the details of the coherence theorems.
We also simplify the definitions somewhat because we don't care about the identity of inhabitants of mere propositions, so we assume that any such inhabitant is ``good'', hence automatically weakly stable.

\begingroup
\def\C{\mathcal{C}}
\def\D{\mathcal{D}}
\def\T{\mathcal{T}}

\begin{defn}
  A comprehension category $(\C,\T)$ with weakly stable $\Pi$-types and identity types is equipped with a \textbf{weakly stable reflective subuniverse} if it has:
  \begin{itemize}
  \item For any $A\in \T(\Gamma)$ a family of ``good'' mere propositions $\ismodal(A)\in\T(\Gamma)$, weakly stable in that the reindexing of any such is another such.
  \item For any $A\in \T(\Gamma)$ a family of ``good'' reflections consisting of a type $\modal A \in \T(\Gamma)$ and a unit $\modalunit[A]:\Gamma.A\to \Gamma.\modal A$ over $\Gamma$, weakly stable under reindexing, such that every good $\ismodal(\modal A)$ over $\Gamma$ has a section.
  \item For any $A,B\in \T(\Gamma)$, any section of a good $\ismodal(\modal B)$ over $\Gamma$, any good reflection $(\modal A,\modalunit[A],\nameless)$ of $A$, and any good (non-dependent) $\Pi$-types $\prd{A}B$ and $\prd{\modal A} B$, the map $\Gamma.(\prd{\modal A}B) \to \Gamma.(\prd{A}B)$ induced by $\modalunit[A]$ is an equivalence over $\Gamma$.
  \end{itemize}
  This is a \textbf{weakly stable modality} if additionally:
  \begin{itemize}
  \item For any $A\in\T(\Gamma)$ and $B\in \T(\Gamma.A)$, every good $\ismodal(A)$ with a section over $\Gamma$, and every good $\ismodal(B)$ with a section over $\Gamma.A$, every good $\ismodal(\sum{A}B)$ has a section over $\Gamma$.
  \end{itemize}
  It is a \textbf{weakly stable lex modality} if furthermore:
  \begin{itemize}
  \item For any $A\in\T(\Gamma)$ and $B\in \T(\Gamma.A)$, and good reflections $\modal A$, $\modal B$, and $\modal(\sm{A}B)$ such that $\modal A$ and $\modal(\sm{A} B)$ are contractible over $\Gamma$, we have that $\modal B$ is contractible over $\Gamma.A$.
  \end{itemize}
  Finally, if we have $A\in\T(\diamond)$ and $B,C\in \T(\diamond.A)$, where $\diamond$ is the terminal object of $\C$, and $F:\diamond.A.B \to \diamond.A.C$ over $\diamond.A$, we say that a weakly stable reflective subuniverse is \textbf{generated by $F$} if
  \begin{itemize}
  \item for any $X\in\T(\Gamma)$, every good $\ismodal(X)\in\T(\Gamma)$ is equivalent over $\Gamma$ to
    \[\prd{a:A}\mathsf{isequiv}(\lam{g:C(a)\to X} g\circ F(a)).\]
\end{itemize}
\end{defn}

\noindent
It is straightforward to see that the syntactic category of a type theory satisfying some or all of the rules in \cref{fig:jd-lex,fig:jd-gen,fig:jd-mod,fig:jd-lex} has the corresponding weakly stable (indeed, strictly stable) structure.
A suitable initiality theorem would imply (together with the local universes coherence theorem to strictify weakly stable structure) that such a type theory can then be interpreted into any model with weakly stable structure.

In the remaining sections of this appendix we will show that the above weakly stable structures correspond to well-known $\infty$-category-theoretic notions.
However, it is worth noting first of all one important example that doesn't require any $\infty$-categorical machinery.

\begin{eg}\label{thm:ttfc-gluing}
  A ``type-theoretic fibration category'' in the sense of~\cite{shulman:invdia} is a particular sort of presentation of a comprehension category.
  It is shown in~\cite{shulman:invdia} (see also~\cite{shulman:eiuniv}) that if $F:\C\to\D$ is a ``strong fibration functor'' between type-theoretic fibration categories, then the ``fibrant gluing category'' ${(\D\downarrow F)}_{\mathbf{f}}$ is again a type-theoretic fibration category.
  The objects of ${(\D\downarrow F)}_{\mathbf{f}}$ consist of an object $\Gamma_0\in\C$ and a display map $F(\Gamma_0).\Gamma_1 \to F(\Gamma_0)$ in $\D$, while its types in such a context are pairs of a type $A_0\in \T_\C(\Gamma_0)$ and a type $A_1 \in\T_\D(F(\Gamma_0).\Gamma_1.F(A_0))$.

  We claim there are two canonical lex modalities on ${(\D\downarrow F)}_{\mathbf{f}}$, representing the canonical open and closed subtoposes of a glued topos.
  The ``gluing proposition'' $Q$ in the empty context $(\diamond,\diamond)$ has $Q_0 = \unit$ and $Q_1 = \emptyt$.
  Since ${(\D\downarrow F)}_{\mathbf{f}}$ inherits type-theoretic strurcture, including $\Pi$-types, we can use the internal construction of the open modality $\open Q$ as $\open Q(A) = (Q\to A)$.
  Note that this automatically restricts to any universe.

  The second lex modality is supposed to be the closed one $\closed Q$.
  If we knew that ${(\D\downarrow F)}_{\mathbf{f}}$ had higher inductive pushouts, then we could use the internal definition of $\closed Q$ as $\closed Q(A) = Q\ast A$.
  However, with the concrete construction of ${(\D\downarrow F)}_{\mathbf{f}}$ we also have a concrete construction of $\closed Q$, which takes a type $(A_0,A_1)$ to $(\unit,\sm{F A_0} A_1)$.
  This has the advantage that it manifestly restricts to any universe in ${(\D\downarrow F)}_{\mathbf{f}}$ determined by universes $\UU_0$ in $\C$ and $\UU_1$ in $\D$ such that $F$ takes $\UU_0$-small types to $\UU_1$-small ones.
\end{eg}
\endgroup

\subsection{Modalities in model categories}%
\label{sec:model-cats}

\begingroup
\def\M{\mathcal{M}}
\def\fibmf{(\mathcal{M}_{\mathbf{f}},\mathcal{F}_{\mathbf{f}})}
\def\F{\mathcal{F}}
\def\C{\mathcal{C}}
\def\Cf/#1{\mathcal{C}_{\sslash#1}}
\def\D{\mathcal{D}}
Let $\M$ be a type-theoretic model category in the sense of~\cite{shulman:invdia}: right proper, with cofibrations closed under limits, and with right adjoints $f_*$ to pullback along fibrations.
Then there is a comprehension category $\fibmf$, in which $\M_{\mathbf{f}}$ is the category of fibrant objects in $\M$ and  $\F_{\mathbf{f}}$ is the category of fibrations with fibrant codomain, and which has sufficient (weakly stable) structure to model type theory with $\Sigma$, $\Pi$, and identity types.
(If $\M$ is sufficiently nice, then it also models various higher inductive types, including localizations; see~\cite{ls:hits} and \cref{sec:localiz-model}.)

Let $h\M$ be the homotopy $(\infty,1)$-category of $\M$.
Then the slice model categories $\M_{/\Gamma}$ present the slice $(\infty,1)$-categories $h\M_{/\Gamma}$.
Moreover, for any fibration $f$, the adjunction $f^* \dashv f_*$ is a Quillen adjunction, hence descends to homotopy categories; this shows that $h\M$ is locally cartesian closed.

\begin{defn}\label{defn:reflective-subfibration}
  A \textbf{reflective subfibration} of an $(\infty,1)$-category $\C$ with finite limits consists of:
  \begin{enumerate}
  \item For every object $x\in\C$, a reflective full sub-$(\infty,1)$-category $\D_x$ of the slice $\C_{/x}$, with reflector $\modal_x$.\label{item:rsf1}
  \item Each pullback functor $f^* :\C_{/y}\to \C_{/x}$ restricts to a functor $\D_y\to\D_x$.\label{item:rsf2}
  \item For any $z\in \C_{/y}$, the induced map $\modal_x(f^*z) \to f^*(\modal_y z)$ is an equivalence.\label{item:rsf3}
  \end{enumerate}
\end{defn}

\begin{thm}\label{thm:ttmc-rsu}
  If $\M$ is a type-theoretic model category equipped with a reflective subfibration of $h\M$, then $\fibmf$ has a weakly stable reflective subuniverse.
\end{thm}
\begin{proof}
  Given a fibration $\Gamma.A \twoheadrightarrow \Gamma$, we define a map $\modalunit[A]:\Gamma.A \to \Gamma.\modal A$ of fibrations over $\Gamma$ to be a good reflection if the map it presents in $h\M_{/\Gamma}$ has the universal property of a reflection into $\D_\Gamma$.
  That is, if the fibration $\Gamma.\modal A \twoheadrightarrow \Gamma$ presents an object of $\D_\Gamma$, and precomposition with $\modalunit[A]$ induces an equivalence of mapping spaces into any object of $\D_\Gamma$.

  We define the good fibrations $\Gamma.\ismodal(A) \twoheadrightarrow \Gamma$ to be those that are equivalent to $\isequiv(\modalunit[A])$.
  This implies that $\ismodal(A)$ is a mere proposition over $\Gamma$, since $\isequiv(\modalunit[A])$ is, and is independent of the chosen $\modalunit[A]$ since they are all equivalent.
  Since $\isequiv(\modalunit[A])$ has a section if and only if $\modalunit[A]$ is actually an equivalence over $\Gamma$, it follows that any $\ismodal(A)$ has a section if and only if $\Gamma.A$ actually lies in $\D_\Gamma$.
  In particular, $\ismodal(\modal A)$ always has a section.

  These definitions are weakly stable because the subcategories $\D_x$ and reflectors $\modal_x$ are assumed to be stable under $(\infty,1)$-categorical pullback in $h\M$, and pullback of fibrations in $\M$ presents the latter.
  It remains to show that for any fibrations $\Gamma.A \twoheadrightarrow\Gamma$ and $\Gamma.B \twoheadrightarrow\Gamma$, with the latter in $\D_\Gamma$, precomposition with any $\modalunit[A]$ induces an equivalence $B^{\modal A} \to B^A$ of local exponentials in $\M_{/\Gamma}$ (which represent the non-dependent $\Pi$-types in $\fibmf$).
  Now these local exponentials do present the $(\infty,1)$-categorical local exponentials in $h\M$, and by assumption we have an induced equivalence of hom-\emph{spaces} (i.e.\ hom-$\infty$-groupoids)
  \[ h\M_{/\Gamma}(\modal A,B) \xrightarrow{\sim} h\M_{/\Gamma}(A,B). \]
  By assumed pullback-stability, for any fibration $p:\Gamma.C\twoheadrightarrow\Gamma$ we also have
  \[ h\M_{/\Gamma.C}(p^*(\modal A),p^*B) \xrightarrow{\sim} h\M_{/\Gamma.C}(p^*A,p^*B) \]
  or equivalently
  \[ h\M_{/\Gamma}(C\times_\Gamma \modal A,B) \xrightarrow{\sim} h\M_{/\Gamma}(C\times_\Gamma  A,B). \]
  Now the local exponential adjunction gives
  \[ h\M_{/\Gamma}(C,B^{\modal A}) \xrightarrow{\sim} h\M_{/\Gamma}(C,B^A) \]
  and hence $B^{\modal A} \xrightarrow{\sim} B^A$ by the $(\infty,1)$-categorical Yoneda lemma.
\end{proof}

Admittedly, reflective subfibrations are not an especially familiar object in category theory.
However, in \cref{sec:localiz-model} we will see that often they can be constructed from ordinary reflective subcategories.
For now, we move on to consider the analogous structure for modalities.

\begin{defn}
  A reflective subfibration $\D$ of an $(\infty,1)$-category $\C$ is \textbf{composing} if whenever $f:y\to x$ lies in $\D_x$ and $g:z\to y$ lies in $\D_y$, the composite $fg:z\to x$ lies in $\D_x$.
\end{defn}

\begin{thm}\label{thm:ttmc-mod}
  If $\M$ is a type-theoretic model category equipped with a composing reflective subfibration of $h\M$, then $\fibmf$ has a weakly stable modality.
\end{thm}
\begin{proof}
  Type-theoretic $\Sigma$-types are presented by composites of fibrations.
\end{proof}

Inside type theory, we proved in \cref{sec:modal-refl-subun} that modalities, i.e.\ $\Sigma$-closed reflective subuniverses, are equivalent to stable orthogonal factorization systems.
The analogous ``external'' categorical fact is:

\begin{thm}
  A composing reflective subfibration of an $(\infty,1)$-category $\C$ is the same as a stable orthogonal factorization system $(\cL,\cR)$ on $\C$, where $\cR$ is the class of maps $f:y\to x$ that lie in $\D_x$.
\end{thm}
\begin{proof}
  A proof for 1-categories can be found in~\cite[2.12]{cjkp:locstab-fact}, and essentially the same proof works as well for $(\infty,1)$-categories.
  More generally, it is shown there that a composing fiberwise-reflective subfibration (satisfying \cref{defn:reflective-subfibration}~\ref{item:rsf1} and~\ref{item:rsf2} but not necessarily~\ref{item:rsf3}) is the same as a not-necessarily-stable orthogonal factorization system.
\end{proof}

Stable orthogonal factorization systems on $(\infty,1)$-categories are also studied in~\cite{GepnerKock}.
In particular, their Theorem 4.10 that the right class of such a factorization system is ``local'' is an $(\infty,1)$-categorical analogue of our \cref{thm:subuniv-sofs} that a stable factorization system is determined by its modal types.

Factorization systems are more familiar categorically than reflective subfibrations.
But in \cref{sec:localiz-model} we will see that they can also be often constructed from reflective subcategories, which are even more familiar.

However, in the lex case, we can go right to the reflective subcategories:

\begin{defn}
  A \textbf{lex reflective subcategory} of an $(\infty,1)$-category $\C$ with finite limits is a full reflective subcategory $\D$ whose reflector $\modal$ preserves finite limits (equivalently, pullbacks, since any reflector preserves the terminal object).
\end{defn}

\begin{thm}
  A lex reflective subcategory of $\C$ induces a stable factorization system $(\cL,\cR)$ of $\C$, where $\cR$ consists of the morphisms $f:y\to x$ such that the naturality square of the reflector is a pullback:
  \[
  \begin{tikzcd}
    y \ar[r,"{\modalunit[y]}"] \ar[d,"f"'] & \modal y \ar[d,"\modal f"]\\
    x \ar[r,"{\modalunit[x]}"'] & \modal x
  \end{tikzcd}
  \]
  and $\cL$ consists of the morphisms $f$ such that $\modal f$ is an equivalence.
  Moreover, this sets up a bijection between lex reflective subcategories of $\C$ and stable factorization systems $(\cL,\cR)$ such that $\cL$ satisfies the 2-out-of-3 property (the nontrivial part of this being left cancellation: if $gf\in\cL$ and $g\in\cL$ then $f\in\cL$).
\end{thm}
\begin{proof}
  Proofs for 1-categories can be found in~\cite[Theorems 4.7 and 2.3]{chk:reflocfact}, and essentially the same proofs work as well for $(\infty,1)$-categories.
\end{proof}

\begin{thm}\label{thm:ttmc-lex}
  If $\M$ is a type-theoretic model category equipped with a lex reflective subcategory of $h\M$, then $\fibmf$ has a weakly stable lex modality.
\end{thm}
\begin{proof}
  It suffices to show that if a stable factorization system satisfies 2-out-of-3, then the weakly stable modality constructed in \cref{thm:ttmc-mod} is lex.
  Now a fibration $\Gamma.A\twoheadrightarrow\Gamma$ is contractible over $\Gamma$ just when it is an acyclic fibration, and $\Sigma$-types are given by composition of fibrations.
  Now, given fibrations $f:\Gamma.A.B\twoheadrightarrow\Gamma.A$ and $g:\Gamma.A\twoheadrightarrow\Gamma$, by construction of the factorization system we have the following homotopy pullback squares, where for clarity we write $r$ for the $(\infty,1)$-categorical reflector:
  \[
  \begin{tikzcd}
    \Gamma.A.\modal B \ar[r] \ar[d] & \Gamma.\modal(\sm{A}B) \ar[r] \ar[d] & r(\Gamma.A.B)\ar[d]\\
    \Gamma.A \ar[r] & \Gamma.\modal A \ar[r] \ar[d] & r(\Gamma.A) \ar[d]\\
    & \Gamma \ar[r] & r\Gamma.
  \end{tikzcd}
  \]
  In particular, therefore, if $\Gamma.\modal A \to \Gamma$ and $\Gamma.\modal(\sm{A}B) \to \Gamma$ are acyclic, then by 2-out-of-3 and pullback so is $\Gamma.A.\modal B \to \Gamma.A$, as desired.
\end{proof}

\begin{eg}
  In \cref{thm:ttfc-gluing} we saw that when a gluing construction is performed at the fibration-category level, the corresponding open and closed modalities  can be constructed at the same level.
  However, these modalities are still present even if the gluing happens only at the $\infty$-categorical level.
  Suppose $F:\C\to\D$ is a finite-limit-preserving functor of $(\infty,1)$-categories with finite limits and colimits.
  Then the comma $(\infty,1)$-category $(\D\downarrow F)$, whose objects are triples $(x\in \C, y\in \D, f:y\to Fx)$, includes $\C$ and $\D$ as full subcategories: $x\in\C$ is identified with $(x,Fx,1_{Fx})$ while $y\in\D$ is identified with $(1,y,!)$.
  The forgetful functors $(\D\downarrow F) \to \C$ and $(\D\downarrow F) \to \D$ are left exact left adjoints to these inclusions, so we have two lex reflective subcategories.
  Hence, if $(\D\downarrow F) = h\M$, we have two induced lex modalities on $\fibmf$, which we can identify internally with the open and closed modalities associated to $Q = (1,0,!)$.
\end{eg}

\subsection{Modalities in syntactic categories}%
\label{sec:syntactic}

We now ask the opposite question: does our syntactic kind of modality induce a higher-categorical structure?
For this the basic tool is~\cite{kapulkin:lccqcat-tt}, which shows that given any contextual category $\C$ with $\Sigma$-types, $\Pi$-types, and identity types, the $(\infty,1)$-category $h\C$ obtained by localizing $\C$ at its type-theoretically defined equivalences is locally cartesian closed.
Moreover, the fibrant slice $\Cf/X$ (consisting of the fibrations, i.e.\ composites of dependent projections, with codomain $X$) similarly localizes to the slice $(\infty,1)$-category $h\C_{/X}$, and the $\Pi$-types in $\C$ present the dependent exponentials in $h\C$.

A contextual category is not quite the same as a comprehension category.
However, every contextual category has a canonical split comprehension category structure, while every split comprehension category has a ``contextual core'' obtained by repeatedly extending the terminal object by type comprehensions.
Moreover, in the case of the syntactic category, and also in the case of $\fibmf$, these two operations preserve the base category up to equivalence (though not isomorphism) of 1-categories.

In particular, once we apply the coherence theorem of~\cite{lw:localuniv} to $\fibmf$ to get a split comprehension category and hence a contextual category, the equivalences therein coincide with the model-categorical right homotopy equivalences between fibrant objects in $\M$.
If all objects of $\M$ are cofibrant, as is usually the case in type-theoretic model categories, then these are also the model-categorical weak equivalences, and so the contextual-category localization of~\cite{kapulkin:lccqcat-tt} coincides with the model-categorical localization of $\M$.

We now proceed to consider modalities and their kin.

\begin{thm}\label{thm:cc-rsf}
  If $\C$ is a contextual category with a (strictly stable) reflective subuniverse, then $h\C$ has a reflective subfibration.
\end{thm}
\begin{proof}
  For $\Gamma\in\C$, define $\D_\Gamma$ to be the full subcategory of $h\C_{/\Gamma}$ on the morphisms presented (up to equivalence) by the comprehensions $\Gamma.A\to \Gamma$ such that $\ismodal(A)$ has a section over $\Gamma$.
  The reflectors are given by the type operations $\modal$ and $\modalunit$, which are pullback-stable in $h\C$ since the type operations on $\C$ are reindexing-stable.
  For the universal property, if $A$ is arbitrary and $B$ is modal over $\Gamma$, then we have
  \[ h\C_{/\Gamma}(A,B) \simeq h\C_{/\Gamma}(1,\prd{A}B) \simeq h\C_{/\Gamma}(1,\prd{\modal A}B) \simeq h\C_{/\Gamma}(\modal A,B)
  \]
  since by assumption, precomposition with $\modalunit$ induces an internal equivalence $\prd{\modal A} B \simeq \prd{A} B$, which is therefore also an external equivalence of local exponentials in $h\C$.
\end{proof}

\begin{thm}\label{thm:cc-ofs}
  If $\C$ is a contextual category with a (strictly stable) modality, then $h\C$ has a stable orthogonal factorization system.
\end{thm}
\begin{proof}
  It suffices to show that the reflective subfibration of \cref{thm:cc-rsf} is composing.
  Suppose $Y\to \Gamma$ and $Z\to Y$ are in $\D_\Gamma$ and $\D_Y$ respectively.
  As a special case of~\cite[Theorem 3.1]{kapulkin:lccqcat-tt} (along with the existence of $\Sigma$-types), every object of $h\C_{/\Gamma}$ is equivalent to one of the form $\Gamma.A \to \Gamma$ for a type $A$ in context $\Gamma$.
  Thus we have $Y\simeq \Gamma.A$ over $\Gamma$, and $A$ is modal.
  Similarly, pulling back $Z\to Y$ to an object of $h\C_{/\Gamma.A}$ and applying the same fact, we get that it is equivalent to a modal type $B$ in context $\Gamma.A$.
  Now the composite $\Gamma.A.B\to \Gamma.A\to \Gamma$ is isomorphic to $\Gamma.\sm{A}B \to \Gamma$, which is modal by assumption.
  Thus, the equivalent composite $Z\to Y\to \Gamma$ is also modal.
\end{proof}

\begin{thm}\label{thm:cc-lex}
  If $\C$ is a contextual category with a (strictly stable) lex modality, then $h\C$ has a lex reflective subcategory.
\end{thm}
\begin{proof}
  It suffices to show that the $\cL$-maps in $h\C$ are left cancellable.
  As in \cref{thm:cc-ofs}, we may assume our maps are of the form $\Gamma.A\to \Gamma$ and $\Gamma.A.B\to \Gamma.A$, with $\Gamma.A\to \Gamma$ and the composite $\Gamma.A.B \to \Gamma.A \to \Gamma$ being $\cL$-maps.
  Now the $\cL$-maps are those whose fiberwise reflection is terminal, which is to say that $\modal A$ and $\modal (\sm{A} B)$ are contractible in context $\Gamma$.
  By assumption, therefore $\modal(B)$ is contractible in context $\Gamma.A$, which is to say that $\Gamma.A.B\to\Gamma.A$ is an $\cL$-map.
\end{proof}

It is straightforward to check that if we start from a reflective subfibration of a type-theoretic model category $\M$, construct a reflective subuniverse on $\fibmf$, and then pass back to $h\fibmf$, which is equivalent to $\M$, we get the same subfibration we started with.

\subsection{Localizations in model categories}%
\label{sec:localiz-model}

Now suppose $\M$ is an \emph{excellent model category} in the sense of~\cite{ls:hits}, i.e.\ in addition to being type-theoretic, it is locally presentable,  cofibrantly generated, simplicial, simplicially locally cartesian closed, and every monomorphism is a cofibration (and in particular, every object is cofibrant).
This implies that $h\M$ is a locally presentable locally cartesian closed $(\infty,1)$-category; and conversely every locally presentable locally cartesian closed $(\infty,1)$-category arises as $h\M$ for some excellent model category $\M$.
Moreover, by the methods of~\cite{ls:hits}, $\fibmf$ also models a localization higher inductive type as in \cref{sec:localization}.\footnote{Specifically~\cite[Example 12.21]{ls:hits} discusses localization using bi-invertibility instead of path-splitness; our $\mathcal{J}_F X$ can be obtained by simply dropping the extra two constructors.
  Pushouts are constructed in~\cite[\S4--6]{ls:hits}, which we can then combine with $\mathcal{J}_F X$ to give our $\mathcal{L}_F X$.}
Thus we have:

\begin{thm}
  For any excellent model category $\M$, and any map $F$ between fibrations over a fibrant object, the comprehension category $\fibmf$ has a weakly stable reflective subuniverse that is generated by $F$.
\end{thm}
\begin{proof}
  We define $\ismodal(X)$ to make \cref{fig:jd-gen} true, and construct the rest of the structure from the higher inductive localization as in \cref{sec:localizing}.
\end{proof}

On the other hand, a full subcategory $\D$ of a (usually locally presentable) $(\infty,1)$-category $\C$ is said to be \textbf{accessible} if there is a family  $F$ of maps in $\C$ such that $X\in \D$ if and only if the map on hom-spaces $\C(C,X) \xrightarrow{\C(f,X)} \C(B,X)$ is an equivalence for each $(f:B\to C)\in F$.
Local presentability of $\C$ implies that any accessible subcategory is reflective.

Now, if we have an accessible subcategory $\D$ of $h\M$, generated by a family of maps ${\{ f_i : B_i \to C_i \}}_{\in I}$, we can form their coproduct in $h\M$:
\[
\begin{tikzcd}
  \coprod_i B_i \ar[rr,"\coprod_i f_i"] \ar[dr] && \coprod_i C_i \ar[dl] \\ & \coprod_i 1
\end{tikzcd}
\]
and represent this by a map between fibrations over fibrant objects in $\M$.
Type-theoretically this represents a family of maps $a:A \vdash F_a : B(a) \to C(a)$, so we can localize at it to obtain a reflective subuniverse, corresponding to a reflective subfibration of $h\M$.

Of course, the fiber over $1$ of a reflective subfibration is a reflective subcategory; but in general this fiber will \emph{not} be the same as the accessible subcategory $\D$ generated by localizing at $F$ in the usual $(\infty,1)$-categorical sense.
The issue is that our higher inductive localization is an \emph{internal} localization: the local types are those for which precomposition induces equivalences on internal exponential objects, not just external hom-spaces.
To see that this makes a difference, note that we showed in \cref{sec:ssrs} that the reflector of any reflective subuniverse preserves products; but there are certainly accessible reflective subcategories whose reflectors do not preserve products.

However, if our original accessible subcategory \emph{does} have this property, and moreover $h\M$ is extensive (so that maps over coproducts faithfully represent external families of maps), then we can show that the two subcategories do coincide.
First we need some lemmas.

\begin{lemma}\label{thm:exp-ideal}
  For a reflective subcategory $\D$ of a cartesian closed $(\infty,1)$-category $\C$, the following are equivalent:
  \begin{enumerate}
  \item The reflector $\C\to\D$ preserves finite (equivalently, binary) products.
  \item $\D$ is an \emph{exponential ideal}, i.e.\ if $X\in\D$ and $Y\in\C$ then $X^Y\in\D$.
  \end{enumerate}
\end{lemma}
\begin{proof}
  This is a standard result for 1-categories (see~\cite[A4.3.1]{johnstone:elephant}), and the same proof applies to $(\infty,1)$-categories.
\end{proof}

We say that an $(\infty,1)$-category is \textbf{(infinitary) extensive}~\cite{clw:ext-dist} if it has small coproducts that are disjoint and pullback-stable.
This is equivalent to saying that the functor
\[ \C_{/\coprod_i A_i} \to \prod_i \C_{/A_i}, \]
induced by pullback along the coproduct injections, is an equivalence of $(\infty,1)$-categories.
In particular, for any family of maps $f_i : A_i \to B_i$ we have pullback squares
\[
\begin{tikzcd}
  A_i \ar[r] \ar[d,"f_i"'] & \coprod_i A_i \ar[d,"\coprod_i f_i"]\\
  B_i \ar[r] & \coprod_i B_i.
\end{tikzcd}
\]
Any $\infty$-topos in the sense of~\cite{lurie2009higher} is extensive (this is a special case of descent for colimits).

\begin{lemma}\label{thm:ext-exp}
  Let $X$ be an object and $F = {\{ f_i : B_i \to C_i \}}_{i\in I}$ a family of maps in an extensive and locally cartesian closed $(\infty,1)$-category $\C$.
  Then the following are equivalent:
  \begin{enumerate}
  \item For any $i\in I$, the map $X^{C_i} \to X^{B_i}$ of exponential objects is an equivalence in $\C$.
  \item If we let $A = \coprod_{i} 1$ and $B = \coprod_i B_i$ and $C = \coprod_i C_i$, with an induced map $F : B\to C$ over $A$, then the induced map ${(A^*X)}^C \to {(A^*X)}^B$ of local exponential objects in $\C_{/A}$ is an equivalence.
  \end{enumerate}
\end{lemma}
\begin{proof}
  By extensivity, ${(A^*X)}^C \to {(A^*X)}^B$ is an equivalence if and only if its pullback along each injection $i:1\to \coprod_i 1 = A$ is an equivalence.
  Since local exponentials are preserved by pullback, these pullbacks are the exponentials ${(i^*A^*X)}^{i^*C} \to {(i^*A^*X)}^{i^*B}$.
  But $i^*A^*X=X$, while by extensivity again we have $i^*C = C_i$ and $i^*B = B_i$.
\end{proof}

If $\C = h\M$, then we can represent the map $F:B\to C$ in $\C_{/A}$ above by a map between fibrations over a fibrant object $A$:
\[
\begin{tikzcd}
  B \ar[dr,->>] \ar[rr] && \ar[dl,->>] C\\ & A.
\end{tikzcd}
\]
We call this a \textbf{fibrant localizing representative} of $F$.

\begin{thm}\label{thm:expideal-rsf}
  Let $\M$ be an excellent model category such that $h\M$ is extensive, and let $\D$ be an accessible exponential ideal in $h\M$ generated by some family of maps $F$.
  Then the fiber over $1$ of the reflective subfibration generated by the higher inductive localization at a fibrant localizing representative of $F$ coincides with $\D$.
\end{thm}
\begin{proof}
  Let $F:B\to C$ be a map of fibrations over a fibrant object $A$ representing the map $\coprod_i f_i$ in $h\M$ as above.
  By definition, a fibrant object $X$ is internally $F$-local if $A_* \mathsf{isequiv}(- \circ F)$ has a global section, or equivalently if $\mathsf{isequiv}(-\circ F)$ has a section over $A$, or equivalently if the map of local exponentials
  \begin{equation}
    (-\circ F) : {(A^*X)}^{C} \to {(A^*X)}^{B}\label{eq:circfp}
  \end{equation}
  is an equivalence over $A$.
  Since these 1-categorical local exponentials in $\M_{/A}$ are between fibrations, they present $(\infty,1)$-categorical exponentials in $h\M_{/A}$.
  Thus, by \cref{thm:ext-exp}, this is equivalent to saying that each map $X^{C_i} \to X^{B_i}$ is an equivalence.

  This certainly implies that $X\in\D$, i.e.\ that it is $F$-local in the external sense, since $h\M(B_i,X) \simeq h\M(1,X^{B_i})$ and similarly.
  Conversely, if $X\in\D$, then to show that $X^{C_i} \to X^{B_i}$ is an equivalence, by the Yoneda lemma it suffices to show that the induced map on hom-spaces $h\M(Y,X^{C_i}) \to h\M(Y,X^{B_i})$ is an equivalence for any $Y\in h\M$.
  But this is equivalent to $h\M(C_i,X^Y) \to h\M(B_i,X^Y)$, which is an equivalence since $X^Y\in \D$, as $\D$ is an exponential ideal.
\end{proof}

In particular, any accessible exponential ideal in an $\infty$-topos $\C$ can be extended to a reflective subfibration, which in turn can be represented by a reflective subuniverse in a type theory that interprets into (a model category presenting) $\C$.
In the other direction, we observe that any internally accessible localization is externally accessible as well:

\begin{thm}
  Suppose $\C$ is a contextual category with an accessible reflective subuniverse $\D$, such that $h\C$ is a locally presentable $(\infty,1)$-category.
  Then the corresponding reflective subcategory $\D_1$ of $h\C$ is accessible.
\end{thm}
\begin{proof}
  By definition, $X\in h\C$ lies in $\D_1$ just when $\prd{a:A}\mathsf{isequiv}(\lam{g:C(a)\to X} g\circ F(a))$ has a global element, which is to say that $\mathsf{isequiv}(\lam{g:C(a)\to X} g\circ F(a))$ has a section over $A$, or equivalently that the induced map ${(\blank)}^F : {(A^* X)}^C \to {(A^*X)}^B$ is an equivalence over $A$.
  Thus, $\D_1$ is the $(\infty,1)$-categorical pullback
  \[
  \begin{tikzcd}
    \D_1 \ar[d] \ar[rr] && h\C_{/A} \ar[d] \\
    h\C \ar[r,"{A^*}"'] & h\C_{/A} \ar[r,"{(\blank)^F}"'] & (h\C_{/A})^{\mathbf{2}}.
  \end{tikzcd}
  \]
  Since all of the functors involved are left or right adjoints between accessible $(\infty,1)$-categories, they are also accessible; thus $\D_1$ is also accessible.
\end{proof}

In general, it seems that a reflective subfibration need not be uniquely determined by its fiber over $1$.
Different choices of generating families $F$ in the above theorem could produce different reflective subfibrations.

Of these, it might happen that some are composing and some are not.
However, we can give a necessary and sufficient condition for there to exist \emph{some} extension of an accessible exponential ideal to a composing reflective subfibration, i.e.\ a stable factorization system.

\begin{lemma}\label{thm:stable-units}
  Let $\C$ be a locally cartesian closed $(\infty,1)$-category and $\D$ a reflective subcategory of it, with reflection units $\modalunit[x]:x \to \modal x$.
  The following are equivalent; when they hold we say that $\D$ has \textbf{stable units}~\cite{chk:reflocfact}.
  \begin{enumerate}
  \item For every $x\in\D$, the reflective subcategory $\D_{/x}$ of $\C_{/x}$ is an exponential ideal.\label{item:su1}
  \item The reflector $\modal$ preserves all pullbacks over an object of $\D$.\label{item:su2}
  \item The reflector $\modal$ inverts any pullback of any $\modalunit[x]$.\label{item:su3}
  \end{enumerate}
\end{lemma}
\begin{proof}
  Since the reflection of $\C_{/x}$ into $\D_{/x}$ (when $x\in\D$) is essentially just $\modal$ itself, condition~\ref{item:su2} says that this reflector always preserves finite products.
  Thus,~\ref{item:su1}$\Leftrightarrow$\ref{item:su2} by \cref{thm:exp-ideal}.
  And~\ref{item:su2}$\Rightarrow$~\ref{item:su3} since $\modalunit[x]$ is inverted by $\modal$ and $\modal x \in \D$.
  Finally, if~\ref{item:su3} then we can factor any pullback over an $x\in \D$ as follows:
  \[
  \begin{tikzcd}
    w \ar[r] \ar[d] \ar[dr,phantom,"\lrcorner" near start] & y \ar[d] \\
    z \ar[r] & x
  \end{tikzcd}
  \qquad=\qquad
  \begin{tikzcd}
    w \ar[r] \ar[dd]  \ar[dr,phantom,"\lrcorner" near start] &
    \bullet \ar[r] \ar[d] \ar[dr,phantom,"\lrcorner" near start] &
    y \ar[d,"{\modalunit[y]}"] \\
    & \bullet \ar[d] \ar[r] \ar[dr,phantom,"\lrcorner" near start] & \modal y \ar[d]\\
    z \ar[r,"{\modalunit[z]}"'] & \modal z \ar[r] & x.
  \end{tikzcd}
  \]
  The lower-right square is a pullback of objects in $\D$, hence its vertex is also in $\D$ and it is preserved by $\modal$.
  The other two squares are pullbacks of some $\modalunit$, hence preserved by $\modal$ by~\ref{item:su3}.
  Thus, $\modal$ preserves the whole pullback, so we have~\ref{item:su2}.
\end{proof}

If a reflective subfibration is composing, then for any $x\in \D_1$ we have $\D_x = {(\D_1)}_{/x}$ as subcategories of $\C_{/x}$.
Since $\D_{x}$ is always an exponential ideal in $\C_{/x}$ (for the same reasons that $\D_1$ is an exponential ideal in $\C$), it follows that so is ${(\D_1)}_{/x}$, hence by \cref{thm:stable-units}\ref{item:su1} $\D_1$ has stable units.
In other words, having stable units is a necessary condition for a reflective subcategory to underlie some stable factorization system.

We now show that in good cases this is also sufficient.
The idea is similar to that of \cref{thm:expideal-rsf}, but using nullification at the fibers (\cref{sec:nullification}) instead of localization, as in \cref{thm:acc-modal}.
We start with an analogue of \cref{thm:ext-exp}.

\begin{lemma}\label{thm:ext-exp-lex}
  Let $X$ be an object and $F = {\{ f_i : B_i \to A_i \}}_{i\in I}$ a family of maps in an extensive and locally cartesian closed $(\infty,1)$-category $\C$.
  Then the following are equivalent:
  \begin{enumerate}
  \item For any $i\in I$, the induced ``constant functions'' map $A_i^*X \to {(A_x^*X)}^{B_i}$ into the local exponential in $\C_{/A_i}$ is an equivalence.\label{item:eel1}
  \item If we let $A=\coprod_i A_i$ and $B= \coprod_i B_i$, with an induced map $F: B\to A$, then the constant functions map $A^*X \to {(A^*X)}^B$ in $\C_{/A}$ is an equivalence.\label{item:eel2}
  \end{enumerate}
\end{lemma}
\begin{proof}
  As in \cref{thm:ext-exp}, $A^*X \to {(A^*X)}^B$ is an equivalence if and only if it becomes so upon pullback along each coproduct injection $A_i \to \coprod_i A_i = A$, and these pullbacks take it to $A_i^*X \to {(A_x^*X)}^{B_i}$.
\end{proof}

If $\C=h\M$ in \cref{thm:ext-exp-lex}, then the map $\coprod_i f_i$ can be further represented by a fibration $F : B \twoheadrightarrow A$ with fibrant codomain in $\M$.
We call this a \textbf{fibrant nullifying representative} of $F$.

\begin{thm}\label{thm:null-mod}
  Let $\M$ be an excellent model category such that $h\M$ is extensive, let $\D$ be an accessible reflective subcategory of $h\M$, and let $F$ a family of maps generating $\D$ with the property that any pullback of any map in $F$ is inverted by the reflector of $\D$.
  (Note that the existence of such a family is an extra condition on $\D$.)
  Then higher inductive nullification at a fibrant nullifying representative of $F$ is a stable factorization system whose fiber over $1$ coincides with $\D$.
\end{thm}
\begin{proof}
  By definition, a fibrant object $X$ is internally $F$-null just when \cref{thm:ext-exp-lex}\ref{item:eel2} holds, hence when~\ref{item:eel1} holds.
  By the Yoneda lemma in $h\M_{/A_i}$, this is equivalent to saying that for any map $Y\to A_i$ in $h\M$, the induced map
  \[ h\M_{/A_i}(Y,A_i^*X) \to h\M_{/A_i}(Y,{(A_i^*X)}^{B_i}) \]
  is an equivalence.
  But this is equivalent to
  \[ h\M_{/A_i}(Y,A_i^*X) \to h\M_{/A_i}(Y\times_{A_i} B_i,A_i^*X) \]
  and thus to
  \begin{equation}
    h\M(Y,X) \to h\M(Y\times_{A_i} B_i,X).\label{eq:model-lex}
  \end{equation}
  If~\eqref{eq:model-lex} is an equivalence for all $i$ and $Y$, then taking $Y=A_i$ we see that $X$ is $f_i$-local for all $i$, hence $X\in\D$.
  Conversely, if $X\in \D$, then since the projection $Y\times_{A_i} B_i \to Y$ is a pullback of $f_i$, it is inverted by the reflector of $\D$ and hence is seen by $X$ as an equivalence; thus~\eqref{eq:model-lex} is an equivalence.
\end{proof}

\begin{corollary}\label{thm:stableunits-mod}
  Let $\M$ be an excellent model category such that $h\M$ is extensive, and let $\D$ be an accessible reflective subcategory of $h\M$ with stable units.
  Then there is a fibration $F:B\twoheadrightarrow A$ with fibrant codomain in $\M$ such that the fiber over $1$ of the stable factorization system generated by higher inductive nullification at $F$ coincides with $\D$.
\end{corollary}
\begin{proof}
  This is a categorical version of \cref{thm:acc-modal}.
  Let $G = {\{ g_i : C_i \to D_i \}}_{i\in I}$ be any set of maps in $h\M$ generating $\D$, and let $F$ be a fibrant nullifying representative of
  \begin{equation}
    \textstyle (\coprod_i\modalunit[C_i]) \sqcup (\coprod_i \modalunit[D_i]) : (\coprod_i C_i) \sqcup (\coprod_i D_i) \to (\coprod_i \modal C_i) \sqcup (\coprod_i \modal D_i).\label{eq:coprod-units}
  \end{equation}
  By extensivity, any pullback of~\eqref{eq:coprod-units} is a coproduct of pullbacks of the units $\modalunit[C_i]$ and $\modalunit[D_i]$.
  Since $\D$ has stable units, any pullback of these units is inverted by its reflector, and the class of maps inverted by any reflector is stable under coproducts.
  Thus, by \cref{thm:null-mod} it suffices to show that $\D$ is generated by the units $\modalunit[C_i]$ and $\modalunit[D_i]$ themselves.
  But these units are certainly inverted by the reflector of $\D$, so that every object of $\D$ is $\modalunit[C_i]$-local and $\modalunit[D_i]$-local; while if an object is $\modalunit[C_i]$-local and $\modalunit[D_i]$-local then by 2-out-of-3 it is also $g_i$-local and hence belongs to $\D$.
\end{proof}

In particular, any accessible reflective subcategory with stable units in an $\infty$-topos $\C$ can be extended to a stable orthogonal factorization system, which in turn can be represented by a modality in a type theory that interprets into (a model category presenting) $\C$.

Finally, in the left exact case, we already know from \cref{thm:ttmc-lex} that any lex reflective subcategory $\D$ of $h\M$ can be extended to a lex modality.
We can show that if the former is topological, then so is the latter.

\begin{thm}\label{thm:lex-lex}
  Let $\M$ be an excellent model category such that $h\M$ is an $\infty$-topos, and let $\D$ be a topological localization of $h\M$ in the sense of~\cite[Definition 6.2.1.4]{lurie2009higher}.
  Suppose also that there exist arbitrarily large inaccessible cardinals.
  Then there is a fibration $F:B\twoheadrightarrow A$ with fibrant codomain in $\M$ such that the higher inductive nullification at $F$ generates a topological lex modality whose fiber over $1$ is $\D$.
\end{thm}
\begin{proof}
  By~\cite[Proposition 6.2.1.5]{lurie2009higher}, there exists a family of \emph{monomorphisms} $F$ in $h\M$ generating $\D$ and such that any pullback of a morphism in $F$ is inverted by the reflector of $\D$.
  Thus, by \cref{thm:null-mod}, higher inductive nullification at a fibrant nullifying representative of $F$ is a modality whose fiber over $1$ coincides with $\D$.
  Moreover, since $F$ consists of monomorphisms, its coproduct is also a monomorphism, and thus any fibrant nullifying representative of it represents a family of mere propositions.

  We would like to conclude by applying \cref{thm:prop-loc-lex} internally to conclude that this topological modality is lex.
  However, the proof of \cref{thm:prop-loc-lex} used univalence unavoidably, whereas in this appendix we are not assuming that our type theory has any universes.
  But we have assumed in this theorem that $h\M$ is an $\infty$-topos and hence has object classifiers, the $\infty$-categorical analogue of univalent universes.

  Our map $F$ must be $\kappa$-compact for some inaccessible $\kappa$; let $U$ be an object classifier for $\kappa$-compact morphisms.
  Then repeating the proofs of \cref{thm:acc-lex,thm:prop-loc-lex} categorically, we can show that $F$-nullification satisfies \cref{thm:lex-modalities}\ref{item:mu2}, and hence is lex.
\end{proof}

The following converse result shows that our definition of ``topological'' is essentially the same as that of~\cite{lurie2009higher}.

\begin{thm}
  Let $\M$ be an excellent model category such that $h\M$ is an $\infty$-topos, and let $F$ be a family of monomorphisms in $h\M$.
  Then the fiber over $1$ of the topological lex modality generated by higher inductive nullification at a fibrant nullifying representative of $F$ is a topological localization of $h\M$.
\end{thm}
\begin{proof}
  Let $\D_1$ be the fiber in question.
  By~\cite[Definition 6.2.1.4]{lurie2009higher}, we must show that the class of all morphism in $h\M$ inverted by the reflector of $\D_1$ is ``generated as a strongly saturated class'' by a class of monomorphisms, and also stable under pullback.
  Stability under pullback follows from the fact that by \cref{thm:cc-lex}, $\D_1$ is a lex reflective subcategory of $h\M$.
  For the first, we claim that it is generated as a strongly saturated class by the class $S$ of all pullbacks of morphisms in $F$, which certainly consists of monomorphisms.
  By~\cite[Proposition 5.5.4.15]{lurie2009higher}, it suffices to show that $\D_1$ consists precisely of the $S$-local objects.
  But this is essentially what we showed in \cref{thm:null-mod}.
\end{proof}

Using~\cite{abfj:lexloc} and \cref{thm:nontop-lex}, we can extend \cref{thm:lex-lex} to arbitrary accessible lex localizations by a similar argument.

An accessible lex reflective subcategory of an $\infty$-topos is called a \emph{subtopos}.
\cref{thm:ttmc-lex} tells us that any subtopos of an $\infty$-topos $\C$ can be represented by a lex modality in a type theory that interprets into (a model category presenting) $\C$, while the generalization of \cref{thm:lex-lex} using~\cite{abfj:lexloc} tells us that this lex modality can be chosen to be accessible as we have defined it internally, and topological if the original subtopos was topological.

\begin{rmk}\label{rmk:universes}
  We end the appendix with a remark about universes.
  In general, the problem of modeling homotopy type theory in $\infty$-toposes with strict univalent universes (i.e.\ univalent universes that are strictly closed under the type-forming operations) is an open problem, although it is known to be possible in a few cases~\cite{kapulkin2012univalence,shulman:invdia,shulman:elreedy,cisinski:elegant,shulman:eiuniv}.
  In \autoref{thm:subtopos-model} we noted that if we have such a model in one $\infty$-topos, and moreover there are enough strict univalent universes closed under the reflector for some sub-$\infty$-topos of it, then we obtain such a model in the sub-$\infty$-topos.
  In particular, in this way we could in principle reduce the problem of modeling homotopy type theory in $\infty$-toposes to that of modeling it in presheaf $\infty$-toposes.

  However, even in the cases where strict univalent universes are known to exist, 
  it is not known how to make them closed under such reflectors.
  In particular, the construction of higher inductive types in~\cite{ls:hits} does not remain inside any universe, because it does not preserve fiberwise smallness of fibrations.
  At present we do not know any extension of the results of this appendix to type theory with universes.
\end{rmk}

\endgroup

\bibliographystyle{alpha}
\bibliography{refs}

\newcommand{\etalchar}[1]{$^{#1}$}
\begin{thebibliography}{dPGM04}

\bibitem[AB04]{AwodeyBauer2004}
S.~Awodey and A.~Bauer.
\newblock Propositions as [types].
\newblock {\em J. Logic Comput.}, 14(4):447--471, 2004.

\bibitem[ABFJ17]{abfj:gen-blakers-massey}
Mathieu Anel, Georg Biedermann, Eric Finster, and Andr\'{e} Joyal.
\newblock A generalized {Blakers--Massey} theorem.
\newblock arXiv:1703.09050, 2017.

\bibitem[ABFJ19]{abfj:lexloc}
Mathieu Anel, George Biederman, Eric Finster, and Andr{\'e} Joyal.
\newblock New methods for left exact localisations of topoi.
\newblock In preparation; slides at
  \url{http://mathieu.anel.free.fr/mat/doc/Anel-LexLocalizations.pdf}, 2019.

\bibitem[ADK17]{Partiality}
Thorsten Altenkirch, Nils~Anders Danielsson, and Nicolai Kraus.
\newblock Partiality, revisited: The partiality monad as a quotient
  inductive-inductive type.
\newblock {\em FOSSACS 2017, LNCS 10203}, pages 534--549, 2017.

\bibitem[Bau11]{bauer:loc-hasse}
Tilman Bauer.
\newblock Bousfield localization and the {Hasse} square.
\newblock
  \\\url{http://math.mit.edu/conferences/talbot/2007/tmfproc/Chapter09/bauer.pdf},
  2011.

\bibitem[BGL{\etalchar{+}}17]{HoTT-CPP}
Andrej Bauer, Jason Gross, Peter~LeFanu Lumsdaine, Michael Shulman, Matthieu
  Sozeau, and Bas Spitters.
\newblock The {HoTT} library: A formalization of homotopy type theory in {C}oq.
\newblock In {\em Proceedings of the 6th ACM SIGPLAN Conference on Certified
  Programs and Proofs}, CPP 2017, pages 164--172. ACM, 2017.

\bibitem[Bou18]{bourke:alginj}
John Bourke.
\newblock Iterated algebraic injectivity and the faithfulness conjecture.
\newblock arXiv:1811.09532, 2018.

\bibitem[CHK85]{chk:reflocfact}
C.~Cassidy, M.~H{\'e}bert, and G.~M. Kelly.
\newblock Reflective subcategories, localizations and factorization systems.
\newblock {\em J. Austral. Math. Soc. Ser. A}, 38(3):287--329, 1985.

\bibitem[Cis14]{cisinski:elegant}
Denis-Charles Cisinski.
\newblock Univalent universes for elegant models of homotopy types.
\newblock arXiv:1406.0058, 2014.

\bibitem[CJKP97]{cjkp:locstab-fact}
A.~Carboni, G.~Janelidze, G.~M. Kelly, and R.~Par{\'e}.
\newblock On localization and stabilization for factorization systems.
\newblock {\em Appl. Categ. Structures}, 5(1):1--58, 1997.

\bibitem[CLW93]{clw:ext-dist}
Aurelio Carboni, Stephen Lack, and R.F.C. Walters.
\newblock Introduction to extensive and distributive categories.
\newblock {\em J. Pure Appl. Algebra}, 84(2):145--158, 1993.

\bibitem[CMR17]{Stacks}
Thierry Coquand, Bassel Mannaa, and Fabian Ruch.
\newblock Stack semantics of type theory.
\newblock {\em CoRR}, abs/1701.02571, 2017.

\bibitem[Coq17]{Coquand:stack}
Thierry Coquand.
\newblock Cubical stacks, 2017.
\newblock \url{http://www.cse.chalmers.se/~coquand/stack.pdf}.

\bibitem[CSS05]{css:large-cardinal}
Carles Casacuberta, Dirk Scevenels, and Jeffrey~H. Smith.
\newblock Implications of large-cardinal principles in homotopical
  localization.
\newblock {\em Adv. Math.}, 197(1):120--139, 2005.

\bibitem[dPGM04]{dpgm:modal-tt}
Valeria de~Paiva, Rajeev Gor\'{e}, and Michael Mendler.
\newblock Modalities in constructive logics and type theories.
\newblock {\em Journal of Logic and Computation}, 14(4):439--446, 2004.

\bibitem[dPR16]{pr:fib-modal-tt}
Valeria de~Paiva and Eike Ritter.
\newblock Fibrational modal type theory.
\newblock {\em Electronic Notes in Theoretical Computer Science}, 323:143 --
  161, 2016.
\newblock Proceedings of the Tenth Workshop on Logical and Semantic Frameworks,
  with Applications (LSFA 2015).

\bibitem[EK17]{ek:partial}
Mart{\'i}n~H. Escard{\'o} and Cory~M. Knapp.
\newblock {Partial Elements and Recursion via Dominances in Univalent Type
  Theory}.
\newblock In Valentin Goranko and Mads Dam, editors, {\em 26th EACSL Annual
  Conference on Computer Science Logic (CSL 2017)}, volume~82 of {\em Leibniz
  International Proceedings in Informatics (LIPIcs)}, pages 21:1--21:16,
  Dagstuhl, Germany, 2017. Schloss Dagstuhl--Leibniz-Zentrum fuer Informatik.

\bibitem[FFLL16]{ffll:blakers-massey}
Kuen-Bang~Hou (Favonia), Eric Finster, Daniel Licata, and Peter~LeFanu
  Lumsdaine.
\newblock A mechanization of the {Blakers--Massey} connectivity theorem in
  homotopy type theory.
\newblock {\em LICS}, 2016.
\newblock arXiv:1605.03227.

\bibitem[GK17]{GepnerKock}
David Gepner and Joachim Kock.
\newblock Univalence in locally cartesian closed $\infty$-categories.
\newblock {\em Forum Math.}, 29(3):617--652, 2017.
\newblock arXiv:1208.1749.

\bibitem[Gol10]{goldblatt2010cover}
Robert Goldblatt.
\newblock Cover semantics for quantified lax logic.
\newblock {\em Journal of Logic and Computation}, 2010.

\bibitem[Joh02]{johnstone:elephant}
P.~Johnstone.
\newblock {\em Sketches of an elephant: A topos theory compendium}.
\newblock Oxford, 2002.

\bibitem[Kap17]{kapulkin:lccqcat-tt}
Krzysztof Kapulkin.
\newblock Locally cartesian closed quasi-categories from type theory.
\newblock {\em Journal of Topology}, 10(4):1029--1049, 2017.
\newblock arXiv:1507.02648.

\bibitem[KL19]{kapulkin2012univalence}
Chris Kapulkin and Peter~LeFanu Lumsdaine.
\newblock The simplicial model of univalent foundations (after {V}oevodsky).
\newblock {\em Journal of the European Mathematical Society}, 2019.
\newblock To appear. arXiv:1211.2851.

\bibitem[KS15]{ks:u-not-ntype}
Nicolai Kraus and Christian Sattler.
\newblock Higher homotopies in a hierarchy of univalent universes.
\newblock {\em ACM Trans. Comput. Logic}, 16(2):18:1--18:12, April 2015.

\bibitem[LF14]{lf:emspaces}
Dan Licata and Eric Finster.
\newblock {Eilenberg--MacLane} spaces in homotopy type theory.
\newblock {\em LICS}, 2014.
\newblock \url{http://dlicata.web.wesleyan.edu/pubs/lf14em/lf14em.pdf}.

\bibitem[LS19]{ls:hits}
Peter~LeFanu Lumsdaine and Michael Shulman.
\newblock Semantics of higher inductive types.
\newblock {\em Mathematical Proceedings of the Cambridge Philosophical
  Society}, 2019.
\newblock arXiv:1705.07088.

\bibitem[Lur09]{lurie2009higher}
J.~Lurie.
\newblock {\em Higher topos theory}, volume 170.
\newblock Princeton University Press, 2009.

\bibitem[LW15]{lw:localuniv}
Peter~LeFanu Lumsdaine and Michael~A. Warren.
\newblock The local universes model: An overlooked coherence construction for
  dependent type theories.
\newblock {\em ACM Trans. Comput. Logic}, 16(3):23:1--23:31, July 2015.
\newblock arXiv:1411.1736.

\bibitem[Mog91]{moggi:monads}
Eugenio Moggi.
\newblock Notions of computation and monads.
\newblock {\em Information and Computation}, 93(1):55 -- 92, 1991.
\newblock Selections from 1989 IEEE Symposium on Logic in Computer Science.

\bibitem[MP12]{mp:more-concise}
J.~P. May and K.~Ponto.
\newblock {\em More concise algebraic topology}.
\newblock Chicago lectures in mathematics. University of Chicago Press, 2012.

\bibitem[NPP05]{npp:ctx-modal-tt}
Aleksandar Nanevski, Frank Pfenning, and Brigitte Pientka.
\newblock Contextual model type theory, 2005.

\bibitem[Rij17]{joinconstruction}
Egbert Rijke.
\newblock The join construction.
\newblock arXiv:1701.07538, 2017.

\bibitem[RS15]{RijkeSpitters:Sets}
Egbert Rijke and Bas Spitters.
\newblock Sets in homotopy type theory.
\newblock {\em Mathematical Structures in Computer Science}, 25:1172--1202, 6
  2015.
\newblock arXiv:1305:3835.

\bibitem[Shu14]{shulman:up-wo-fe}
Michael Shulman.
\newblock Universal properties without function extensionality.
\newblock
  \\\url{https://homotopytypetheory.org/2014/11/02/universal-properties}, 2014.

\bibitem[Shu15a]{shulman:elreedy}
Michael Shulman.
\newblock The univalence axiom for elegant {Reedy} presheaves.
\newblock {\em Homology, Homotopy, and Applications}, 17(2):81--106, 2015.
\newblock arXiv:1307.6248.

\bibitem[Shu15b]{shulman:invdia}
Michael Shulman.
\newblock Univalence for inverse diagrams and homotopy canonicity.
\newblock {\em Mathematical Structures in Computer Science}, 25:1203--1277, 6
  2015.
\newblock arXiv:1203.3253.

\bibitem[Shu17]{shulman:eiuniv}
Michael Shulman.
\newblock Univalence for inverse {EI} diagrams.
\newblock {\em Homology, Homotopy and Applications}, 19(2):219--249, 2017.
\newblock arXiv: 1507.03634.

\bibitem[Shu18]{shulman:bfp-realcohesion}
Michael Shulman.
\newblock Brouwer's fixed-point theorem in real-cohesive homotopy type theory.
\newblock {\em Mathematical Structures in Computer Science}, 28(6):856--941,
  2018.
\newblock arXiv:1509.07584.

\bibitem[SS12]{SchreiberShulman}
U.~Schreiber and M.~Shulman.
\newblock Quantum gauge field theory in cohesive homotopy type theory (extended
  abstract).
\newblock In {\em Quantum Physics and Logic 2012}, 2012.

\bibitem[{Uni}13]{TheBook}
The {Univalent Foundations Program}.
\newblock {\em Homotopy Type Theory: Univalent Foundations of Mathematics}.
\newblock \url{http://homotopytypetheory.org/book/}, first edition, 2013.

\bibitem[Wil94]{wilson:frames}
J.~T. Wilson.
\newblock {\em The assembly tower and some categorical and algebraic aspects of
  frame theory}.
\newblock PhD thesis, Carnegie Mellon University, 1994.

\end{thebibliography}

\end{document}